\documentclass[reqno,a4paper]{amsart}

\usepackage[italian,english]{babel}
\usepackage{eucal,amsfonts,amssymb,amsmath,amsthm,epsfig,mathrsfs}
\usepackage{color}
\usepackage{amscd,amsxtra}
\usepackage{enumerate}
\usepackage{latexsym}
\usepackage{bm}

\allowdisplaybreaks

\newcounter{ipotesi}

\Alph{ipotesi}
 \makeatletter \@addtoreset{equation}{section}

\makeatother \makeatletter

\newtheorem{thm}{Theorem}[section]
\newtheorem{hyp}[thm]{Hypotheses}{\rm}
{\rm}

\newtheorem{coro}[thm]{Corollary}
\newtheorem{prop}[thm]{Proposition}
\newtheorem{defi}[thm]{Definition}
\newtheorem{rmk}[thm]{Remark}{\rm}
\newtheorem{example}[thm]{Example}
\newcommand{\Pp}{\mathbb P}
\newcommand{\F}{\mathcal F}

\newcounter{parentenv}

\newcommand{\R}{{\mathbb R}}
\newcommand{\N}{{\mathbb N}}

\newcommand{\Rd}{\mathbb R^d}
\newcommand{\Rm}{\mathbb R^m}

\renewcommand{\H}{{\bf H}}

\newcommand{\bd}{\begin{defi}}
\newcommand{\Z}{{\bf Z}}
\newcommand{\ed}{\end{defi}}
\newcommand{\nnm}{\nonumber}
\newcommand{\be}{\begin{equation}}
\newcommand{\ee}{\end{equation}}
\newcommand{\barr}{\begin{array}}
\newcommand{\earr}{\end{array}}
\newcommand{\bmn}{\begin{eqnarray}}
\newcommand{\emn}{\end{eqnarray}}
\newcommand{\bnm}{\begin{eqnarray*}}
\newcommand{\enm}{\end{eqnarray*}}
\newcommand{\bln}{\begin{subequations}}
\newcommand{\eln}{\end{subequations}}

\newcommand{\ba}{\begin{align}}
\newcommand{\ea}{\end{align}}
\newcommand{\banm}{\begin{align*}}
\newcommand{\eanm}{\end{align*}}
\newcommand{\one}{\mbox{$1\!\!\!\;\mathrm{l}$}}

\newcommand{\A}{\mathcal A}
\newcommand{\vv}{{\bf v}}
\newcommand{\ww}{{\bf w}}
\newcommand{\f}{{\bf f}}

\newcommand{\g}{{\bf g}}
\newcommand{\h}{{\bf h}}
\renewcommand{\k}{{\bf k}}
\newcommand{\G}{{\bf G}}

\newcommand{\E}{\mathbb E}
\newcommand{\Y}{\bf Y}
\newcommand{\uu}{{\bf u}}

\allowdisplaybreaks

\begin{document}

\title[On coupled systems of Kolmogorov equations]{On coupled systems of Kolmogorov equations with applications to stochastic differential games}
\thanks{The authors are members of G.N.A.M.P.A. of the Italian Istituto Nazionale di Alta Matematica (INdAM)}
\thanks{Work partially supported by the M.I.U.R. Research Project PRIN 2010-11 ``Problemi differenziali di evoluzione: approcci deterministici e stocastici e loro interazioni''
and by the INdAM-GNAMPA Project 2014 ``Equazioni ellittiche e paraboliche a coefficienti illimitati''.}
\author[D. Addona, L. Angiuli, L. Lorenzi, G. Tessitore]{Davide Addona, Luciana Angiuli, Luca Lorenzi and Gianmario Tessitore}
\address{D.A \& G.T.: Dipartimento di Matematica e Applicazioni, Universit\`a degli Studi di Milano Bicocca, Via Cozzi 55, I-20125 MILANO (Italy)}
\address{L.A.: Dipartimento di Matematica e Fisica ``Ennio De Giorgi'', Universit\`a del Salento, Via per Arnesano, I-73100 LECCE (Italy)}
\address{L.L.: Dipartimento di Matematica e Informatica, Universit\`a degli Studi di Parma, Parco Area delle Scienze 53/A, I-43124 PARMA (Italy)}
\email{d.addona@campus.unimib.it}
\email{luciana.angiuli@unisalento.it}
\email{luca.lorenzi@unipr.it}
\email{gianmario.tessitore@unimib.it}

\date{}


\begin{abstract}
We prove that a family of linear bounded evolution operators $({\bf G}(t,s))_{t\ge s\in I}$ can be associated, in the space of vector-valued bounded and continuous functions, to a class of systems of elliptic operators $\bm{\mathcal A}$ with unbounded coefficients defined in $I\times \Rd$ (where $I$ is a right-halfline or $I=\R$) all having the same principal part. We establish some
continuity and representation properties of $({\bf G}(t,s))_{t \ge s\in I}$ and a sufficient condition for the evolution operator to be compact in $C_b(\Rd;\R^m)$.
We prove also a uniform weighted gradient estimate and some of its more relevant consequence.
\end{abstract}

\maketitle

\section{Introduction}

In recent years, the study of second-order elliptic operators with unbounded coefficients has been object of increasing interest since they appear in many models of probability and mathematical finance. Whereas the theory of Kolmogorov equations is well developed in the scalar case, as the considerable literature shows (see e.g., \cite{BerLorbook} and the references therein), the case of systems of equations is nowadays at a preliminary level.

In this paper, we consider a family of nonautonomous second-order uniformly elliptic operators $\bm{\mathcal A}$ (having all the same principal part), defined on smooth functions ${\bf w}:\Rd\to \R^m$ by
\begin{equation}
(\bm{\mathcal A}(t){\bf w})(x)=\sum_{i,j=1}^dQ_{ij}(t,x)D^2_{ij}{\bf w}(x)+\sum_{j=1}^d B_j(t,x)D_j{\bf w}(x)+C(t,x){\bf w}(x),
\label{operat-A}
\end{equation}
for any $t \in I$ ($I$ being a right-halfline or $I = \R$) and $x\in \Rd$. In \eqref{operat-A}, the entries $Q_{ij}$ of the matrix-valued function $Q$ are smooth functions, possibly unbounded, whose associated quadratic form is positive definite, namely $\inf_{I\times \Rd}\lambda_Q (t,x)$ is positive, where
$\lambda_Q(t,x)$ is the minimum eigenvalue of the matrix $Q(t,x)$.
As far as $B_j$ and $C$ are concerned, they are $m\times m$ matrices whose elements are smooth enough and possibly unbounded real valued functions (see Hypotheses \ref{base}).

We deal with the parabolic Cauchy problem associated to $\bm{\mathcal A}$ in $C_b(\Rd;\R^m)$, i.e., we look for a locally in time bounded classical solution $\uu=(u_1,u_2,\ldots, u_m)$ of the system
\begin{equation}\label{eq:cauchy_problem_system}
\left\{\begin{array}{ll}
D_t\uu(t,x)=(\bm{\mathcal A}\uu)(t,x),\qquad\;\, &(t,x)\in (s,+\infty)\times \Rd,\\[1mm]
\uu(s,x)={\bf f}(x),& x \in \Rd,
\end{array}
\right.
\end{equation}
where $s \in I$ and ${\bf f}\in C_b(\Rd;\R^m)$.
By a locally in time bounded classical solution of \eqref{eq:cauchy_problem_system} we mean a function
$\uu\in C^{1,2}((s,+\infty)\times\Rd;\R^m)\cap C_b([s,T]\times \Rd;\R^m)$ satisfying \eqref{eq:cauchy_problem_system}
and bounded in each strip $[s,T]\times\Rd$.

As already noticed, the theory of linear systems of Kolmogorov equations is at a preliminary level.
In \cite{DelLor11OnA,HieLorPruRhaSch09Glo} the simpler case of weakly coupled equations (i.e., $B_j(t,x)=b_j(x)I_m$ for any $(t,x)\in I\times \Rd$, $j=1,\ldots,d$
and some real valued function $b_j$) is considered. More precisely, in \cite{DelLor11OnA}, the analysis is carried over in the space of vector-valued bounded
and continuous functions, whereas in \cite{HieLorPruRhaSch09Glo} also the $L^p$-setting is studied, assuming that the diffusion coefficients are bounded.
Very recently, taking advantage of some results contained in this paper, the $L^p$-theory has been  extended also to first-order coupled systems as in \eqref{operat-A} (see \cite{{AngLorPal15Lpe}}).

This paper is devoted to keep on the analysis started in \cite{DelLor11OnA} aiming at considering more general system than those studied in \cite{DelLor11OnA}.
In our case, the presence of a first order term as in \eqref{operat-A}, where the first partial derivatives of all the components of $\uu$ are mixed together, makes the problem quite involved and already the unique solvability of the problem \eqref{eq:cauchy_problem_system} is a not trivial question.
Indeed, the method used in \cite{DelLor11OnA} takes strongly into account the special structure of the equations and can not be immediately adapted to our situation.
To overcome this difficulty, we provide two sets of assumptions (see Hypoyheses \ref{hyp:existence_uniqueness} and \ref{hyp_comp}) which yield uniqueness of a locally in time bounded solution to problem \eqref{eq:cauchy_problem_system}. Under Hypotheses \ref{hyp:existence_uniqueness}, we extend to our situation the method used by Kresin and Maz'ia in \cite{KreMaz12Max} in the case of bounded coefficients. Such assumptions reduce to those assumed in \cite{DelLor11OnA} in the case of weakly coupled equations and represent the natural generalization to the vector case of those typically assumed in the case of a single equation.
On the other hand, Hypotheses \ref{hyp_comp} allows to get uniqueness of the solution $\uu$ as above by a comparison argument: we show that $\uu$ can be estimated pointwise from above by $G(t,s)|\f|^2$, where $G(t,s)$ is the
evolution operator associated to the elliptic operator $\A(t)={\rm Tr}(QD^2)+\langle {\bf b},\nabla\rangle$ in $C_b(\Rd)$ and ${\bf b}=(b_1,\ldots,b_d)$ represents the diagonal part of the drift matrices $B_j$. Indeed,
we limit ourselves to considering the case when $B_j(t,x)=b_j(t,x)I_m+\tilde B_j(t,x)$ for any $(t,x)\in I\times\Rd$, any $j=1,\ldots,d$ and some $b_j\in C^{\alpha/2,\alpha}_{\rm loc}(I\times\Rd)$ and $\tilde B_j\in C^{\alpha/2,\alpha}_{\rm loc}(I\times\Rd;\R^{m^2})$.
Once uniqueness is guaranteed, the existence of a classical solution of the problem \eqref{eq:cauchy_problem_system} is then proved by some compactness and localization argument based on interior Schauder estimates recalled in the Appendix. Hence, under the previous assumptions, we can associate an evolution operator ${\bf G}(t,s)$ to $\bm{\mathcal A}$ in $C_b(\Rd;\R^m)$, i.e., for any ${\bf f}\in C_b(\Rd;\R^m)$, ${\bf G}(\cdot,s){\bf f}$ represents the unique locally in time bounded classical solution to \eqref{eq:cauchy_problem_system}, which satisfies $\|{\bf G}(t,s){\bf f}\|_{\infty}\le c_{s,T}\|{\bf f}\|_{\infty}$ for any $t>s \in I$, and some positive constant $c_{s,T}$.

The next (natural) step in our investigation consists in proving some continuity properties of the evolution operator, which hold in the scalar case.
In particular, an integral representation formula, in terms of some finite Borel measures, is available for ${\bf G}(t,s)$.
More precisely, for any $i=1,\ldots,m$, $t>s\in I$ and $x \in \Rd$, there exists a family $\{p_{ij}(t,s,x,dy):j=1,\ldots,m\}$ of finite Borel measures such that
\begin{equation}
({\bf G}(t,s)\f)_i(x)=\sum_{j=1}^m\int_{\Rd}f_j(y)p_{ij}(t,s,x,dy),\qquad\;\,\f\in C_b(\Rd;\R^m).
\label{eq:inte_intro}
\end{equation}
Formula \eqref{eq:inte_intro} allows to extend the evolution operator ${\bf G}(t,s)$ to the space of bounded Borel vector-valued functions and to prove the strong Feller property for ${\bf G}(t,s)$.
With some considerable efforts, we prove that the measures $p_{ij}(t,s,x,dy)$ are absolutely continuous with respect the Lebesgue measure.
One of the main difficulty is represented by the fact that, differently from the scalar case, $p_{ij}(t,s,x,dy)$ are signed measures.

In Section \ref{comp_sect}, we use the pointwise estimate of $|{\bf G}(t,s){\bf f}|^2$ in terms of $G(t,s)|{\bf f}|^2$ to prove that the compactness
of $G(t,s)$ in $C_b(\Rd)$ is a sufficient condition for the compactness of ${\bf G}(t,s)$ in $C_b(\Rd;\R^m)$. To prove the quoted estimate, besides
the standard regularity assumptions on the coefficients of $\bm{\mathcal A}$, the just mentioned assumption on the drift matrices $B_j$ ($j=1,\ldots,d$)
and the existence of a Lyapunov function for the operator $\mathcal{A}(t)$, we impose that all the entries of the matrices $\tilde B_j$ ($j=1,\ldots,d$)
can grow no faster than $\lambda_Q^{\sigma}$, for some $\sigma\in (0,1)$, and that the quadratic form associated to the potential $C$ is bounded from above.

Section \ref{sect-5} is devoted to prove a weighted gradient estimate for the function ${\bf G}(t,s)\f$. More precisely,
under suitable assumptions (see Hypotheses \ref{Hyp-stime-gradiente}), which are essentially growth and algebraic conditions on the coefficients of the
operator $\bm{\mathcal A}$, their derivatives and the positive definite $(t,x)$-dependent matrix $M$, we show that the map $M(J_x\G(\cdot,s)\f)^T$ is continuous and bounded in $(s,T]\times \Rd$ and that for any $T>s \in I$ there exists a positive constant $C_{s,T}$ such that
\begin{equation}\label{grest_intro}
\|M(J_x\G(t,s)\f)^T\|_{\infty}\leq C_{s,T}(t-s)^{-1/2}\|\f\|_\infty,\qquad\;\, t \in (s,T),
\end{equation}
for any $\f \in C_b(\Rd,\R^m)$.
Unweighted uniform gradient estimates are classical when the coefficients of $\bm{\mathcal A}$ are bounded and have been recently extended in the scalar case (see e.g., \cite{AngLor13OnT,AngLor14Non,BerLor07,KunLorLun09Non,lunardi})
to the case of unbounded coefficients. On the other hand, weighted gradient estimates seem to be new also in the scalar case. We stress that
we can allow $M(t,\cdot)$ to grow at most linearly at infinity. This condition could seem too restrictive, but as the classical case of bounded coefficients shows in general,
the gradient of the solution to problem \eqref{eq:cauchy_problem_system} does not vanishes with polynomial rate as $|x|\to +\infty$. Moreover, in view of the applications
to stochastic differential games, this bound on the growth of $M$ is not restrictive at all.
Estimate \eqref{grest_intro} is obtained by adapting the Bernstein
method (see \cite{Ber06Sur}), which has been already used in the case of a single elliptic equation with unbounded coefficients  (see \cite{KunLorLun09Non,lorenzi-1})
and in case of domain with sufficiently smooth boundary under Dirichlet and first-order, non tangential homogeneous boundary conditions (see \cite{AngLor13OnT,AngLor14Non}).

As a first application of \eqref{grest_intro}, we provide
a sufficient condition for the compactness of the vector evolution operator ${\bf G}(t,s)$ in $C_b(\Rd;\R^m)$ to imply the compactness of $G(t,s)$ in $C_b(\Rd)$.
The main step in this direction is the proof of the following representation formula of a component of ${\bf G}(t,s)\f$ in terms of $G(t,s)$:
\begin{align}
(\G(t,s)\f)_{\bar k}(x)=(G(t,s)f_{\bar k})(x)+\int_s^t (G(t,r)({\mathcal S}\G(\cdot,s)\f(r, \cdot))(x)dr,\qquad\;\,t>s\in I,\;\,x\in\Rd,
\label{repr_formula_intro}
\end{align}
where ${\mathcal S}\G(\cdot,s)\f=\sum_{i=1}^d \langle row_{\bar k}\tilde{B}_i,D_i\G(\cdot,s)\f\rangle +\langle row_{\bar k} C,\G(\cdot,s)\f\rangle$
and $row_{\bar k}\tilde{B}_i$, $row_{\bar k} C$ denote the $\bar k$- row of the matrices $\tilde{B}_i$ and $C$, respectively.
To make formula \eqref{repr_formula_intro} meaningful, we need to guarantee that the integral term is well defined. We prove this fact assuming that $row_{\bar k} C$ is bounded for some $\bar k \in \{1,\ldots,m\}$ and
using \eqref{grest_intro}.

The second (and more relevant for the applications) consequence of the gradient estimate is
an existence result for the
system of forward backward stochastic differential equations
\begin{equation}
\left\{
\begin{array}{ll}
-d{\bf Y}_\tau=\H(\tau,X_\tau,{\bf Z}_\tau)d\tau-{\bf Z}_\tau dW_\tau, & \tau\in[t,T], \\[1mm]
dX_\tau={\bf b}(\tau,X_\tau) d\tau+G(\tau,X_\tau) dW_\tau, & \tau\in[t,T], \\[1mm]
{\bf Y}_T=\g(X_T), \\[1mm]
X_t=x, & x\in\R^d,
\end{array}
\right.
\label{eq:FBSDE}
\end{equation}
and identification formulae for the pair ${\bf Y}$ and ${\bf Z}$ in terms of the mild solution to a
semilinear problem associated with an autonomous elliptic operator of the type \eqref{operat-A}.
The main novelty lies in the fact that we do not assume that ${\bf H}$ is (globally) Lipschitz continuous
as assumed in \cite{pardoux-peng-1992, fuhrman-tessitore}: we just assume $\beta$-H\"older continuity (for some $\beta\in (0,1)$) with respect to the second set of variables, which is
not uniform with respect to the other variables.

The above identification formulae are, finally, used to prove the existence of a Nash equilibrium for a nonzero-sum stochastic differential game.
We follow the approach of \cite{hamadene-1}, where the coefficients of the controlled system are assumed to be bounded,
and, as in our case, the diffusion is assumed to be independent of the control.
The results in \cite{hamadene-1} have been then extended to the infinite
dimensional setting in \cite{fuhrman-hu} still assuming the coefficients of the controlled system to be bounded.
Very recently, in \cite{hamadene-2} the authors have proved the existence of a Nash equilibrium,
relaxing the boundedness of the drift of the
controlled system but still assuming the diffusion to be bounded. We stress that, in our situation we can allow the diffusion
of the controlled system to be unbounded.

\subsection*{Notation}
Functions with values in $\R^m$ are displayed in bold style. Given a function $\f$ (resp. a sequence
$(\f_n)$) as above, we denote by $f_i$ (resp. $f_{n,i}$) its $i$-th component (resp. the $i$-th component
of the function $\f_n$).
By $B_b(\Rd;\Rm)$ we denote the set of all the bounded Borel measurable functions $\f:\Rd\to\Rm$.
For any $k\ge 0$, $C^k_b(\R^d;\R^m)$ is the space of all the functions
whose components belong to $C^k_b(\R^d)$, where the notation $C^k(\R^d)$ ($k\ge 0$) is standard and we use the subscript ``$c$'' and ``$b$''
for spaces of functions with compact support and bounded, respectively.
Similarly, when $k\in (0,1)$, we use the subscript ``loc'' to denote the space of all $f\in C(\Rd)$
which are H\"older continuous in any compact set of $\Rd$.
We assume that the reader is familiar also with the parabolic spaces $C^{\alpha/2,\alpha}(I\times \Rd)$
($\alpha\in (0,1)$) and $C^{1,2}(I\times \Rd)$, and we use the subscript ``loc'' with the same meaning as above.

The Euclidean inner product of the vectors $x,y\in\R^d$ is denoted by $\langle x,y\rangle$.
Square matrices of size $m$ are thought as elements of $\R^{m^2}$.
For any $M \in \R^{m^2}$, we denote by $M_{ij}$, and $row_j M$, the $ij$-th
element and the
$j$-th row vector of the matrix $M$. By ${\rm Tr}(M)$ and $M^T$ we denote the trace of $M$ and the transposed matrix of $M$.
Finally, $\lambda_M$ and $\Lambda_M$ indicate the minimum and the maximum eigenvalue of
the (symmetric) matrix $M$. For any $k \in \N$, by $I_k$ we denote the identity matrix of size $k$. When the matrix $M$
depends on $x$ (resp. $(t,x)$) we write $\lambda_M(x)$ and $\Lambda_M(x)$ (resp. $\lambda_M(t,x)$ and $\Lambda_M(t,x)$)
instead of $\lambda_{M(x)}$ and $\Lambda_{M(x)}$ (resp. $\lambda_{M(t,x)}$ and $\Lambda_{M(t,x)}$).

By $\chi_A$ and ${\bf e}_j$ we denote the characteristic function of the set
$A\subset\R^d$ and the $j$-th vector of the Euclidean basis of $\R^m$. Further, we set $\one=\chi_{\Rd}$.
The open ball with center $x_0$ and radius $R>0$ and its closure are
denoted by $B_{R}(x_0)$ and $\overline B_R(x_0)$; when $x_0=0$ we simply
 write $B_R$ and $\overline B_R$. For any interval $J\subset \R$ we denote by $\Sigma_J$ the set $\{(t,s)\in J\times J:\;\, t>s\}$.

We set ${\mathcal A}_0={\rm Tr}(QD^2)$, $\tilde{\mathcal A}={\rm Tr}(QD^2)+\langle {\bf b},\nabla_x\rangle$ and denote by $\tilde{\bm{\mathcal A}}$
the diagonal vector-valued operator whose $m$-components coincide with $\tilde{\mathcal A}$.
Next, for any $R>0$ we denote by $\G_R^{\mathcal D}(t,s)$ (resp.
$\G_R^{\mathcal N}(t,s)$) and $G_R^{\mathcal D}(t,s)$ (resp. $G_R^{\mathcal N}(t,s)$) the evolution
operator associated with the realization of the operators $\bm{\mathcal A}$ and
$\tilde {\mathcal A}$ in $C_b(B_R;\R^m)$ and $C_b(B_R)$, respectively, with homogeneous Dirichlet (resp. Neumann) boundary conditions on $\partial B_R$.
Finally, $G(t,s)$ denotes the evolution operator associated to the operator
$\tilde{\mathcal A}$ in $C_b(\Rd)$, whose existence has been proved in
\cite{KunLorLun09Non}.

\section{Existence and uniqueness of locally in time bounded classical solutions to \eqref{eq:cauchy_problem_system}}

Let $I$ be an open right-halfline or $I=\R$ and let $\bm{\mathcal A}$ be the system of elliptic operators defined in \eqref{operat-A}.
In this section we prove that, for any $s \in I$ and any
$\f \in C_b(\Rd;\Rm)$ there exists a unique locally in time bounded classical solution to the Cauchy problem \eqref{eq:cauchy_problem_system}.

The following are standing assumptions that we will not mention anymore.
\begin{hyp}\label{base}
The coefficients $Q_{ij}=Q_{ji}$, $(B_i)_{hk}$ and $C_{hk}$ belong to $C^{\alpha/2,\alpha}_{\rm loc}(I\times\R^d)$, for any $i,j=1,\ldots,d$, $h,k=1, \ldots,m$. Moreover, $\lambda_0:=\inf_{I\times \Rd}\lambda_Q>0$.
\end{hyp}

In what follows we will consider, alternatively, two sets of assumptions.
\begin{hyp}
\label{hyp:existence_uniqueness}
\begin{enumerate}[\rm (i)]
\item
There exist $\varepsilon>0$ and a function $\kappa:I\times\Rd\to\R$, bounded from above by a constant $\kappa_0$ such that
the function ${\mathcal K}_{\bm\eta,\varepsilon}:=\sum_{i,j=1}^d(Q^{-1})_{ij}[\langle B_i\bm\eta,\bm\eta\rangle\langle B_j\bm\eta,\bm\eta\rangle-\langle B_i^*\bm\eta,B_j^*\bm\eta\rangle]-4\langle C\bm\eta,\bm\eta\rangle+4\varepsilon\kappa$ is nonnegative in $I\times\R^d$ for any $\bm\eta\in \partial B_1\subset\R^m$;
\item
for any bounded interval $J\subset I$ there exist $\mu_J\in \R$ and a positive $($Lyapunov$)$ function
$\varphi_J\in C^2(\R^d)$ blowing up as $|x|\to +\infty$ such that
$\sup_{\eta \in \partial B_1}\sup_{J\times\R^d}(\A_{\eta}\varphi_J-\mu_J\varphi_J)<+\infty$,
where $\A_{\bm\eta}={\rm Tr}(QD^2)+\langle {\bf b}_{\eta},\nabla_x\rangle+2\varepsilon\kappa$ and $b_{\eta,j}=\langle B_j\eta,\eta\rangle$ for $j=1,\ldots,d$.
\end{enumerate}
\end{hyp}

\begin{hyp}\label{hyp_comp}
\begin{enumerate}[\rm (i)]
\item
There exist functions $b_i\in C^{\alpha/2,\alpha}_{\rm{loc}}(I\times
\R^d)$ and $\tilde B_i \in C^{\alpha/2, \alpha}_{\rm{loc}}(I \times \R^d; \R^{m^2})$ such that
$B_i:=b_iI_m+\tilde B_i$ in $I\times\R^d$ for any $i=1, \ldots,d$ and
$|(\tilde{B}_{i})_{jk}|\leq \xi\lambda_Q^{\sigma}$
in $I \times \Rd$, for any $j,k=1,\ldots,m$, $i=1,\ldots,d$ and some locally bounded function
$\xi:I\to (0,+\infty)$ and $\sigma \in (0,1)$;
\item
$H_J:=\sup_{J\times\Rd}(\Lambda_C+4^{-1}m^2d\xi^2\lambda_Q^{2\sigma-1})<+\infty$ for any bounded interval $J\subset I$;
\item
Hypothesis $\ref{hyp:existence_uniqueness}(ii)$ is satisfied with $\A_{\bm\eta}$ being replaced by $\tilde\A={\rm Tr}(QD^2)+\langle {\bf b}, \nabla\rangle$, where ${\bf b}=(b_1,\ldots,b_m)$.
\end{enumerate}
\end{hyp}

\begin{rmk}
{\rm
\begin{enumerate}[\rm(i)]
\item
Hypothesis $\ref{hyp:existence_uniqueness}$(i) can be replaced with the weaker requirement
that $\mathcal K_{\eta,\varepsilon}$ is bounded from below in $J\times\Rd$, uniformly with respect to $\eta\in\partial B_1$, for any bounded interval $J\subset I$.
Indeed, in this case, for any $J$ as above, let $c_J>0$ be such that
$\mathcal K_{\eta,\varepsilon}\geq -c_J$  in $J\times\R^d$ for any $\eta \in\partial B_1$. The change of unknowns
$\vv(t,x):=e^{-c_J(t-s)/4}\uu(t,x)$ transforms the elliptic operator $\bm{\mathcal A}$ into the operator
$\bm{\mathcal A}-c_J/4$, which satisfies Hypothesis $\ref{hyp:existence_uniqueness}$(i) and, clearly, the uniqueness of $\vv$ is equivalent
to the uniqueness of $\uu$.
\item
In the scalar case when the elliptic operator in \eqref{operat-A} is ${\mathcal A}={\rm Tr}(QD^2)+\langle {\bf b},\nabla\rangle+c$ and $c$ is bounded from above (otherwise, Proposition \ref{prop:max_principle} fails in general), taking $\varepsilon=1$ and $\kappa=c$, one easily realizes that Hypothesis \ref{hyp:existence_uniqueness}(i) is trivially satisfied. Moreover, Hypothesis \eqref{hyp:existence_uniqueness}(ii) reduces to require the existence of a Lyapunov function for the operator $\A+c$,
for any bounded interval $J\subset I$. This condition seems to be much more general than that typically assumed (i.e., the existence of a Lyapunov function for the operator $\A$) and, as the proof of Proposition \ref{prop:max_principle} shows, it appears naturally when one considers the problem solved by $u^2$, where $u$ is the solution to the Cauchy problem \eqref{eq:cauchy_problem_system}, with $\A$ and $f\in C_b(\Rd)$ instead of $\bm{\mathcal A}$ and ${\bf f}$.
In view of this fact, Proposition \ref{prop:max_principle} and Theorem \ref{thm:existence_solution} below hold true provided there exists $\gamma\ge 1$ and a Lyapunov function for the operator $\A+\gamma c$.
\item
Hypotheses \ref{hyp:existence_uniqueness} and \ref{hyp_comp} are independent in
general. Indeed Hypotheses \ref{hyp_comp}(i)-(ii) imply Hypothesis \ref{hyp:existence_uniqueness}(i), whereas Hypothesis \ref{hyp:existence_uniqueness}(ii) is stronger  than Hypothesis \ref{hyp_comp}(iii). Indeed, if Hypothesis \ref{hyp_comp}(i) holds, then
$\sum_{i,j=1}^d(Q^{-1})_{ij}[\langle \tilde B_i{\bm \eta},{\bm \eta}\rangle\langle \tilde B_j{\bm \eta},{\bm \eta}\rangle-\langle \tilde B_i{\bm \eta},\tilde B_j{\bm \eta}\rangle]$
is negative and of order $\lambda_Q^{2\sigma-1}$. This fact together with Hypothesis \ref{hyp_comp}(ii) implies Hypothesis  \ref{hyp:existence_uniqueness}(i) (taking
(i) into account). On the other hand, assuming Hypothesis \ref{hyp_comp}(i), the function ${\mathcal K}_{\bm \eta,\varepsilon}$ which, clearly, depends only upon $\tilde{B}_i$ (since the diagonal parts cancel), can be of order less than $\lambda_Q^{1-2\sigma}$. For instance, assume $d=m=2$, $Q={\rm diag}(\lambda_Q,\Lambda_Q)$,
$B_1=b_1I_2$ diagonal and $\tilde{B}_2\neq 0$. Then,
${\mathcal K}_{\bm \eta,0}=\Lambda_Q^{-1}(\langle\tilde{B}_2{\bm \eta},{\bm \eta}\rangle^2-|\tilde{B}_2{\bm \eta}|^2)-4\langle C{\bm \eta},{\bm \eta}\rangle$
is bounded from below if $\Lambda_C+\xi^2\lambda_Q^{2\sigma}\Lambda_Q^{-1}<+\infty$,
which is weaker than the condition in Hypothesis \ref{hyp_comp}(ii) if $\lambda_Q=o(\Lambda_Q)$.
Finally, concerning Hypotheses \ref{hyp:existence_uniqueness}(ii) and \ref{hyp_comp}(iii), the latter requires the existence of a Lyapunov function for
{\em one} decomposition of each drift matrix, while the former requires the existence of a Lyapunov function for {\em any}
decomposition $B_i = {\bf b}_\eta I_m + \tilde{B}_{\eta,i}$, $\eta \in \partial B_1$.
\end{enumerate}
}
\end{rmk}

\subsection{The case when Hypotheses \ref{hyp:existence_uniqueness} hold true}
The uniqueness of the classical solution to problem \eqref{eq:cauchy_problem_system}
which is bounded in any strip $[s,T]\times\R^d$, $T>s \in I$, is a straightforward consequence of the following
result, whose proof is an adaption to our situation of the method in \cite[Thm. 8.7]{KreMaz12Max}
which deals with the case of bounded coefficients.

\begin{prop}
Let $s \in I$, $\f\in C_b(\R^d;\R^m)$ and let $\uu$ be a locally in time bounded classical solution
to problem \eqref{eq:cauchy_problem_system}. If Hypotheses
$\ref{hyp:existence_uniqueness}$ hold true, then
$\|\uu(t,\cdot)\|_{\infty}\leq e^{\varepsilon\kappa_0(t-s)}\|\f\|_{\infty}$ for any $t>s$.
\label{prop:max_principle}
\end{prop}

\begin{proof}
Fix $T>s \in I$ and let $\mu=\mu_{[s,T]}$ and $\varphi=\varphi_{[s,T]}$. Up to replacing
$\mu$ with a larger constant if needed, we can assume that
$\sup_{\eta \in \partial B_1}\sup_{[s,T]\times\R^d}\left (\A_{\eta}\varphi-\mu\varphi\right )<0$
and $\mu>2\varepsilon\kappa_0$.

For any $n\in\N$, we set
$v_n(t,x):=e^{-\mu(t-s)}|\uu(t,x)|^2-e^{-(\mu-2\varepsilon\kappa_0)(t-s)}\|\f\|_{\infty}^2-n^{-1}\varphi(x)$ for any $(t,x)\in [s,T]\times \Rd$.
As it is immediately seen,
\begin{align*}
D_tv_n(t,\cdot)
=& e^{-\mu(t-s)}(\A_0(t)+2\varepsilon\kappa(t,\cdot)
-\mu)|\uu(t,\cdot)|^2+(\mu\!-\!2\varepsilon\kappa_0)e^{-(\mu-2\varepsilon\kappa_0)(t-s)}\|\f\|_{\infty}^2\\
&-2e^{-\mu(t-s)}V(t,\cdot,D_1\uu(t,\cdot),\ldots,D_d\uu(t,\cdot),\uu(t,\cdot)),
\end{align*}
for any $t\in (s,T]$, where
$V(\cdot,\cdot,{\bm\xi}^1,\ldots,{\bm \xi}^d,{\bm \zeta}):=
\sum_{i,j=1}^dQ_{ij}\langle{\bm \xi^i},{\bm \xi^j}\rangle-\sum_{j=1}^d\langle B_j{\bm\xi^j},{\bm\zeta}\rangle-\langle (C-\varepsilon\kappa){\bm\zeta},{\bm\zeta}\rangle$
for any ${\bm\xi^1},\ldots,{\bm\xi^d},{\bm\zeta}\in\R^m$. Since $\mu>2\varepsilon\kappa_0$, we can estimate
\begin{align}
&D_tv_n(t,\cdot) -(\A_0(t)+2\varepsilon\kappa(t,\cdot)-\mu)v_n(t,\cdot)
-2\varepsilon(\kappa(t,\cdot)-\kappa_0)e^{-(\mu-2\varepsilon\kappa_0)(t-s)}\|{\bf f}\|_{\infty}\nonumber\\
<&n^{-1}(\A_0(t)+2\varepsilon\kappa(t,\cdot)-\mu)\varphi
-2e^{-\mu(t-s)}V(t,\cdot,D_1\uu(t,\cdot),\ldots,D_d\uu(t,\cdot),\uu(t,\cdot)),
\label{davide-lecce}
\end{align}
for any $t\in (s,T]$ and $x\in\R^d$.

Our aim consists in proving that $v_n\leq0$ in $[s,T]\times\R^d$ for any $n\in\N$. Indeed, in this case, letting $n\rightarrow\infty$ and recalling that
$T$ has been arbitrarily fixed, we will obtain $e^{-2\varepsilon\kappa_0 (t-s)}|\uu(t,x)|^2-\|\f\|_{\infty}^2\le 0$ for any $t\in [s,T]$ and $x\in\Rd$,
i.e, the estimate in the statement will follow from the arbitrariness of $T>s$.

Since $v_n$ tends to $-\infty$ as $|x|\rightarrow+\infty$, uniformly with respect to $t\in [s,T]$, it has a
maximum attained at some point $(t_0,x_0)\in [s,T]\times\R^d$. If $t_0=s$, then we are done since $v_n(s,\cdot)<0$.
Suppose that $t_0>s$ and assume, by contradiction, that $v_n(t_0,x_0)>0$. Since $2\varepsilon\kappa(t_0,x_0)-\mu\le 2\varepsilon\kappa_0-\mu<0$, the left-hand side of \eqref{davide-lecce} is strictly positive at $(t_0,x_0)$.

Let us prove that the right-hand side of \eqref{davide-lecce} is nonpositive at $(t_0,x_0)$. This will lead us to a contradiction and we
will conclude that $v_n\le 0$ in $[s,T]\times\R^d$.

Since $\nabla_x v_n$ vanishes at $(t_0,x_0)$,
$\langle D_j\uu(t_0,x_0),\uu(t_0,x_0)\rangle=e^{\mu(t_0-s)}D_j\varphi(x_0)/(2n)$ for any $j=1,\ldots,d$. Hence, it is enough to
show that the maximum of the function
\begin{align*}
F_{n,{\bm\zeta}}({\bm\xi^1},\ldots,{\bm\xi^d}):=n^{-1}(\A_0(t_0)+2\varepsilon\kappa(t_0,\cdot)-\mu)\tilde\varphi(x_0)-2V(t_0,x_0,{\bm\xi^1},\ldots,{\bm\xi^d},{\bm\zeta}),
\end{align*}
in the set $\Sigma=\left\{({\bm\xi^1},\ldots,{\bm\xi^d})\in\R^{md}: \langle {\bm\xi^j},{\bm\zeta}\rangle=(2n)^{-1}D_j\tilde\varphi(x_0),\,j=1,\ldots,d\right\}$ is nonpositive,
where $\tilde\varphi=e^{\mu(t_0-s)}\varphi$. Note that the function $F_{n,{\bm\zeta}}$ has a maximum in $\Sigma$ attained at some point
$({\bm\xi^1_0},\ldots,{\bm\xi^d_0})$, since
it tends to $-\infty$ as $\|({\bm\xi^1},\ldots,{\bm\xi^d})\|\to +\infty$. Applying the Lagrange multipliers theorem, we easily see that
$({\bm\xi^1_0},\ldots,{\bm\xi^d_0})$ satisfies
\begin{equation}
2\sum_{k=1}^dQ_{jk}(t_0,x_0)\xi^k_{0,i}-\sum_{k=1}^m(B_j(t_0,x_0))_{ki}\zeta_{k}-\gamma_j\zeta_{i}=0,\qquad\; i=1,\ldots,d,\;\,j=1,\ldots,m,
\label{sette}
\end{equation}
for some real numbers $\gamma_1,\ldots,\gamma_d$, where $\xi^k_{0,i}$ and $\zeta_i$ ($i=1,\ldots,m$) denote, respectively, the components of the vectors
${\bm\xi^k_0}$ and ${\bm\zeta}$. Multiplying both sides of \eqref{sette} by $\zeta_i$ and summing over $i$, we get
$\gamma_j=|{\bm\zeta}|^{-2}[n^{-1}(Q(t_0,x_0)\nabla\tilde\varphi(x_0))_j-\langle B_j(t_0,x_0){\bm\zeta},{\bm\zeta}\rangle]$ for any $j=1,
\ldots,m$.
Replacing the expression of $\gamma_j$ in \eqref{sette}, we deduce that
\begin{align*}
{\bm\xi^j_0}=&\frac{1}{2n}|{\bm\zeta}|^{-2}{\bm\zeta} D_j\tilde\varphi(x_0)+\frac{1}{2}\sum_{k=1}^d(Q^{-1})_{jk}(t_0,x_0)\left[
B_k^*(t_0,x_0){\bm\zeta}-|{\bm\zeta}|^{-2}\langle B_k(t_0,x_0){\bm\zeta},{\bm\zeta}\rangle{\bm\zeta}\right],
\end{align*}
for $j=1,\ldots,d$.
Hence, a direct computation shows that
\begin{align*}
V(t_0,x_0,{\bm\xi_0^1},\ldots,{\bm\xi_0^d})=&
\frac{1}{4n^2|{\bm\zeta}|^2}|\sqrt{Q}(t_0,x_0)\nabla\tilde\varphi(x_0)|^2
+\frac{1}{4}|\bm\zeta|^2{\mathcal K}_{\bm\zeta/|\bm\zeta|}(t_0,x_0)\\
&-\frac{1}{2n|{\bm\zeta}|^2}\sum_{j=1}^dD_j\tilde\varphi(x_0)\langle B_j(t_0,x_0){\bm\zeta},{\bm\zeta}\rangle
\end{align*}
and, consequently,
\begin{align*}
\max_{\Sigma}
F_{n,{\bm\zeta}}
=\frac{1}{n}(\A_{{\bm\zeta}/|{\bm\zeta}|}(t_0)\tilde\varphi(x_0)-\mu\tilde\varphi(x_0))
-\frac{1}{2n^2|{\bm\zeta}|^2}|\sqrt{Q}(t_0,x_0)\nabla\tilde\varphi(x_0)|^2
-\frac{1}{2}|{\bm\zeta}|^2{\mathcal K}_{\bm\zeta/|\bm\zeta|,\varepsilon}(t_0,x_0).
\end{align*}
By Hypothesis \ref{hyp:existence_uniqueness}(ii) and the choice of $\mu$, the right-hand side of the previous formula is nonpositive.
The proof is now complete.
\end{proof}

We now turn to show the existence of a unique locally in time bounded classical solution $\uu$ to problem
\eqref{eq:cauchy_problem_system}.

\begin{thm}
Under Hypotheses $\ref{hyp:existence_uniqueness}$, for any $\f\in C_b(\R^d;\R^m)$ and $s \in I$, the Cauchy problem
\eqref{eq:cauchy_problem_system}
admits a unique locally in time bounded classical solution $\uu$.
Moreover, $\uu\in C^{1+\alpha/2,2+\alpha}_{\rm loc}((s,+\infty)\times\Rd)$ and satisfies
\begin{equation}
\|\uu(t,\cdot)\|_{\infty}\leq e^{\varepsilon\kappa_0(t-s)}\|\f\|_{\infty}, \qquad\;\,t>s.
\label{eq:norma_condition_classical solution}
\end{equation}
\label{thm:existence_solution}
\end{thm}

\begin{proof}
Fix ${\bf f}\in C_b(\R^d;\R^m)$ and let $\uu_n$ be the unique classical solution to the Cauchy-Dirichlet problem
\begin{equation}\label{p_palla}
\left\{
\begin{array}{lll}
D_t\uu_n(t,x)=(\bm{\mathcal A}\uu_n)(t,x), & t\in(s,+\infty), &  x\in B_n, \\[.5mm]
\uu_n(t,x)={\bf 0}, & t\in(s,+\infty), & x\in\partial B_n, \\[.5mm]
\uu_n(s,x)=\f(x), & & x\in B_n,
\end{array}
\right.
\end{equation}
(see \cite[Thm. IV.5.5]{LadSolUra68Lin}). By classical solution we mean a function which belongs to $C^{1,2}((s,+\infty)\times B_n)$ which
is continuous in $([s,+\infty)\times\overline{B}_n)\setminus (\{s\}\times\partial B_n)$.

Let us prove that the sequence
$(\uu_n)$ converges to a solution to problem
\eqref{eq:cauchy_problem_system} which satisfies the properties in the statement. The same arguments as in the proof of Proposition \ref{prop:max_principle} show that
\begin{equation}
\|\uu_n(t,\cdot)\|_{\infty}\le e^{\varepsilon\kappa_0(t-s)}\|{\bf f}\|_{\infty},\qquad\;\,t>s.
\label{stima-un}
\end{equation}
Hence, the interior Schauder estimates in Theorem \ref{thm-A2}
guarantee that, for any compact set $K \subset (s,+\infty)\times \Rd$ and large $n$, the sequence $\|\uu_n\|_{C^{1+\alpha/2,2+\alpha}(K;\R^m)}$ is bounded by a constant independent of $n$.
By the Ascoli-Arzel\`a Theorem, a diagonal argument and the arbitrariness of $K$, we can determine a subsequence $(\uu_{n_k})$ which converges to a function
$\uu\in C^{1+\alpha/2,2+\alpha}_{\rm loc}((s,+\infty)\times\R^d;\R^m)$ in $C^{1,2}(K;\R^m)$
for any $K$ as above. Clearly, $\uu$ satisfies
the differential equation in \eqref{eq:cauchy_problem_system} as well as
the estimate \eqref{eq:norma_condition_classical solution}, as it is easily seen
letting $n\to +\infty$ in \eqref{stima-un}; we just need to
show that $\uu$ is continuous in $t=s$ and it therein equals the function $\f$.
As a byproduct, we will deduce that the whole sequence
$(\uu_n)$ converges in $C^{1,2}(K;\R^m)$, for any compact set
$K$ as above, since any subsequence of $(\uu_n)$ has a
subsequence which converges in $C^{1,2}(K;\R^m)$.

Fix $R\in\N$ and let $\vartheta$ be any smooth function such that $\chi_{B_{R-1}}\le \eta \le \chi_{B_R}$.
For any $n_k>R$ the function $\vv_k:=\eta\uu_{n_k}$ belongs to $C([s,T]\times\overline{B}_R;\R^m)\cap C^{1,2}((s,T]\times B_R;\R^m)$,
it vanishes on $(s,T]\times \partial B_R$, $\vv_k(s,\cdot)=\vartheta\f$ and
$D_t\vv_k-\bm{\mathcal A}\vv_k={\bf g}_k$ in $(s,T]\times B_R$,
where ${\g}_k=-{\rm Tr}(QD^2\eta)\uu_{n_k}-2(J_x\uu_{n_k})Q\nabla\eta-\sum_{j=1}^d(B_j\uu_{n_k})D_j\eta$, for any $n_k>R$.
Since the function $t\mapsto (t-s)\|\uu_{n_k}(t,\cdot)\|_{C^2_b(B_{n_k})}$ is bounded in $(s,s+1)$ (by a constant depending on $k$) we can apply Proposition \ref{prop-A1} and, taking \eqref{stima-un} into account, we can estimate $|\g_k(t,x)|\le c_R(1+(t-s)^{-1/2})\|\f\|_\infty$, for any $(t,x)\in (s,s+1)\times B_R$ and any $n_k>R$, where $c_R$ is a positive constant independent of $k$. Let us represent $\vv_k$ by means of the variation-of-constants formula
\begin{eqnarray*}
\vv_k(t,x)=({\bf G}_R^{\mathcal D}(t,s)(\vartheta\f))(x)+\int_s^t ({\bf G}_R^{\mathcal D}(t,r){\bf g}_k(r,\cdot))(x)dr,\qquad\;\,t\in [s,T],\;\,x\in B_R.
\end{eqnarray*}
Recalling that $\vv_k\equiv\uu_{n_k}$ in $B_{R-1}$ and taking the previous estimate into account, it follows that
$|\uu_{n_k}(t,\cdot)-\f|\le |{\bf G}_R^{\mathcal D}(t,s)(\eta\f)-\f|
+c'_R\sqrt{t-s}\|\f\|_{\infty}$ in $B_{R-1}$, for any $t\in (s,s+1)$
and some positive constant $c'_R$ independent of $k$. Letting first $k$ tend to $+\infty$ and, then, $t$ tend to $s^+$,
we deduce that $\uu$ is continuous at $t=s$ for any $x\in B_{R-1}$. Since $R\in\N$ is arbitrary, we conclude that $\uu\in C([s,T]\times\R^d;\R^m)$ and $\uu(s,\cdot)=\f$.
\end{proof}

\subsection{The case when Hypotheses \ref{hyp_comp} hold true}
Now we show that the assertions in Theorem \ref{thm:existence_solution} continue to hold if,
as an alternative to Hypotheses \ref{hyp:existence_uniqueness}, we consider Hypotheses \ref{hyp_comp}.
Note that, since we no longer assume Hypotheses  \ref{hyp:existence_uniqueness}, we can not apply
Proposition \ref{prop:max_principle} to guarantee
the uniqueness of the solution to the Cauchy problem \eqref{eq:cauchy_problem_system}.
The role of the following theorem is twofold. First, it replaces Proposition \ref{prop:max_principle},
and, combined with Theorem \ref{thm:existence_solution}, it shows that
the Cauchy problem \eqref{eq:cauchy_problem_system} admits a unique locally in time bounded classical solution $\uu$.
Secondly, it shows that, for any $\f \in C_b(\R^d;\R^m)$, the function $|\uu|^2$
can be estimated pointwise in terms of $G(t,s)|\f|^2$.

\begin{thm}
\label{point_prop}
Let us assume that Hypotheses $\ref{hyp_comp}$ are satisfied. Then, for any $s\in I$ and
any $\f\in C_b(\Rd;\R^m)$, the Cauchy problem \eqref{eq:cauchy_problem_system} admits a unique locally in time bounded classical solution
$\uu$. Moreover, $\uu\in C^{1+\alpha/2,2+\alpha}_{\rm loc}((s,+\infty)\times\Rd)$ and, for any $T>s \in I$,
there exists a positive constant $c= c(s,T)$ such that
\begin{equation}\label{pointwise}
|\uu(t,x)|^2\leq c(G(t,s)|\f|^2)(x),\qquad\;\,(t,x)\in [s,T]\times \Rd,\;\,\f\in C_b(\Rd;\R^m).
\end{equation}
\end{thm}

\begin{proof}
Clearly, estimate \eqref{pointwise} yields the uniqueness of the solution to problem \eqref{eq:cauchy_problem_system}. Moreover,
the arguments here below can also be applied to prove that the solution $\uu_n$ to the Cauchy problem \eqref{p_palla} satisfies \eqref{pointwise}, with $\Rd$ being replaced by $B_n$. This estimate replaces \eqref{stima-un} and allows us to repeat verbatim the proof of Theorem \ref{thm:existence_solution} getting
the existence of a classical solution $\uu$ to the problem \eqref{eq:cauchy_problem_system}.
To prove \eqref{pointwise}, we fix $\f\in C_b(\Rd;\R^m)$, $T>s\in I$ and  consider the function $v:[s,T]\times\R^d\to\R$, defined by
$v(t,x):=e^{-2H(t-s)}|\uu(t,x)|^2-(G(t,s)|\f|^2)(x)$ for any $(t,x)\in [s,T] \times \Rd$,
where $H=H_{[s,T]}$ is defined in Hypothesis \ref{hyp_comp}(ii). The function $v$ belongs to $C_b([s,T]\times \Rd) \cap C^{1,2}((s,T)\times\Rd)$ and $v(s,\cdot)\equiv 0$. Moreover, taking Hypothesis \ref{hyp_comp}(i) into account,
by a straightforward computation we get
$D_tv(t,x)=\tilde\A v(t,x)+2e^{-2H(t-s)}g(t,x)$
for any $(t,x)\in (s, T]\times \Rd$, where
$g=\sum_{i=1}^d\langle \tilde B_iD_i\uu,\uu\rangle
-{\rm Tr}(J_x\uu Q (J_x\uu)^T)+\langle C\uu,\uu\rangle-H|\uu|^2$.
From Hypothesis \ref{base}, the
Young and Cauchy-Schwarz inequalities, and Hypothesis \ref{hyp_comp}(i)
we get
\begin{align*}
g\le & -\lambda_Q|J_x\uu|^2+ 2m\xi\lambda_Q^{\sigma}|\uu|\sum_{i=1}^d|D_i\uu|+(\Lambda_C-H)|\uu|^2\\
\le & (\varepsilon dm^2\xi^2-1)\lambda_Q|J_x\uu|^2+[(4\varepsilon)^{-1}\lambda_Q^{2\sigma-1}+\lambda_C-H]|\uu|^2
\end{align*}
in $[s,T] \times \Rd$, where $\varepsilon=\varepsilon(t)$ is an arbitrary positive function.
Choosing $\varepsilon=(m^2\xi^2 d)^{-1}$ and taking Hypothesis \ref{hyp_comp}(ii)  into account, we get $D_t v-\tilde\A v\le 0$
in $(s,T]\times\Rd$.
The maximum principle in \cite[Thm. 2.1]{KunLorLun09Non} shows that
$v\leq 0$ in $[s,T]\times \Rd$, which is the claim with $c=e^{2H(T-s)}$.
\end{proof}

\begin{rmk}
\label{rem-neumann}
{\rm We stress that the arguments used in the proof of Theorem \ref{thm:existence_solution} (and, hence, in the proof of Theorem \ref{point_prop}) to guarantee the existence of a solution to problem
\eqref{eq:cauchy_problem_system} works as well if we approximate this problem by homogeneous Neumann-Cauchy problems in the ball $B_n$.
We will use this approximation in the proof of Theorem \ref{thm-comp-rev}.}
\end{rmk}

As a consequence of Theorem \ref{thm:existence_solution} (resp. Theorem \ref{point_prop}) we can associate an evolution operator $\{{\bf G}(t,s)\}_{t\geq s\in I}$ to  $\bm{\mathcal A}$ in $C_b(\Rd,\Rm)$, by setting ${\bf G}(\cdot,s)\f:=\uu$, where $\uu$ is the unique locally in time bounded classical solution to the Cauchy problem
\eqref{eq:cauchy_problem_system}. The uniqueness statement in Proposition \ref{prop:max_principle} and Theorem \ref{point_prop} yield the evolution property of the family $\{{\bf G}(t,s)\}_{t\geq s\in I}$. Moreover, the estimates
\eqref{eq:norma_condition_classical solution} and \eqref{pointwise} (together with the estimate
$\|G(t,s)\|_{\mathcal{L}(C_b(\Rd))}\le 1$), imply
that the family  is an evolution operator in $C_b(\Rd;\R^m)$.
Moreover, for any $T>s \in I$,
\begin{equation}
\|\G(t,s)\f\|_{\infty}\le c(t-s)\|\f\|_{\infty},\qquad\;\, t\in [s,T],\;\,\f\in C_b(\R^d;\R^m),
\label{stima-evol-op}
\end{equation}
with $c(t-s)= e^{\varepsilon \kappa_0(t-s)}$ (resp. $c(t-s)=e^{H_{[s,T]}(t-s)}$).

\begin{rmk}
{\rm Note that, under Hypotheses \ref{hyp:existence_uniqueness},
 estimate \eqref{stima-evol-op} holds true for any $t>s\in I$.
The same is true even if Hypotheses \ref{hyp_comp} are satisfied with $J=I$. In this latter case
 $c(t-s)=e^{H(t-s)}$, for any $t>s$ and some positive constant $H$, independent of $t$ and $s$.}
\end{rmk}

\section{Properties of the evolution family $\{{\bf G}(t,s)\}$}
Here, we investigate the main properties of the evolution family (simply denoted by) $\G(t,s)$.
All the results contained in this section hold true when at least one between Hypotheses \ref{hyp:existence_uniqueness}
and Hypotheses \ref{hyp_comp} are satisfied.
We start by proving some continuity properties of ${\bf G}(t,s)$.

\begin{prop}
Let $(\f_n)$ be a bounded sequence of functions in $C_b(\R^d;\R^m)$. Then, the following properties are satisfied:
\begin{enumerate}[\rm (i)]
\item
if $\f_n$ converges pointwise to $\f\in C_b(\R^d;\R^m)$, then
${\bf G}(\cdot,s)\f_n$ converges to ${\bf G}(\cdot,s)\f$ in $C^{1,2}(D)$ for any compact set $D\subset (s,+\infty)\times\Rd$;
\item
if $\f_n$ converges to $\f$ locally uniformly in $\R^d$, then ${\bf G}(\cdot,s)\f_n$ converges to
${\bf G}(\cdot,s)\f$ locally uniformly in $[s,+\infty)\times\Rd$.
\end{enumerate}
\label{prop:convergence_properties}
\end{prop}

\begin{proof}
(i) From the inequality \eqref{stima-evol-op} and the interior Schauder estimates in Theorem \ref{thm-A2},
we deduce that $\sup_{n\in\N}\|{\bf G}(\cdot,s)\f_n\|_{C^{1+\alpha/2,2+\alpha}(D)}<+\infty$ for any compact set $D\subset (s,+\infty)\times\Rd$.
Therefore, using the same arguments as in the proof of Theorem \ref{thm:existence_solution}, we can prove that
there exists a function $\vv\in C^{1+\alpha/2,2+\alpha}_{\rm loc}((s,+\infty)\times\Rd)$ and a subsequence
${\bf G}(\cdot,s)\f_{n_k}$ which converges to $\vv$ in $C^{1,2}(D)$ as $k\to +\infty$, for any $D$ as above.
Clearly, $D_t\vv-\bm{\mathcal A}\vv=0$ in $(s,+\infty)\times\Rd$.

To complete the proof, we need to show that $\vv$ can be extended by continuity on $\{s\}\times\Rd$ and
$\vv(s,\cdot)\equiv\f$. Indeed, once this property is proved, we can conclude that $\vv$ is a local in time bounded classical solution to
problem \eqref{eq:cauchy_problem_system}. Hence, by uniqueness,
we conclude that $\vv\equiv {\bf G}(\cdot,s)\f$.
Since this argument can be applied to any subsequence of $({\bf G}(\cdot,s)\f_n)$ which converges in
$C^{1,2}((s,+\infty)\times\Rd)$, and the limit is ${\bf G}(\cdot,s)\f$, we conclude that the whole
sequence $({\bf G}(\cdot,s)\f_n)$ converges to ${\bf G}(\cdot,s)\f$ locally uniformly in $(s,+\infty)\times\Rd$.

To prove that $\vv$ can be extended by continuity at $t=s$, we fix $m, M\in\N$, with $m>M$.
From the proof of Theorem \ref{thm:existence_solution} and recalling that
$\sup_{n\in\N}\|\f_n\|_{\infty}<+\infty$, we deduce that
$|({\bf G}_m^{\mathcal D}(t,s)\f_n)(x)-\f_n(x)|\le |{\bf G}_M^{\mathcal D}(t,s)(\vartheta\f_n)(x)-\vartheta(x)\f_n(x)|
+c_M\sqrt{t-s}$
for any $(t,x)\in (s,s+1)\times B_{M-1}$, and some positive constant $c_M$ independent of $m$,
Thus, letting $m\to +\infty$ we conclude that
\begin{equation}
|({\bf G}(t,s)\f_n)(x)-\f_n(x)|\le |{\bf G}_M^{\mathcal D}(t,s)(\vartheta\f_n)(x)-\vartheta(x)\f_n(x)|
+c_M\sqrt{t-s},
\label{mensa-massi}
\end{equation}
for any $(t,x)\in (s,s+1)\times B_{M-1}$.
Next step consists in letting $n\to +\infty$. Clearly, the left-hand side of
\eqref{mensa-massi} converges to $|\vv(t,x)-\f(x)|$ for any $(t,x)\in (s,+\infty)\times\Rd$.
As far as the right-hand side is concerned, we observe that Riesz's representation theorem (see \cite[Rem. 1.57]{AmbFusPal}) shows
that there exists a family $\{p_{ij}^M(t,s,x,dy): t>s,\, x\in B_M,\, i,j=1,\ldots m\}$
of Borel finite measures such that
\begin{eqnarray*}
({\bf G}_M^{\mathcal D}(t,s)\g)_i(x)=\sum_{j=1}^m\int_{\Rd}g_j(y)p_{ij}^M(t,s,x,dy),\qquad\;\,\g\in C_c(B_M;\R^m),
\end{eqnarray*}
for any $t>s$, $x\in\Rd$, $i=1,\ldots,m$.
Since each function $\vartheta\f_n$ is compactly supported in $B_M$, from the previous representation
formula it follows that
${\bf G}_M^{\mathcal D}(\cdot,s)(\vartheta\f_n)$ converges to
${\bf G}_M^{\mathcal D}(\cdot,s)(\vartheta\f)$ pointwise in $[s,+\infty)\times\Rd$, as $n\to +\infty$.
Hence, we can take the limit in \eqref{mensa-massi} and conclude that
$|\vv(t,\cdot)-\f|\le |{\bf G}_M^{\mathcal D}(t,s)(\vartheta\f)-\vartheta\f|+c_M\sqrt{t-s}$ in $B_{M-1}$, for any $t\in (s,s+1)$, which implies that
$\vv$ can be extended by continuity to $\{s\}\times B_{M-1}$ by setting ${\bf v}(s,\cdot)=\f$, since the function ${\bf G}_M^{\mathcal D}(\cdot,s)(\vartheta\f)$
is continuous in $[s,+\infty)\times B_M$. The arbitrariness of $M$ allows us to complete the proof.

(ii) Fix $T>s\in I$. In view of property (i), we just need to prove that, for any compact set $K\subset\Rd$ and $\varepsilon>0$,
there exists $\delta \in (0,1)$ such that
\begin{equation}
\limsup_{n\to +\infty}\|{\bf G}(\cdot,s)\f_n-{\bf G}(\cdot,s)\f\|_{C([s,s+\delta]\times K;\R^m)}\le\varepsilon.
\label{emma}
\end{equation}
For this purpose, we fix $M\in\N$ such that $K\subset B_{M-1}$. Taking first the supremum over $B_{M-1}$ in both the sides of \eqref{mensa-massi} and, then,
the limsup as $n\to +\infty$, we conclude that
\begin{align*}
\limsup_{n\to +\infty}\|{\bf G}(\cdot,s)\f_n-\f_n\|_{C([s,s+\delta]\times\overline{B}_{M-1};\R^m)}\le
\|{\bf G}_M^{\mathcal D}(\cdot,s)(\vartheta\f)-\vartheta\f\|_{C([s,s+\delta]\times\overline{B}_{M-1};\R^m)}+c_M\sqrt{\delta}.
\end{align*}
Indeed, since $\vartheta\f_n$ tends to $\vartheta\f$, uniformly in $B_M$, from \eqref{stima-un}
${\bf G}_M^{\mathcal D}(\cdot,s)(\vartheta\f_n)$ converges to ${\bf G}_M^{\mathcal D}(\cdot,s)(\vartheta\f)$ uniformly in $[s,s+1]\times\overline{B}_{M-1}$.
Finally, splitting
${\bf G}(\cdot,s)\f_n-{\bf G}(\cdot,s)\f={\bf G}(\cdot,s)\f_n-\f_n+(\f_n-\f)+\f
-{\bf G}(\cdot,s)\f$, using the above estimate and recalling that $\f_n$ tends to $\f$, locally uniformly in $\Rd$, we deduce that
\begin{align*}
&\limsup_{n\to +\infty}\|{\bf G}(\cdot,s)\f_n-{\bf G}(\cdot,s)\f\|_{C([s,s+\delta]\times\overline{B}_{M-1};\R^m)}\\
\le &\|{\bf G}_M^{\mathcal D}(\cdot,s)(\vartheta\f)-\vartheta\f\|_{C([s,s+\delta]\times\overline{B}_{M-1}\R^m)}
+c_M\sqrt{\delta}+\|{\bf G}(\cdot,s)\f-\f\|_{C([s,s+\delta]\times\overline{B}_{M-1};\R^m)}.
\end{align*}
Since the functions ${\bf G}_M^{\mathcal D}(\cdot,s)(\vartheta\f)$ and ${\bf G}(\cdot,s)\f$ are continuous in
$[s,s+1]\times\overline{B}_{M-1}$, from
the previous estimate, it follows immediately that \eqref{emma} holds true.
\end{proof}

In the following theorem we prove that the evolution operator ${\bf G}(t,s)$ can be extended to the set of all the bounded
Borel measurable functions $\f:\Rd\to\Rm$. This is a consequence of the fact that for any $\f\in C_b(\Rd;\R^m)$ each component of ${\bf G}(t,s)\f$,
admits an integral representation formula in terms of some finite Borel measures. These measures
are absolutely continuous with respect to the Lebesgue measure but, in general, differently from the scalar case,
they are signed measures.

\begin{thm}
\label{prop-strong}
There exists a family $\{p_{ij}(t,s,x,dy): t>s\in I,\, x\in\Rd,\, i,j=1,\ldots, m\}$
of finite Borel measures, which are absolutely continuous with respect to the Lebesgue measure, such that
formula \eqref{eq:inte_intro} holds true for any $t>s$, $x\in\Rd$, $i=1,\ldots,m$. Moreover, through formula \eqref{eq:inte_intro}, the evolution operator $\G(t,s)$ extends to $B_b(\Rd;\R^m)$ with
a strong Feller evolution operator. Actually, $\G(\cdot,s)\f\in C^{1+\alpha/2,2+\alpha}_{\rm loc}((s,+\infty)\times\Rd;\R^m)$ for any $\f\in B_b(\Rd;\R^m)$ and $s\in I$.
\end{thm}

\begin{proof}
Throughout the proof, $s$ is arbitrarily fixed in $I$.
Since, for any $(t,x)\in (s,+\infty)\times\R^d$, the map
$\f\mapsto ({\bf G}(t,s)\f)(x)$
is bounded from $C_0(\Rd;\R^m)$ into $\R^m$, from the Riesz's Representation Theorem (see e.g., \cite[Rem. 1.57]{AmbFusPal}) it follows that there exists a family
$\{p_{ij}(t,s,x,dy): t>s\in I,\, x\in\Rd,\, i,j=1,\ldots m\}$ of finite Borel measures such that
\eqref{eq:inte_intro} is satisfied by any $\f\in C_0(\Rd;\R^m)$.
To extend the previous formula to any $\f\in C_b(\Rd;\R^m)$, it suffices to approximate such an $\f$, locally uniformly in $\Rd$, by a sequence $(\f_n)\subset C_0(\R^d;\R^m)$, write \eqref{eq:inte_intro}, with $\f$ being replaced by $\f_n$, and use both Proposition \ref{prop:convergence_properties}$(ii)$ and the dominated convergence theorem, applied to the positive and negative
parts of the measures $p_{ij}(t,s,x,dy)$, to let $n$ tend to $+\infty$.
Clearly, formula \eqref{eq:inte_intro} allows us to extend the evolution operator $\G(t,s)$ to $B_b(\R^d;\R^m)$.

Let us now prove that each measure $p_{ij}(t,s,x,dy)$ is absolutely continuous with respect to the Lebesgue measure.
Equivalently, we prove that, for any $(t,x)\in (s,+\infty)\times\Rd$ and any $i,j=1,\ldots,m$,
the positive and negative parts of $p_{ij}(t,s,x,dy)$ are
absolutely continuous with respect to the Lebesgue measure.
For this purpose, we recall that, by the Hahn decomposition theorem (see e.g., \cite[Thm. 6.14]{Rudin}), for any $(t,x)\in (s,+\infty)\times\Rd$ there exist two Borel
sets $P=P(t,s,x)$ and $N=N(t,s,x)$ such that the maps
$p_{ij}^+(t,s,x,dy)$ and $p_{ij}^-(t,s,x,dy)$, defined, respectively, by
$p_{ij}^+(t,s,x,A)=p_{ij}(t,s,x,A\cap P)$ and $p_{ij}^-(t,s,x,A)=-p_{ij}(t,s,x,A\cap N)$ for any Borel set $A\subset\Rd$, are positive measures and
$p_{ij}(t,s,x,dy)=p_{ij}^+(t,s,x,dy)-p_{ij}^-(t,s,x,dy)$.

Being rather long, we split the proof into several steps.

{\em Step 1.} We claim  that, for any $f\in B_b(\Rd)$ and $j=1,\ldots,m$, the function $\G(\cdot,s)(f{\bf e}_j)$ belongs to $C^{1+\alpha/2,2+\alpha}_{\rm loc}((s,+\infty)\times\Rd;\R^m)$,
$D_t\G(\cdot,s)(f{\bf e}_j)=\bm{\mathcal A}\G(\cdot,s)(f{\bf e}_j)$ in $(s,+\infty)\times\Rd$ and $\|\G(t,s)(f{\bf e}_j)\|_{\infty}\le c_{s,T}\|f\|_{\infty}$ for any $t \in (s, T]$ and any $T>s$, where $c_{s,T}$ is the constant in \eqref{stima-evol-op}.
Clearly, since $\G(\cdot,s)\f=\sum_{j=1}^m\G(\cdot,s)(f_j{\bf e}_j)$ for any $\f\in B_b(\Rd,\R^m)$, the claim implies that the function $\G(\cdot,s)\f$ enjoys the regularity properties in the statement and $\|\G(t,s)\f\|_{\infty}\le \sqrt{m}c_{s,T}\|\f\|_{\infty}$ for any $t\in (s,T]$.\footnote{This estimate will be improved in Corollary \ref{coro-strong}, removing
the constant $\sqrt{m}$.}

To prove the claim, we begin by recalling that the space $B(\Rd)$ of all the real valued Borel functions coincides with the set $B^{\omega_1}(\Rd)=\bigcup_{\eta<\omega_1}B^{\eta}(\Rd)$,
where, throughout this step, we denote by $\eta$ the ordinal numbers and $\omega_1$ is the first nonnumerable ordinal number. The sets $B^{\eta}(\Rd)$ are defined as follows:
$B^0(\Rd)=C(\Rd)$ and, if $\eta>0$, the definition of $B^{\eta}(\Rd)$ depends on the fact that $\eta+1$ is a successor ordinal or not. In the first case, $B^{\eta}(\Rd)$ is the set of the pointwise limits, everywhere in $\Rd$, of  sequences of functions in $B^{\eta-1}(\Rd)$; in the second one, $B^{\eta}(\Rd)=\bigcup_{\eta_0<\eta}B^{\eta_0}(\Rd)$.
Hence, any Borel function belongs to $B^{\eta}(\Rd)$ for some ordinal less than $\omega_1$. We refer the reader
to \cite[Chpt. 30]{Kuratowski} \cite[Introduction]{Morayne} and \cite{VanrooijSchik} for further details.

We fix $j\in\{1,\ldots,m\}$ and, for any ordinal $\eta<\omega_1$, we denote by ${\mathcal P}_j(\eta)$ the set of all the functions $f\in B^{\eta}_b(\Rd)$ which satisfy the claim, where, as usually, the subscript ``$b$'' means that we are considering bounded functions.
We use the transfinite induction to prove that ${\mathcal P}_j(\eta)=B^{\eta}_b(\Rd)$ for any ordinal less than $\omega_1$.
In view of Theorem \ref{thm:existence_solution}, ${\mathcal P}_j(0)=B^0_b(\Rd)=C_b(\Rd)$.
Fix now an ordinal $\eta$ and suppose that ${\mathcal P}_j(\beta)=B_b^{\eta}(\Rd)$ for any ordinal $\beta\le\eta$.
We first assume that $\eta+1$ is a successor ordinal.
In such a case, $f$ is the pointwise limit, everywhere in $\Rd$, of a sequence $(f_n)\in B_b^{\eta}(\Rd)$. By assumptions, $f$ is bounded; hence, up to replacing
$f_n$ by $f_n\wedge \|f\|_{\infty}$, which still belongs to $B^{\eta}_b(\Rd)$,
we can assume that $\|f_n\|_{\infty}\le \|f\|_{\infty}$ for any $n\in\N$.
Since $\G(\cdot,s)(f_n{\bf e}_j)\in C^{1+\alpha/2,2+\alpha}_{\rm loc}((s,+\infty)\times\Rd;\R^m)$ for any $n\in\N$, using the interior Schauder estimates in Theorem \ref{thm-A2},
as in the proof of Theorem \ref{thm:existence_solution}, we can prove that, up to a subsequence, ${\bf G}(t,s)(f_n{\bf e}_j)$
converges in $C^{1,2}(K;\R^m)$, for any compact set $K\subset (s,+\infty)\times\Rd$, to a function $\vv\in C^{1+\alpha/2,2+\alpha}_{\rm loc}((s,+\infty)\times\Rd;\R^m)$,
which solves the equation $D_t\vv=\bm{\mathcal A}\vv$ in $(s,+\infty)\times\Rd$ and satisfies
the estimate  $\|\vv(t,\cdot)\|_{\infty}\le c_{s,T}\|f\|_{\infty}$ for any $t \in (s, T]$ and any $T>s$.
The representation formula \eqref{eq:inte_intro} reveals that $\vv={\bf G}(\cdot,s)(f{\bf e}_j)$; hence, $f\in B^{\eta+1}_b(\Rd)$.
Suppose now that $\eta+1$ is a limit ordinal. Then, $f\in B^{\beta}_b(\Rd)$ for some ordinal $\beta$ less than $\eta+1$.
Since ${\mathcal P}_j(\beta)=B^{\beta}_b(\Rd)$, it is clear that $f\in {\mathcal P}_j(\eta+1)$, and we are done also in this case.

{\em Step 2.} Now, we prove that, for any $M>0$, there exists a positive constant $c$, depending on $M$, but being independent of $t$ and $f\in C_c(B_M)$, such that
\begin{align}
|(\G_M^{\mathcal D}(t,s)(f{\bf e}_j))(x)|\le |(G_{M,0}^{\mathcal D}(t,s)f)(x)|+c\sqrt{t-s}\,\|f\|_{\infty},\qquad\;\,(t,x)\in (s,s+1)\times B_M.
\label{dec-14-12-2}
\end{align}
for any $j\in\{1,\ldots,m\}$.
Here, $G_{M,0}^{\mathcal D}(t,s)$ denotes the evolution operators associated
with the realization of the operator $\A_0$ in $C_b(B_M)$, with homogeneous Dirichlet boundary conditions.

In the rest of the proof, we denote by $c$ a positive constant, which is independent of $t\in (s,s+1)$, $x\in B_M$ and may vary from line to line.

Fix $f\in C_c(B_M)$, $j\in\{1,\ldots,m\}$, set $\uu=\G_M^{\mathcal D}(\cdot,s)(f{\bf e}_j)$ and observe that
\begin{equation}
\uu(t,x)=(G_{M,0}^{\mathcal D}(t,s)f)(x){\bf e}_j+\int_s^t(\G_{M,0}^{\mathcal D}(r,s)\g(r,\cdot))(x)dr,\qquad\;\,(t,x)\in (s,s+1)\times B_M,
\label{dec-14-12-0}
\end{equation}
where $(\G_{M,0}^{\mathcal D}(t,s)\h)_k=G_{M,0}^{\mathcal D}(t,s)h_k$, for any $k=1,\ldots,m$ and any vector valued function $\h$, and
$\g=\sum_{i=1}^dB_jD_j\uu+C\uu$.
Differentiating both sides of \eqref{dec-14-12-0}, taking the
norms and using the estimate $\|G_{M,0}^{\mathcal D}(t,s)\psi\|_{C_b^1(B_M)}\le c(t-s)^{-1/2}\|\psi\|_{\infty}$, which holds true for any $\psi\in C_b(B_M)$ and any $t\in (s,s+1)$ (see \cite[Thm. 4.6.3]{Fri64Par}),
we can estimate
\begin{align*}
\|J_x\uu(t,\cdot)\|_{C_b(B_M;\R^m)}
\le\frac{c}{\sqrt{t-s}}\|f\|_{\infty}+c\int_s^t\frac{1}{\sqrt{r-s}}\|J_x\uu(r,\cdot)\|_{C_b(B_M;\R^m)}dr,\qquad\;\,t\in (s,s+1).
\end{align*}
To get this estimate we also took advantage of the
fact that $\|\uu(t,\cdot)\|_{\infty}\le c\|f\|_{\infty}$ for any $t\in (s,s+1)$. The generalized Gronwall lemma (see \cite{Gronwall}) shows that
$\|J_x\uu(t,\cdot)\|_{C_b(B_M;\R^m)}\le c(t-s)^{-1/2}\|f\|_{\infty}$ for any $t\in (s,s+1)$.
We thus deduce that $\|\g(t,\cdot)\|_{C_b(B_M;\R^m)}\le c(t-s)^{-1/2}\|f\|_{\infty}$ for any $t\in (s,s+1)$ and, from \eqref{dec-14-12-0}, estimate \eqref{dec-14-12-2} follows at once.

{\em Step 3.} Here, we prove that, for any Borel set ${\mathcal O}\subset\Rd$ with zero Lebesgue measure, any $M>0$, $j\in\{1,\ldots,m\}$
and $t\in(s,s+1)$, it holds that
\begin{equation}
|({\bf G}(t,s)\chi_{\mathcal O}{\bf e}_j)(x)|\le c\sqrt{t-s},\qquad\;\,t\in (s,s+1),\;\,x\in B_{M/2}.
\label{eq:davide2}
\end{equation}
This inequality follows once we prove that
\begin{align}
|(\G(t,s)(f{\bf e}_j))(x)|\le |(G_{M,0}^{\mathcal D}(t,s)|f|)(x)|+c\sqrt{t-s}\,\|f\|_{\infty},
\label{eq:davide3}
\end{align}
for any $t$, $x$, $f$ and $j$ as above.
Indeed, it is well known that $G_{M,0}^{\mathcal D}(t,s)$ admits an integral representation (see
\cite[Thm. 3.16]{Fri64Par}) and this implies that $G_{M,0}^{\mathcal D}(t,s)$ can be extended to any function $f\in B_b(B_M)$ and $G_{M,0}^{\mathcal D}(t,s)\chi_{\mathcal O}=0$,
if ${\mathcal O}$ has null Lebesgue measure.

We first prove \eqref{eq:davide3} for functions $f\in C_b(\Rd)$.
For this purpose, we
fix $M>0$ and a function $\vartheta\in C_c^\infty(\Rd)$ such that $\chi_{B_{M/4}}\leq\vartheta\leq\chi_{B_{M/2}}$.
By the proof of Theorem \ref{thm:existence_solution}, ${\bf G}(\cdot,s)(f{\bf e}_j)$ is the local uniform limit in
$(s,+\infty)\times\R^d$ of the unique classical solution $\uu_n$ to the Cauchy problem \eqref{p_palla}, with $\f$ being replaced by $\vartheta f{\bf e}_j$,
Moreover, $\|{\bf G}_M^{\mathcal D}(t,r)\g_n(r,\cdot)\|_{L^\infty(B_M)}\le c_1\|\g_n(r,\cdot)\|_{L^{\infty}(B_{M/2})}
\le c_M\|f\|_{\infty}(1+(r-s)^{-\frac{1}{2}})$ in $B_{M-1}$, for any $r\in (s,t)$. Letting $n$ tend to $+\infty$, estimate \eqref{eq:davide3} follows
recalling that $|G_{M,0}^{\mathcal D}(t,s)(\vartheta f)|\le G_{M,0}^{\mathcal D}(t,s)|\vartheta f|\le G_{M,0}^{\mathcal D}(t,s)|f|$ in $B_M$.

By transfinite induction,  arguing as in Step 1, we extend \eqref{eq:davide3}
to any $f\in B_b(\Rd)$. To make the induction work, it suffices to observe that, if $f\in B_b(\Rd)$ is the pointwise limit everywhere in $\Rd$ of a sequence $(f_n)\subset B_b(\Rd)$ of functions which satisfy \eqref{eq:davide3} and $\|f_n\|_{\infty}\le \|f\|_{\infty}$ for any $n\in\N$, then, $f$ satisfies \eqref{eq:davide3} as well.
This can be seen, writing \eqref{eq:davide3} with $f$ being replaced by $f_n$, taking \eqref{eq:inte_intro} into account and letting $n\to +\infty$.

{\em Step 4.} We can now complete the proof. We fix $i,j\in \{1,\ldots,m\}$, $t_0>s$, $x_0\in\Rd$ and a Borel set $A$ with null Lebesgue measure.
Then, estimate \eqref{eq:davide2} shows  that
$|{\bf G}(t,s)(\chi_{A\cap R}{\bf e}_j)|\le c\sqrt{t-s}$ in $B_{M/2}$ for any $M>0$ and $t\in (s,s+1)$, where $R=P$ or $R=N$.
By the arbitrariness of $M$, it thus follows that ${\bf G}(t,s)(\chi_{A\cap R}{\bf e}_j)$
vanishes, locally uniformly in $\Rd$, as $t\to s^+$. Step 1 shows that the function $\vv$, defined by $\vv(s,\cdot)={\bf 0}$ and $\vv(t,\cdot)=\G(t,s)(\chi_{A\cap R}{\bf e}_j)$,
if $t>s$, belongs to $C^{1+\alpha/2,2+\alpha}_{\rm loc}((s,+\infty)\times\Rd;\R^m)\cap C([s,+\infty)\times\Rd;\R^m)$ and is bounded in each strip $[s,T]\times\Rd$. Moreover,
$D_t\vv=\bm{\mathcal A}\vv$ in $(s,+\infty)\times\Rd$ and
$\vv(s,\cdot)={\bf 0}$ in $\Rd$.
By Proposition \ref{prop:max_principle}, it follows that $\vv\equiv {\bf 0}$ in $(s,+\infty)\times\Rd$. Thus, we conclude that
$(\G(\cdot,s)(\chi_{A\cap R}{\bf e}_j))(x_0)={\bf 0}$, which implies that
$0=(\G(t_0,s)(\chi_{A\cap P}{\bf e}_j))_i(x_0)=p_{ij}^+(t_0,s,x_0,A)$ and, similarly, $p_{ij}^-(t_0,s,x_0,A)=0$.
\end{proof}

\begin{coro}
\label{coro-strong}
The following properties are satisfied.
\begin{enumerate}[\rm (i)]
\item
Estimate \eqref{stima-evol-op} is satisfied by any $\f\in B_b(\Rd;\R^m)$.
\item
Proposition $\ref{prop:convergence_properties}(i)$ holds true for any bounded sequence $(\f_n)$ of Borel
functions which converges pointwise $($almost everywhere in $\Rd)$ to a Borel measurable function $\f$.
\end{enumerate}
\end{coro}

\begin{proof}
(i) Fix $\f\in B_b(\Rd;\Rm)$ and let the sequence $(\f_n)\in C_b(\Rd;\R^m)$ converge to $\f$ almost everywhere in $\Rd$ and satisfy
the estimate $\|\f_n\|_{\infty}\le \|\f\|_{\infty}$. Since the measures $p_{ij}(t,s,x,dy)$ are absolutely continuous with respect to the Lebesgue measure for any $t>s\in I$ and any $x\in\Rd$,
by formula \eqref{eq:inte_intro} and the dominated convergence theorem, $\G(t,s)\f_n$ converges to $\G(t,s)\f$ pointwise
everywhere in $\Rd$, as $n\to +\infty$.
Using \eqref{stima-evol-op}, we can estimate $|\G(\cdot,s)\f_n|\le c_{s,T}\|\f_n\|_{\infty}\le c_{s,T}\|\f\|_{\infty}$
in $[s,T]\times\Rd$ and, letting $n$ tend to $+\infty$, we conclude the proof.

(ii) Fix $s\in I$ and $(\f_n)$, $\f$ as in the statement. For any $\varepsilon>0$, the functions ${\bf G}(s+\varepsilon,s)\f_n$
and ${\bf G}(s+\varepsilon,s)\f$ are bounded and continuous in $\Rd$, thanks to property (i) and Theorem \ref{prop-strong}. Moreover,
the proof of property (i) shows that ${\bf G}(s+\varepsilon,s)\f_n$ converges
pointwise in $\Rd$ to ${\bf G}(s+\varepsilon,s)\f$ as $n\to +\infty$. Splitting
$\G(t,s)\f_n=\G(t,s+\varepsilon)\G(s+\varepsilon)\f_n$ for any $n\in\N$ and using Proposition \ref{prop:convergence_properties}(i) we conclude the proof.
\end{proof}

\section{Compactness of $\G(t,s)$ in $C_b(\Rd;\Rm)$}\label{comp_sect}

In this section we study the compactness of the evolution operator ${\bf G}(t,s)$ in $C_b(\R^d;\R^m)$. First of all, we show that
it is equivalent to the tightness of the total variation of the measures $\{p_{ij}(t,s,x,dy): x \in \Rd\}$
introduced in Theorem \ref{prop-strong}.

\begin{thm}\label{comp_tigh}
Let assume that either $\ref{hyp:existence_uniqueness}$ or $\ref{hyp_comp}$ are satisfied and fix $b>a \in I$.
The evolution operator ${\bf G}(t,s)$ is compact in $C_b(\Rd;\R^m)$ for any $(t,s)\in\Lambda_{[a,b]}$ if and only if the families $\{|p_{ij}|(t,s,x,dy): x\in \Rd\}$ are tight for any $(t,s)\in \Lambda_{[a,b]}$ and $i,j\in\{1,\ldots,m\}$.
\end{thm}

\begin{proof}
Let us suppose that ${\bf G}(t,s)$ is compact in $C_b(\Rd;\R^m)$ for any $(t,s)\in\Lambda_{[a,b]}$. We fix $i,j\in\{1,\ldots,m\}$, $(t,s)\in\Lambda_{[a,b]}$, $n\in \N$, $x \in \Rd$ and recall that
\begin{align}\label{var_tot}
|p_{ij}|(t,s,x,\Rd\setminus B_n)&=\sup\left\{\int_{\Rd}f(y)p_{ij}(t,s,x,dy):\, f\in C_c(\Rd\setminus B_n),\, \|f\|_\infty\le 1\right\}\\
&\le\sup\,\{\|{\bf G}(t,s)(f{\bf e}_j)\|_\infty:\, f\in C_c(\Rd\setminus B_n),\, \|f\|_\infty\le 1\},\nnm
\end{align}
(see \cite[Prop. 1.4.7]{AmbFusPal}). Then, for any $n\in\N$, there exist functions $f_{n}\in C_c(\Rd\setminus B_n)$ with $\|f_{n}\|_\infty \le 1$ such that
\begin{equation}\label{acqua}
\sup_{x \in \Rd}|p_{ij}|(t,s,x,\Rd\setminus B_n)\le \|{\bf G}(t,s)(f_{n}{\bf e}_j)\|_\infty+n^{-1}.
\end{equation}
Clearly, $f_{n}{\bf e}_j$ vanishes pointwise in $\Rd$ as $n\to +\infty$ and $\|f_{n}{\bf e}_j\|_{\infty}\le 1$ for any $n\in\N$.
By compactness and Proposition \ref{prop:convergence_properties},  we can extract a subsequence $(f_{n_h}{\bf e}_j)$ such that ${\bf G}(t,s)(f_{n_h}{\bf e}_j)$ vanishes uniformly in $\Rd$,
as $h \to +\infty$. Now, writing \eqref{acqua} with $n$ being replaced by $n_h$ and letting $h \to +\infty$ we deduce the tightness of the family $\{|p_{ij}|(t,s,x,dy):\, x \in \Rd\}$.

Vice versa, let us suppose that the families $\{|p_{ij}|(t,s,x,dy): x\in \Rd\}$ are tight for any $(t,s)\in \Lambda_{[a,b]}$ and $i,j\in\{1,\ldots,m\}$.
We fix $(t,s)\in \Lambda_{[a,b]}$, $r \in (s,t)$ and consider the operators ${\bf R}_n:=\G(t,r)(\chi_{B_n}\G(r,s))$ in $C_b(\Rd;\R^m)$.
Since $\G(t,r)$ is strong Feller (see Theorem \ref{prop-strong}),
each operator ${\bf R}_n$ is bounded in $C_b(\Rd;\R^m)$. We claim that ${\bf R}_n$ is compact in $C_b(\Rd;\R^m)$ for any $n\in\N$.
To this aim, let $(\f_k)$ be a bounded sequence in $C_b(\Rd;\R^m)$. From the interior Schauder estimates in Theorem \ref{thm-A2}
it follows that the sequence $(\G(r,s)\f_k)$ is bounded in $C^{2+\alpha}(B_n;\R^m)$. Hence, there exists a subsequence
$(\G(r,s)\f_{k_j})$ converging uniformly in $B_n$ to some function $\g$ as $j\to +\infty$. As a byproduct, $\chi_{B_n}\G(r,s)\f_{k_j}$ converges to
$\chi_{B_n}\g$ uniformly in $\R^d$ as $j\to +\infty$. Since the estimate \eqref{stima-evol-op} holds true also for bounded Borel functions (see Corollary \ref{coro-strong}(i)),
we conclude that ${\bf R}_n\f_{k_j}$ converges uniformly in $\R^d$ to $\G(t,r)(\chi_{B_n}\g)$ as $j\to +\infty$. Hence, ${\bf R}_n$ is compact in $C_b(\R^d;\R^m)$.

To complete the proof, we show that ${\bf R}_n$ converges to
$\G(t,s)$ as $n\to +\infty$ in ${\mathcal L}(C_b(\Rd;\R^m))$. For this purpose we fix $\f\in C_b(\Rd,\R^m)$, $i\in \{1,\ldots,m\}$. Using formula \eqref{eq:inte_intro}
we can write
\begin{align*}
(\G(t,s)\f)_i(x)-({\bf R}_n\f)_i(x)=&(\G(t,r)(\chi_{\Rd\setminus B_n}\G(r,s)\f)_i(x)
=\sum_{k=1}^m\int_{\Rd\setminus B_n}(\G(r,s)\f(y))_kp_{ik}(t,r,x,dy)
\end{align*}
for any $x\in\Rd$. Hence,
\begin{align*}
\|\G(t,s)\f-{\bf R}_n\f\|_{\infty}\le &\sum_{i,k=1}^m\sup_{x \in \Rd}\int_{\Rd\setminus B_n}|(\G(r,s)\f(y))_k||p_{ik}|(t,r,x,dy)\\
\le &c_{a,b}\|\f\|_{\infty}\sum_{i,k=1}^m\sup_{x \in \Rd}|p_{ik}|(t,r,x,\Rd\setminus B_n).
\end{align*}

Letting $n$ tend to $+\infty$ and using the tightness of the family $\{|p_{ij}|(t,r,x,dy): x \in \Rd\}$,
we conclude that ${\bf R}_n$ converges to ${\bf G}(t,s)$ in ${\mathcal L}(C_b(\Rd;\R^m)$.
\end{proof}

In the following theorem we provide sufficient conditions for the compactness of $\G(t,s)$ when Hypotheses \ref{hyp_comp} are satisfied.
Actually, these assumptions guarantee the compactness of $G(t,s)$ in $C_b(\Rd)$, which has been already studied in \cite{AngLor10Com, Lun10Com} extending
the results in the autonomous case proved in \cite{MPW}.

\begin{thm}\label{compactness}
Let $J\subset I$ be a bounded interval. Under Hypotheses $\ref{hyp_comp}$, if $G(t,s)$ is compact in $C_b(\Rd)$ for every $(t,s)\in \Sigma_J$, then $\G(t,s)$ is compact in $C_b(\Rd;\R^m)$ for every $(t,s)\in \Sigma_J$.
In particular, if there exist a $C^2$ function
$W:\Rd\to \R$ such that $\lim_{|x|\rightarrow\infty}W(x)=+\infty$, a number $R>0$ and a convex increasing function $g:[0,+\infty)\to \R$ with
$1/g\in L^1((a,+\infty))$ for large $a$ and $\tilde\A W \le -g\circ W$ in $I\times (\Rd\setminus B_R)$, then $\G(t,s)$ is compact in $C_b(\Rd;\R^m)$ for any $t>s\in I$.
\end{thm}

\begin{proof}
In view of Theorem \ref{comp_tigh}, we show that the family
$\{|p_{ij}|(t,s,x,dy): x\in\Rd\}$ is tight for any $(t,s)\in\Sigma_J$ and $i,j\in\{1,\ldots,m\}$.
From \eqref{var_tot}, \eqref{pointwise} and the positivity of the evolution operator $G(t,s)$, it follows that
\begin{align*}
|p_{ij}|(t,s,x,\Rd\setminus B_n)\le & e^{H_J(t-s)/2}\sup\,\{((G(t,s)|f|^2)(x))^{1/2}:\, f\in C_c(\Rd\setminus B_n),\, \|f\|_\infty\le 1\}\\
\le & e^{H_J(t-s)/2}((G(t,s)\chi_{\Rd\setminus B_n})(x))^{1/2}=e^{H_J(t-s)/2}(g_{t,s}(x,\Rd\setminus B_n))^{1/2}
\end{align*}
for any $\Rd$, any $(t,s)\in\Sigma_J$ and $i,j\in\{1,\ldots,m\}$, where $g_{t,s}(\cdot,dy)$ are
the transition kernels associated with the evolution operator $G(t,s)$. The assertion now follows from \cite[Prop. 4.2]{AngLor10Com}, which shows that the compactness of the scalar evolution operator $G(t,s)$ in $C_b(\Rd)$, for any $(t,s)\in\Sigma_{J}$, is equivalent to the tightness of the family $\{g_{t,s}(x,dy): x\in\Rd\}$ for any $(t,s)\in\Sigma_{J}$.

The last assertion follows from \cite[Thm. 3.3]{Lun10Com}
\end{proof}

\section{Uniform Gradient estimates}
\label{sect-5}

In this section, we prove a (weighted) uniform gradient estimate satisfied by the function ${\bf G}(t,s)\f$, which, besides, its own interest, leads
to some remarkable consequences that we illustrate in the next section.
We assume the following additional assumptions.

\begin{hyp}
\label{Hyp-stime-gradiente}
\begin{enumerate}[\rm (i)]
\item
The coefficients $Q_{ij}$, $b_j$ and the entries of the matrices $B_i$ and C belong to $C^{0,1+\alpha}_{\rm loc}(\overline J\times\R^d)$ for any $i,j=1,\ldots,d$, some $\alpha\in (0,1)$ and some bounded interval $J$ with $\overline J\subset I$; further,
$\langle b(t,x),x\rangle\le b_0(t,x)|x|$ for any $t\in J$, $x\in\R^d$ and some negative function $b_0$;
\item
there exists a matrix-valued function $M$ such that  $M_{ij}=M_{ji}\in C^{1,2+\alpha}_{\rm loc}(J\times\R^d)$ for any $i,j=1,\ldots,m$ and
$\inf_{J\times\Rd}\lambda_M>0$.
\item
there exist positive functions
$\psi_h:J\times\Rd\to\R$ $(h=1,\ldots,6)$ such that
$|\tilde B_i|\leq \psi_1$, $|D_k\tilde B_i|\leq\psi_2$, $|D_kC|\le\psi_3$
$|D_tM|\le \psi_4$, $|D_kM|\le \psi_5$, $|D_kQ|\leq \psi_6$
in $J\times\R^d$, for any $i,k=1,\ldots,d$, and
\begin{align}
\sup_{J\times\Rd}\frac{\psi_1^2\Lambda_{M^2}}{\lambda_Q\Lambda_M^2}<+\infty,\qquad
\sup_{(t,x)\in J\times\Rd}\frac{(\Lambda_Q(t,x))^2\Lambda_{M^2}(t,x)}{(1+|x|^2)(\lambda_Q(t,x))^2}<+\infty
\label{iesolo-1}
\end{align}
\begin{align}
\;\;\;\;\;\sup_{(t,x)\in J\times\Rd}\frac{(\Lambda_Q(t,x))^2(1+|x|^4)^{-1}}{(2\Lambda_C(t,x)+\Lambda_{{\mathcal M}+{\mathcal M}^T}(t,x))}<+\infty,
\qquad\sup_{J\times\Rd}\frac{\Lambda_M^2\psi_3^2}{2\Lambda_C+\Lambda_{{\mathcal M}+{\mathcal M}^T}}<+\infty,
\label{iesolo-2}
\end{align}
\begin{align}
\lim_{|x|\to +\infty}\sup_{t\in J}{\mathcal H}(t,x)&=\lim_{|x|\to +\infty}\sup_{t\in J}\frac{\Lambda_M(t,x)\psi_2(t,x)+\psi_4(t,x)}{\lambda_M(t,x)(2\Lambda_C(t,x)+\Lambda_{{\mathcal M}+{\mathcal M}^T}(t,x))}\notag\\
&=\lim_{|x|\to +\infty}\sup_{t\in J}\frac{\Lambda_Q(t,x)\psi_5(t,x)}{b_0(t,x)}=0,
\label{iesolo-3}
\end{align}
where
${\mathcal H}=(\Lambda_Q^2\psi_5^2+\Lambda_M^2\psi_6^2+\lambda_Q\psi_1^2)(\lambda_Q\lambda_M^2(2\Lambda_C+\Lambda_{{\mathcal M}+{\mathcal M}^T}))^{-1}$
and
${\mathcal M}:=M(J_x b)^TM^{-1}-\sum_{j=1}^d
b_j(D_jM)M^{-1}-\sum_{i,j=1}^d Q_{ij}(D_{ij}M)M^{-1}$.
\end{enumerate}
\label{hyp:system_maximum_principle}
\end{hyp}

\begin{rmk}
{\rm
We stress that the second condition in \eqref{iesolo-1} forces $M$ to grow no faster than linearly as $|x|\to +\infty$.
This condition might seem a bit strong but, already in the classical scalar case when the coefficients are bounded, in general the gradient of
$G(t,s)f$ does not vanishes as $|x|\to +\infty$. For instance, consider the one-dimensional autonomous operator $\A=u''+u$ and take $f(x)=\sin(x)$ for any $x\in\R$. Then,
$G(t,s)f=f$ for any $t\ge s\in\R$ and, clearly, $D_xT(t)f$ does not vanish at infinity.
}\end{rmk}

\begin{thm}
\label{prop:system_gradient_estimates}
Assume that Hypotheses $\ref{hyp_comp}$ and $\ref{hyp:system_maximum_principle}$
are satisfied. Then, for any $s\in J$, the function $M(J_x\G(\cdot,s)\f)^T$ is bounded and continuous in $(J\cap (s,+\infty))\times \Rd$ and there exists a positive constant $c=c_J$ such that
\begin{equation}
\sqrt{t-s}\,\|M(t,\cdot)(J_x\G(t,s)\f)^T\|_{\infty}\leq c\|\f\|_\infty,\qquad\;\,t\in J\cap (s,+\infty),\;\,\f\in C_b(\Rd;\R^m).
\label{eq:system_system_gradient_estimate}
\end{equation}
\end{thm}
\begin{proof}
To simplify the notation we set $\uu:=\G(\cdot,s)\f$ and
$\uu_n=\G_n^{\mathcal D}(t,s)\f$.
Moreover, we set ${\mathcal F}_n:=\sum_{j=1}^m|M\nabla_xu_{n,j}|^2$, ${\mathcal G}_n:=\sum_{i=1}^d\sum_{j=1}^m|M\nabla_xD_iu_{n,j}|^2$ and, throughout the proof, we denote by $c$ a positive constant, which may vary from line to line,
may depend on $s$ and $T=\sup J$, but is independent of $n$.
Let us consider the function $v_n=|\uu_n|^2+a(\cdot-s)\vartheta_n^2{\mathcal G}_n$, where $a$ is positive parameter to be fixed later on, $\vartheta_n(x)=\vartheta(|x|/n)$, for any $x\in\R^d$, and  $\vartheta\in C^{\infty}_c(\R)$ satisfies $\chi_{(-\infty,1/2]}\le\vartheta\le\chi_{(-\infty,1]}$.
From now on, we do not stress the dependence on $n$ of the component of $\uu_n$. Moreover, to ease the notation, we set $\phi_s(t):=t-s$.
The results in \cite[Thms 9.7 \& 9.11]{Fri64Par} and straightforward computations show that
$v_n$ is smooth, vanishes on $(s,T]\times \partial B_n$, $v_n(s,\cdot)=|\f|^2$ and
$D_tv_n-\tilde\A v_n=g_n$ in $(s,T]\times B_n$,
where $g_n=\sum_{i=1}^5g_{i,n}$ with
\begin{align*}
g_{1,n}=&-2\sum_{j=1}^m|\sqrt{Q}\nabla_xu_j|^2 -2a\phi_s\vartheta_n^2\sum_{i,k=1}^d\sum_{j=1}^mQ_{ik}\langle M\nabla_xD_iu_j,M\nabla_x D_k u_j\rangle,\\[2mm]
g_{2,n}=& a\phi_s\vartheta_n^2\sum_{j=1}^m \langle (\mathcal M+{\mathcal M}^T) M\nabla_x u_j,M\nabla_x u_j\rangle
+2a\phi_s\vartheta_n^2\sum_{j,k=1}^m C_{jk}\langle M \nabla_x u_{k},M \nabla_x u_j\rangle\\
&-2a\phi_s\vartheta_n^2\sum_{i,k=1}^d\sum_{j=1}^mQ_{ik}\langle
D_iM\nabla_xu_j,D_kM\nabla_x u_j\rangle+2a\phi_s\vartheta_n^2\sum_{j=1}^m\langle D_tM\nabla_xu_{n,j},\nabla_xu_{n,j}\rangle, \\[2mm]
g_{3,n}=&-2a\phi_s\langle Q \nabla\vartheta_n, \nabla\vartheta_n\rangle{\mathcal F}_n-2a\phi_s\vartheta_n(\tilde\A\vartheta_n){\mathcal F}_n\\
& -8a\phi_s\vartheta_n\sum_{k=1}^d\sum_{j=1}^m(Q\nabla\vartheta_n)_k\big (\langle M \nabla_x D_k u_j,M\nabla_x u_j\rangle+\langle D_kM\nabla_xu_j,M\nabla_x u_j\rangle\big ),\\[2mm]
g_{4,n}=& 2a\phi_s\vartheta_n^2\sum_{i,k=1}^d\sum_{j=1}^mD^2_{ik}u_j
\langle M\nabla_xQ_{ik},M\nabla_xu_j\rangle-4a\phi_s\vartheta_n^2\sum_{i,k=1}^d\sum_{j=1}^mQ_{ik}\langle D_iM \nabla_x D_ku_j,M\nabla_x u_j\rangle \\
& -4a\phi_s\vartheta_n^2\sum_{i,k=1}^d\sum_{j=1}^mQ_{ik}\langle D_iM\nabla_xu_j,M\nabla_x D_k u_j\rangle, \\[1.5mm]
g_{5,n}=&  2\sum_{i=1}^d\langle \uu_n, \tilde{B}_iD_i\uu_n\rangle
+2\langle C\uu_n,\uu_n\rangle+a\vartheta_n^2{\mathcal F}_n+2a\phi_s\vartheta_n^2\sum_{j,k=1}^m u_{k}\langle M \nabla_x C_{jk},M \nabla_x u_j\rangle \\
& +\!2a\phi_s\vartheta_n^2\sum_{i=1}^d\!\sum_{j,k=1}^m\!({\tilde B}_i)_{jk}\langle M\nabla_xD_iu_{k},M\nabla_xu_j\rangle
+ 2a\phi_s\vartheta_n^2\sum_{i=1}^d\sum_{j,k=1}^mD_iu_{k}\langle M \nabla_x(\tilde{B}_i)_{jk},M \nabla_x u_j\rangle.
\end{align*}

Let us estimate the function $g_n$. Recalling that, for any pair of nonnegative definite matrices $M_1$ and $M_2$, it holds that ${\rm Tr}(M_1M_2)\ge \lambda_{M_1}{\rm Tr}(M_2)$
and $|M\xi|^2\le\Lambda_{M^2}|\xi|^2$ in $J\times\Rd$, for any $\xi\in\R^d$, we conclude that
$g_{1,n}\leq -2\sum_{j=1}^m\lambda_Q|\nabla_xu_{j}|^2-2a\phi_s\vartheta_n^2\lambda_Q{\mathcal G}_n
\le -2\lambda_Q\Lambda_{M^2}^{-1}{\mathcal F}_n-2a\phi_s\vartheta_n^2\lambda_Q{\mathcal G}_n$.
The assumptions on ${\mathcal M}$, $M$ and $C$ allow us to estimate
$g_{2,n}\le a\phi_s\vartheta_n^2 (2\Lambda_C+2\lambda_M^{-1}\psi_4+\Lambda_{{\mathcal M}+{\mathcal M}^T}){\mathcal F}_n$.

Let us now consider the function $g_{3,n}$. Its first term is negative; hence, we disregard it. As far as the other terms are concerned, using the estimates
$|Q\nabla\vartheta_n|\leq cn^{-1}\Lambda_Q\chi_{B_{2n}\setminus B_n}|\vartheta'(|\cdot|n^{-1})|$ and ${\rm Tr}(QD^2\vartheta_n)\leq c\Lambda_Qn^{-2}\chi_{B_{2n}\setminus B_n}$, which hold in $I\times\Rd$, and the Young inequality $cxy\le a^{-1/2}c^2x^2+a^{1/2}y^2$ we get
\begin{align*}
|g_{3,n}|
\leq & ac\phi_s\Lambda_Qn^{-2}\vartheta_n\chi_{B_{2n}\setminus B_n}{\mathcal F}_n+2a\phi_s|\vartheta'(|\cdot|n^{-1})|\vartheta_nn^{-1}\langle {\bf b},x\rangle |x|^{-1}\chi_{B_{2n}\setminus B_n}{\mathcal F}_n\\
& +ac\phi_s\vartheta_nn^{-1}\Lambda_Q{\mathcal F}_n^{1/2}{\mathcal G}_n^{1/2}\chi_{B_{2n}\setminus B_n}
+ ac\phi_s|\vartheta'(|\cdot|n^{-1})|\vartheta_nn^{-1}\Lambda_Q\lambda_M^{-1}\psi_5\chi_{B_{2n}\setminus B_n}{\mathcal F}_n\\
\leq & \sqrt{a}c\phi_s{\mathcal F}_n+a^{3/2}c\phi_sn^{-4}\Lambda_Q^2\vartheta_n^2\chi_{B_{2n}\setminus B_n}{\mathcal F}_n
+\sqrt{a}c\phi_s\lambda_Q^{-1}\Lambda_Q^2n^{-2}\chi_{B_{2n}\setminus B_n}{\mathcal F}_n\\
&+ a\phi_s|\vartheta'(|\cdot|n^{-1})|\vartheta_nn^{-1}(2b_0+c\Lambda_Q\psi_5)\chi_{B_{2n}\setminus B_n}{\mathcal F}_n+a^{3/2}\phi_s \lambda_Q\vartheta_n^2{\mathcal G}_n.
\end{align*}

Further, we split $g_{4,n}=g_{4,1,n}+g_{4,2,n}+g_{4,3,n}$. To estimate $g_{4,1,n}$, we observe that
\begin{align*}
\bigg |\sum_{i,k=1}^d\sum_{j=1}^mD^2_{ik}u_{j}
\langle M\nabla_xQ_{ik},M\nabla_xu_{j}\rangle\bigg |
=&\bigg |\sum_{i,h,k=1}^d\sum_{j=1}^m\langle M\nabla_xQ_{ik}M^{-1}_{kh}(M\nabla_xD_iu_{j})_h,M\nabla_xu_{j}\rangle\bigg |\\
\le & \sqrt{d}{\mathcal F}_n^{1/2}{\mathcal G}_n^{1/2}\Lambda_M\lambda_M^{-1}\psi_6\\
\le &\varepsilon^{-1/2}c\Lambda_M^2\lambda_M^{-2}\lambda_Q^{-1}\psi_6^2{\mathcal F}_n+\varepsilon^{1/2}\lambda_Q{\mathcal G}_n.
\end{align*}
Using the Young inequality, we conclude that
$|g_{4,1,n}|\le a\varepsilon^{-1}c\phi_s\Lambda_M^2\lambda_M^{-2}\lambda_Q^{-1}\vartheta_n^2\psi_6^2{\mathcal F}_n+2a\varepsilon\phi_s\vartheta_n^2\lambda_Q{\mathcal G}_n$ for any $\varepsilon>0$.

The other terms in the definition of $g_{4,n}$ can be estimated in a similar way, splitting
$D_iM\nabla_x=(D_iMM^{-1})M\nabla_x$. We obtain
$|g_{4,2,n}+g_{4,3,n}|\le a\varepsilon^{-1}c\phi_s\vartheta_n^2\psi_5^2\lambda_M^{-2}\lambda_Q^{-1}\Lambda_Q^2{\mathcal F}_n+a\varepsilon\phi_s\vartheta_n^2\lambda_Q{\mathcal G}_n$, for any $\varepsilon>0$.
Collecting everything together, and choosing $\varepsilon=1/6$ we get
\begin{align*}
|g_{4,n}|\leq ac\phi_s\lambda_Q^{-1}\vartheta_n^2\lambda_M^{-2}(\Lambda_Q^2\psi_5^2+\Lambda_M^2\psi_6^2){\mathcal F}_n+\frac{1}{2}a\phi_s\vartheta_n^2\lambda_Q{\mathcal G}_n.
\end{align*}

Finally, we consider the function $g_{5,n}$ and observe that
\begin{align*}
g_{5,1,n}=&2\sum_{i=1}^d\sum_{j,h,k=1}^m(M^{-1})_{ih}(\tilde B_i)_{jk}(M\nabla_xu_k)_hu_j\\
g_{5,6,n}=&2a\phi_s\vartheta_n^2\sum_{i=1}^d\sum_{j,h,k=1}^m(M\nabla_xu_k)_h\langle M\nabla_x(\tilde B_i)_{jk}(M^{-1})_{ih},M\nabla_xu_{j,n}\rangle.
\end{align*}
Using Hypothesis \ref{Hyp-stime-gradiente}(iii) and, again, Young inequality, we obtain
\begin{align*}
|g_{5,n}|
\leq & (a^{-1/2}c+ca\phi_s+2\Lambda_C)|\uu_n|^2
+(a^{1/2}\lambda_M^{-2}\psi_1^2+a){\mathcal F}_n\\
& +c\phi_s\vartheta_n^2(a\lambda_M^{-1}\Lambda_M\psi_2+a^{3/2}\Lambda_M^2\psi_3^2
+a\lambda_M^{-2}\psi_1^2){\mathcal F}_n+2a^{3/2}\phi_s\vartheta_n^2\lambda_Q{\mathcal G}_n.
\end{align*}

Now, collecting all the terms together, we get
\begin{eqnarray*}
g_n \leq (ca^{-1/2}+ca\phi_s+2\Lambda_C)|\uu_n|^2+{\mathcal I}_n\mathcal{F}_n +3a\phi_s\vartheta_n^2\lambda_Q\bigg (\sqrt{a}-\frac{1}{2}\bigg ){\mathcal G}_n,
\end{eqnarray*}
where
\begin{align*}
{\mathcal I}_n(t,x)=& a\phi_s|\vartheta'(|x|n^{-1})|\vartheta_n(x)n^{-1}(2b_0(t,x)+c\Lambda_Q(t,x)\psi_5(t,x))\chi_{B_{2n}\setminus B_n}(x)+a+\sqrt{a}c(T-s)\\
&+\sqrt{a}\frac{(\psi_1(t,x))^2}{(\lambda_M(t,x))^2}+(T-s)\sqrt{a}c\frac{(\Lambda_Q(t,x))^2}{\lambda_Q(t,x)(1+|x|^2)}
-2\frac{\lambda_Q(t,x)}{\Lambda_{M^2}(x)}+a\phi_s(t)(\vartheta_n(x))^2{\mathcal J}_n(t,x),
\end{align*}
for any $(t,x)\in [s,T]\times\R^d$ and
\begin{align*}
{\mathcal J}_n(t,x) =&[2\Lambda_C(t,x)+\Lambda_{\mathcal M+\mathcal M^T}(t,x)](1+c{\mathcal H}(t,x))+2\frac{\psi_4(t,x)}{\lambda_M(t,x)} \\
&+c\sqrt{a}\bigg (\frac{(\Lambda_Q(t,x))^2}{1+|x|^4}+(\Lambda_M(t,x))^2(\psi_3(t,x))^2\bigg )
+c\psi_2(t,x)\frac{\Lambda_M(t,x)}{\lambda_M(t,x)}.
\end{align*}
Clearly, the coefficients in front of $|\uu_n|^2$ and ${\mathcal G}_n$ are, respectively, bounded in $[s,T]$, for any choice of $a$,
and negative, if $a<1/4$.
Let us consider the term ${\mathcal I}_n$.
Condition \eqref{iesolo-3} shows that $b_0(t,x)+c\Lambda_Q(t,x)\psi_5(t,x)$ vanishes as $|x|\to +\infty$, uniformly with respect to
$t\in J$. Hence, there exists a positive constant $K$ such that $b_0+c\Lambda_Q\psi_5\le K$ in $J\times\Rd$ and we can estimate the first term in the first line
of ${\mathcal I}_n$ by $a(T-s)Kn^{-1}$. Now, taking \eqref{iesolo-1} into account, it follows that ${\mathcal I}_n<0$
 provided that $a$ is small enough.
Finally, \eqref{iesolo-2} and \eqref{iesolo-3} imply that ${\mathcal J}_n$ is bounded from above in
$(s,T]\times\R^d$ if $\varepsilon\in (0,1/6)$ is fixed small enough. We have so proved that $|g_n|\le cv_n$ in $[s,T]\times\R^d$, and, by the classical maximum principle, we infer that $v_n\leq c\|\f\|_\infty$,
i.e., $\|\G_n^{\mathcal D}(t,s)\f\|_{\infty}+(t-s)^{1/2}\|\vartheta_n M(J_x\G_n^{\mathcal D}(t,s)\f)^T\|_{\infty}\le
c\|\f\|_{\infty}$ for any $t\in (s,T]$. As
the proof of Theorem \ref{thm:existence_solution}
shows, $\G_n^{\mathcal D}(t,s)\f$ converges to $\G(t,s)\f$ in $C^2(B_M;\R^m)$ for any $M>0$. Therefore, letting $n\to
+\infty$ in the previous estimate, we complete the proof.
\end{proof}

\begin{rmk}
\label{rem-4.7}
{\rm In the particular case when $M=I$, we can relax a bit Hypotheses \ref{Hyp-stime-gradiente}. Indeed, by Remark \ref{rem-neumann}
we can approximate the evolution operator $\G(t,s)$ with the evolution operator $\G_n^{\mathcal N}(t,s)$ associated with the realization of
$\bm{\mathcal A}$ in $B_n$ with homogeneous Neumann boundary conditions. In this way, we do not need to introduce the cut-off sequence $(\vartheta_n)$
since the normal derivatives of the function $|J_x\G_n^{\mathcal N}(t,s)\f|^2$ is nonpositive on $\partial B_n$.
As a byproduct, the term $g_{3,n}$ disappears and we do not need to assume anymore
the first condition in \eqref{iesolo-2} and the last condition in \eqref{iesolo-1} and \eqref{iesolo-3}.
Moreover, in this case the matrix ${\mathcal M}$ reduces to the matrix $J_xb$. Therefore, we are assuming a bound on the growth as $|x|\to +\infty$ of
the quadratic form associated with the matrix $J_xb$, i.e., we are assuming a dissipativity condition on the diagonal part of the drift of
$\bm{\mathcal A}$. In the scalar case, this is a typical assumption used to prove gradient estimates both
in the autonomous and nonautonomous setting (see e.g., \cite{AngLorLun12Asy, BerLor07, BerLorbook, KunLorLun09Non, lorenzi-1, lunardi}).
Finally, if the operator $\bm{\mathcal A}$ satisfies Hypotheses \ref{Hyp-stime-gradiente}, then the scalar operator $\tilde {\mathcal A}$ satisfies
the same conditions, and therefore, the scalar evolution operator $G(t,s)$ satisfies \eqref{eq:system_system_gradient_estimate} as well.}
\end{rmk}

\section{Some applications of the gradient estimate \eqref{eq:system_system_gradient_estimate}}

\subsection{The converse of Theorem \ref{compactness}}

As Theorem \ref{compactness} shows, the compactness of the evolution operator $G(t,s)$ in $C_b(\Rd)$ implies the compactness of the evolution operator $\G(t,s)$ in $C_b(\Rd;\R^m)$.
Now, we are interested in finding out sufficient condition for the converse. The main step in this direction, consists in proving formula \eqref{repr_formula_intro}.
Once it is proved, we can adapt to our situation the arguments in the proof of \cite[Thm. 3.6]{DelLor11OnA}.

Since, in general, the evolution operator $G(t,s)$ is well defined only on bounded (and Borel measurable) functions, to make formula \eqref{repr_formula_intro} meaningful, we need to guarantee that the Borel measurable function
$({\mathcal S}\G(\cdot,s)\f)(r,\cdot)$ is bounded in $\Rd$ for any $r\in (s,t)$
and that the integral in the right-hand side of \eqref{repr_formula_intro} is finite.
Note that in the weakly coupled case considered in \cite{DelLor11OnA}, $\tilde B_i\equiv 0$ for any $i=1,\ldots,d$. Hence,
the boundedness of $row_{\bar k}C$ and the uniform estimate \eqref{stima-evol-op} were enough to guarantee the
existence of the integral term in \eqref{repr_formula_intro}. In our situation things are much more difficult since
we have to guarantee that also the function $\sum_{i=1}^d \langle row_{\bar k}\tilde{B}_i(r,\cdot),D_i\G(r,s)\f\rangle$ is bounded in $\Rd$ for
any $r\in (s,t)$.
As we will see in the following proposition, thanks to the estimate
\eqref{eq:system_system_gradient_estimate}, the boundedness of the function $({\mathcal S}\G(\cdot,s)\f)(r,\cdot)$ can be guaranteed under the following two different additional assumptions.

\begin{hyp} \label{bdd_pot}
Hypotheses $\ref{hyp_comp}$  and Hypotheses $\ref{Hyp-stime-gradiente}$ are satisfied. Further, there exists a positive constant $c_*$ such that
$|\tilde B_i|\le c_*\lambda_M$ in $J\times\Rd$ and $row_{\bar k}C\in C_b(J\times \Rd;\Rm)$ for any $i=1,\ldots,d$ and some $\bar k\in\{1,\ldots,m\}$.
\end{hyp}

\begin{hyp}\label{bdd_pot_drift}
Hypotheses $\ref{hyp_comp}$ and Hypotheses $\ref{Hyp-stime-gradiente}$ with $M=I$ are satisfied. Further,
there exist $\bar k\in\{1,\ldots,m\}$ and a bounded interval $J$ such that $row_{\bar k}C$ and $row_{\bar k}\tilde B_i$
belong to $C_b(J\times \Rd;\Rm)$ for any $i=1, \ldots,d$.
\end{hyp}

\begin{thm}
\label{thm-comp-rev}
Assume that Hypotheses $\ref{bdd_pot}$ $($resp. Hypotheses $\ref{bdd_pot_drift})$ are satisfied. If $\G(t,s)$
is compact in $C_b(\Rd;\R^m)$ for any $(t,s)\in \Sigma_{J}$, then $G(t,s)$ is compact in $C_b(\Rd)$ for the same values of $t$ and $s$.
\end{thm}

\begin{proof}
To simplify the notation, we set $\uu:=\G(\cdot,s)\f$ and $T:=\sup J$. Moreover, by $c$ we will denote a positive constant, which
is independent of $j$, $n$, $t$ and $\f$, and may vary from line to line.

As a first step, we show that the integral in the right-hand side of \eqref{repr_formula_intro} is well defined when $\f\in B_b(\Rd;\R^m)$. Since
$\|G(t,s)\|_{\mathcal{L}(C_b(\Rd))}\le 1$ for any $I\ni s\le t$, we just prove that the function
$r\mapsto \|({\mathcal S}\G(\cdot,s)\f)(r,\cdot)\|_{\infty}$  belongs to $L^1((s,t))$ for any $t \in (s,T)$.
The boundedness of $row_{\bar k}C$ and \eqref{stima-evol-op} yield that
$\langle row_{\bar k}C,\uu\rangle \in C_b([s,T]\times\Rd)$ and $\|\langle row_{\bar k}C,\uu\rangle\|_{\infty}\le c_{s,T}\|row_{\bar k}C\|_\infty\|\f\|_\infty$.
Moreover, if Hypotheses \ref{bdd_pot} are satisfied, then from \eqref{eq:system_system_gradient_estimate}, we get
\begin{align*}
\bigg |\sum_{i=1}^d \langle row_{\bar k}\tilde{B}_i(r,\cdot), D_i\uu(r,\cdot)\rangle\bigg |
\le & c\bigg (\sum_{i=1}^d\sum_{j=1}^m|\tilde{B}_i(r,\cdot)_{{\bar k}j}|^2 (D_i u_j(r,\cdot))^2\bigg )^{1/2}
\le  c\lambda_M(r,\cdot)|J_x \uu(r,\cdot)|\notag\\
\le &c\|M(r,\cdot)(J_x \uu(r,\cdot))^T\|_{\infty}
\le  c(r-s)^{-1/2}\|\f\|_{\infty}
\end{align*}
for any $r \in (s,T)$.
On the other hand, if Hypotheses \ref{bdd_pot_drift} hold true, then, from \eqref{eq:system_system_gradient_estimate}, with $M=I$, it follows that
\begin{align*}
\bigg |\sum_{i=1}^d \langle row_{\bar k}\tilde{B}_i(r,\cdot), D_i\uu(r,\cdot)\rangle\bigg |
\le c\max_{1\le i\le d}\|row_{\bar k}B_i\|_{\infty}(r-s)^{-1/2}\|\f\|_{\infty}.
\end{align*}
In both the cases, the function $({\mathcal S}\G(\cdot,s)\f)(r,\cdot)$ is bounded in $\Rd$ for any $r \in (s,T)$ and
\begin{equation}
\|({\mathcal S}\G(\cdot,s)\f)(r,\cdot)\|_\infty\le K(r-s)^{-1/2}\|\f\|_{\infty},\qquad\;\,r \in (s,+\infty)\cap J,
\label{poirot}
\end{equation}
for some constant $K$, depending on $m$, $d$, $s$, $J$, $\|row_{\bar k}C\|_{\infty}$, and also on $\|row_{\bar k}\tilde{B}_i\|_\infty$, $(i=1, \ldots,d)$ when Hypotheses \ref{bdd_pot_drift} hold true.
 This, in particular, yields that the integral in the right-hand side of formula \eqref{repr_formula_intro} is well defined.

We now split the rest of the proof into three steps. The first two steps are devoted to prove formula \eqref{repr_formula_intro} for functions $\f\in C^{2+\alpha}_c(\Rd;\R^m)$.
Once it is proved for such functions, this formula can be easily extended to functions in $B_b(\Rd;\R^m)$ by density.
Indeed, if $\f\in B_b(\Rd;\R^m)$, then we fix a sequence $(\f_n)\subset C^{2+\alpha}_c(\Rd;\R^m)$, bounded with respect to the sup-norm and converging to $\f$ almost
everywhere in $\Rd$. Writing \eqref{repr_formula_intro} with $\f$ being replaced by $\f_n$ and letting $n\to +\infty$, from Corollary \ref{coro-strong}(ii) we conclude that
$\G(t,s)\f_n$ and $G(t,s)f_{n,\bar k}$ converge to $\G(t,s)\f$ and $G(t,s)f_{\bar k}$, respectively, locally uniformly in $\Rd$, as $n\to +\infty$.
Similarly, ${\mathcal S}\G(\cdot,s)\f_n$ converges locally uniformly in $(s,+\infty)\times\Rd$ to
${\mathcal S}\G(\cdot,s)\f$, as $n\to +\infty$. Taking
\eqref{poirot} (with $\f$ being replaced by $\f_n$) and the contractivity of the evolution operator $G(t,s)$ into account, we can apply the
dominated convergence theorem to infer that the integral term in \eqref{repr_formula_intro} converges to the corresponding one with $\f_n$ being
replaced by $\f$, and formula \eqref{repr_formula_intro} follows.

{\em Step 1.} Here, we assume Hypotheses \ref{bdd_pot_drift} and
denote by $\G_n^{\mathcal N}(t,s)$ the evolution operator associated with the operator $\bm{\mathcal A}$ in $C_b(B_n;\R^m)$,
with homogeneous Neumann boundary conditions, which, by Remark \ref{rem-4.7}, satisfies the gradient estimate
$\|\G_n^{\mathcal N}(t,s)\f\|_{C_b(B_n;\R^m)}+\sqrt{t-s}\|J_x\G_n^{\mathcal N}(t,s)\f\|_{C_b(B_n;\R^m)}\le c\|\f\|_{\infty}$
for any $n\in\N$.

As it has been stressed in Remark \ref{rem-neumann}, $G_n^{\mathcal N}(\cdot,s)\f$ tends to $\uu$ locally uniformly in $\Rd$, as $n\to +\infty$. If we split $(\bm{\mathcal A}\tilde\uu_n)_{\bar k}=\tilde {\mathcal A}u_{n,k}+\Phi_{n,\bar k}$,
where $\Phi_{n,\bar k}=\sum_{j=1}^d\langle row_{\bar k} \tilde B_j,D_j\tilde\uu_n\rangle+\langle row_{\bar k} C, \tilde\uu_n\rangle$, then we can write
\begin{equation}
\tilde u_{\bar k}(t,x)= (G_n^{\mathcal N}(t,s)f_{\bar k})(x)+\int_s^t(G_n^{\mathcal N}(t,r)\Phi_{n,\bar k}(r,\cdot))(x)dr,\qquad\;\,t\in (s,T),\;\,x\in\overline{B}_n.
\label{repr-form}
\end{equation}
Clearly, $\tilde u_{\bar k}$ and $G_n^{\mathcal N}(\cdot,s)f_{\bar k}$ converge to $(\G(t,s)\f)_{\bar k}$ and $G(\cdot,s)f_{\bar k}$, respectively, locally uniformly
in $(s,+\infty)\times\Rd$. As far as the integral term in \eqref{repr-form} is concerned, the boundedness of $row_{\bar k}C$ and $row_{\bar k}B_i$ ($i=1,\ldots,d$) and the above gradient estimate,
imply, first, that $\|\Phi_{n,\bar k}(r,\cdot)\|_{C_b(B_n)}\le c(r-s)^{-1/2}\|\f\|_{\infty}$
and, then, that $\|G_n^{\mathcal N}(t,r)\Phi_{n,\bar k}(r,\cdot)\|_{C_b(B_n)}\le c(r-s)^{-1/2}\|\f\|_{\infty}$ for any $r\in (s,T)$ and any $n\in\N$, since
each operator $G_n^{\mathcal N}(t,r)$ is contractive.
Next, we estimate
\begin{align}
&|G_n^{\mathcal N}(t,r)\Phi_{n,\bar k}(r,\cdot)-G^{\mathcal N}(t,r)({\mathcal S}\G(\cdot,s)\f)(r,\cdot)|\notag\\
\le &G_n^{\mathcal N}(t,r)|\Phi_{n,\bar k}(r,\cdot)-({\mathcal S}\G(\cdot,s)\f)(r,\cdot)|
+|G_n^{\mathcal N}(t,r)({\mathcal S}\G(\cdot,s)\f)(r,\cdot)-G(t,r)({\mathcal S}\G(\cdot,s)\f)(r,\cdot)|.
\label{chain}
\end{align}
As $n\to +\infty$, the last term in the right-hand side of \eqref{chain} vanishes, locally uniformly in $\Rd$, due to Remark \ref{rem-neumann}.
To show that also the first term  vanishes, we claim that, for any $R>0$ and $\varepsilon>0$ and $r\in (s,t)$, there exists $M\in\N$ such that
$G_n^{\mathcal N}(t,r)\chi_{\Rd\setminus B_M}\le\varepsilon$ in $B_R$, for any $n$ sufficiently large.
Once the claim is proved, we estimate
\begin{align*}
\|G_n^{\mathcal N}(t,r)g_n(r,\cdot)\|_{C_b(B_R)}
\le &\|G_n^{\mathcal N}(t,r)(\chi_{B_M}g_n(r,\cdot))\|_{C_b(B_R)}
+\|G_n^{\mathcal N}(t,r)(\chi_{\Rd\setminus B_M}g_n(r,\cdot))\|_{C_b(B_R)}\\
\le &\|g_n(r,\cdot)\|_{C_b(B_M)}
+\|G_n^{\mathcal N}(t,r)\chi_{\Rd\setminus B_M}\|_{C_b(B_R)}\sup_{n\in\N}\|g_n(r,\cdot)\|_{C_b(B_n)}\\
\le &\|g_n(r,\cdot)\|_{C_b(B_M)}+\varepsilon c(r-s)^{-1/2}\|f\|_{\infty},
\end{align*}
where $g_n=\Phi_{n,\bar k}-{\mathcal S}\G(\cdot,s)\f$.
Since $g_n$ vanishes, locally uniformly in $(s,+\infty)\times\Rd$ as $n\to +\infty$,
$\limsup_{n\to +\infty}\|G_n^{\mathcal N}(t,r)g_n(r,\cdot)\|_{C_b(B_R)}\le \varepsilon c(r-s)^{-1/2}\|f\|_{\infty}$ and, letting $\varepsilon\to 0^+$, we conclude that
$G_n^{\mathcal N}(t,r)g_n(r,\cdot)$ vanishes uniformly in $B_R$ as $n\to +\infty$, for any $r\in (s,T)$.

To prove the claim, we fix a sequence $(\psi_n)\subset C_b(\Rd)$ satisfying $\chi_{\R^d\setminus B_n}\le \psi_n\le\chi_{\Rd\setminus B_{n-1}}$.
Since $\psi_n$ vanishes locally uniformly in $\Rd$, by \cite[Prop. 3.1]{KunLorLun09Non}, $G(t,r)\psi_n$ tends to $0$ locally uniformly in $\Rd$, for any $r\in (s,t)$. Therefore,
for any fixed $\varepsilon,R>0$, there exists $M=M(r)>0$ such that $G(t,r)\psi_M\le \varepsilon/2$ in $B_R$.
Since $G_n^{\mathcal N}(t,r)\psi_M$ converges to $G(t,r)\psi_M$ in $B_R$, we can determine $n_0=n_0(r)\in\N$ such that
$\|G(t,r)\psi_M-G_n^{\mathcal N}(t,r)\psi_M\|_{C_b(B_R)}\le\varepsilon/2$ for any $n\ge n_0$, which
yields the claim.

{\em Step 2.} Here, we assume Hypotheses \ref{bdd_pot}. In such a case, the argument in Step 1 does not work, since
from the proof of Theorem \ref{prop:system_gradient_estimates} we now just infer
that $\sqrt{t-s}\|\vartheta_n\sqrt{Q}(t,\cdot)J_x\uu_n(t,\cdot)\|_{C_b(B_n;\R^m)}+\|\uu_n(t,\cdot)\|_{C_b(B_n;\R^m)}\le c\|\f\|_{\infty}$ for any $n\in\N$,
where, as usually, $(\vartheta_n)$ is a sequence of cut-off functions, such that ${\rm supp}(\vartheta_n)\subset B_n$ for any $n\in\N$, and
$\uu_n:=\G_n^{\mathcal D}(\cdot,s)\f$. From this inequality we can not deduce the crucial estimate
$\|\Phi_{n,\bar k}(r,\cdot)\|_{C_b(B_n)}\le c(r-s)^{-1/2}\|\f\|_{\infty}$. To overcome this difficulty, we use a slightly different approximation argument.
We denote by ${\bf\Psi}_n$ the function whose components are $\Psi_{n,i}=\vartheta_n\sum_{j=1}^d\langle row_i \tilde
B_j,D_j\uu_n\rangle+\langle row_i C, \uu_n\rangle$,
for any $i=1,\ldots,m$.
Since $\uu_n\in C^{1+\alpha/2,2+\alpha}((s,T)\times B_n;\R^m)$
(see \cite[Thm. IV.5.5]{LadSolUra68Lin}),
the function ${\bf\Psi}_n$ belongs to $C^{\alpha/2,\alpha}((s,T)\times B_n;\R^m)$ and is compactly supported in $[s,T]\times B_n$.
Hence, the same theorem shows that, for any $n\in\N$ such that ${\rm supp}(\f)\subset B_n$,
there exists a unique function $\ww_n\in C^{1+\alpha/2,2+\alpha}((s,T)\times B_n;\R^m)$, which satisfies
$D_t \ww_n=\tilde{\bm \A}\ww_n+{\bf\Psi}_n$ in $(s,T)\times\Rd$, $\ww_n(s,\cdot)=\f$ and it vanishes on $(s,T)\times \partial B_n$.
For any $i=1,\ldots,m$, the component $w_{n,i}$ of $\ww_n$ can be represented through the formula
\begin{equation}\label{formula_natale}
w_{n,i}(t,x)= (G_n^{\mathcal D}(t,s)f_i)(x)+\int_s^t(G_n^{\mathcal D}(t,r)\Psi_{n,i}(r,\cdot))(x)dr,\qquad\;\,t\in (s,T),\;\,x\in\overline{B}_n.
\end{equation}

We claim that $\uu$ is the limit of the sequence $(\ww_n)$. By
Theorem \ref{thm-A2}, $\|\ww_n\|_{C^{1+\alpha/2,2+\alpha}(K;\R^m)}\leq c$ for any compact set
$K\subset(s,T)\times\Rd$ and $n$ large enough. Hence, we can extract a subsequence $(\ww_{n_j})$ which, as $j\to +\infty$, converges in
$C^{1,2}([s+\tau^{-1},T]\times \overline{B}_\tau;\R^m)$, for any $\tau>(T-s)^{-1}$, to some function $\ww \in
C^{1+\alpha/2,2+\alpha}_{\rm loc}((s,T)\times \Rd;\R^m)$. Since $\uu_n$ converges to $\uu$ in $C^{1,2}([s+\tau^{-1},T]\times \overline{B}_{\tau};\R^m)$ (see the proof of Theorem \ref{thm:existence_solution}),
we conclude that $D_t \ww=\tilde{\bm \A}\ww+\bm{\mathcal S}\uu$ in $(s,T)\times \Rd$ where
${\mathcal S}_k\uu= \sum_{i=1}^d\langle row_k \tilde B_i,D_i\uu\rangle+\langle
row_kC, \uu\rangle$ for any $k=1,\ldots,m$.
To conclude that $\ww=\uu$, it suffices to show that $\ww$ can be extended by continuity at $t=s$, where it equals $\f$.
For this purpose, we follow the same strategy
as in the proof of Theorem \ref{thm:existence_solution}, localizing the problem in the ball $B_R$ for any $R>0$. To make the arguments therein contained work, we need to show that
$|{\g}_{n_j}(t,x)|\le c(t-s)^{-1/2}\|\f\|_\infty$ for any $t\in (s,s+1)$,
$x\in B_R$, where ${\g}_{n_j}=-\ww_{n_j}\bm \tilde \A\eta-2J_x\ww_{n_j}(Q\nabla\eta)
+\eta{\bf \Psi}_{n_j}$ ($n_j>M$) and $\eta$ is as in the proof of the quoted theorem. The term ${\bf \Psi}_{n_j}$ can be estimated using the proof of Theorem \ref{prop:system_gradient_estimates},
which shows that $\|\uu_n(t,\cdot)\|_{C_b(B_n;\R^m)}+\sqrt{t-s}\|\vartheta_nJ_x\uu_n(t,\cdot)\|_{C_b(B_n;\R^m)}\le c\|\f\|_{\infty}$ for any $t\in (s,T)$ and $n\in\N$,
and implies that $\|{\bf\Psi}_{n_j}(t,\cdot)\|_{C_b(B_{n_j};\R^m)}\le c(t-s)^{-1/2}\|\f\|_{\infty}$ for any $j\in\N$.
As far as the function $\ww_{n_j}$ is concerned, we observe that the operator $\tilde{\mathcal A}$ satisfies Hypotheses \ref{Hyp-stime-gradiente}.
Therefore, $\sqrt{t-s}\|\vartheta_n\nabla_xG_{n_j}^{\mathcal D}(t,s)g\|_{C_b(B_n;\R^m)}\le c\|g\|_{C_b(B_n;\R^m)}$ for any $t\in (s,T)$ and any $g\in C_c(B_n)$.
Differentiating formula \eqref{formula_natale} with respect to $x$, taking the sup norm in $B_R$ of both sides, and using
the previous two estimates, we conclude that
\begin{align*}
\|J_x\ww_{n_j}(t,\cdot)\|_{C_b(B_{R})}&\le
\sum_{i=1}^d\|\nabla_xG_{n_j}^{\mathcal D}(t,s)f_i\|_{C_b(B_R)}+
\sum_{i=1}^d\int_s^t\|\nabla_xG_{n_j}^{\mathcal D}(t,r)\Psi_{n_j,i}(r,\cdot)\|_{C_b(B_R)}dr\\
&\le
c\bigg(\frac{\|\f\|_{\infty}}{\sqrt{t-s}}+\int_s^t\frac{\|{\bf
\Psi}_{n_j}(r,\cdot)\|_{C_b(B_{n_j};\R^m)}}{\sqrt{t-r}}dr\bigg )\le c \frac{\|\f\|_{\infty}}{\sqrt{t-s}}
\end{align*}
for any $t \in (s, T]$ and $n_j>R$.
The wished estimate on $\g_{n_j}$ follows. Since the above arguments can be applied to any convergent subsequence of $({\bf w}_n)$,
the whole sequence $({\bf w}_n)$ converges to $\uu$.

We now fix $i=\bar k$ in \eqref{formula_natale} and let $n\to +\infty$.
Since $\Psi_{n,i}$ converges to ${\mathcal S}\G(\cdot,s)\f$, locally uniformly in $(s,+\infty)\times\Rd$, as $n\to +\infty$, and
$\|\Psi_{n,i}(t,\cdot)\|_{C_b(B_n)}\le c(t-s)^{-1/2}\|\f\|_{\infty}$ for any $t\in (s,T)$, we can
repeat the same arguments as in Step 1, with $G_n^{\mathcal N}(t,s)$ being replaced by
the operator $G_n^{\mathcal D}(t,s)$, and complete the proof of \eqref{repr_formula_intro}.

{\em Step 3.} Let $(f_j)\subset C_b(\Rd)$ be bounded sequence and set $\f_j=f_j {\bf e}_{\bar k}$ for any $j\in\N$. Without loss of generality, we assume that
$\|f_j\|_{\infty}\le 1$ for any $j\in\N$.
We fix $(t,s)\in\Sigma_J$ and $s_0\in [s,t]$ satisfying $s_0-s\le (8K)^{-2}$, where $K$ is the constant in \eqref{poirot}.
Since $\G(s_0,s)$ is compact in $C_b(\Rd; \R^m)$, there exists a subsequence $(\G(s_0,s)\f_{j_n^0})$ converging uniformly in $\Rd$, as $n \to +\infty$, to some function ${\bf g}_{s_0}\in C_b(\Rd;\R^m)$.
Clearly, $(\G(t,s){\bf f}_{j_n^0})_{\bar k}$ converges uniformly to the $\bar k$-th component of $G(t,s_0){\bf g}_{s_0}$. Moreover, recalling that
$G(t,s)$ is a contractive evolution operator, we can estimate
\begin{align}
\sup_{x\in\Rd}\bigg |\int_s^t (G(t,r)({\mathcal S}\G(\cdot,s)(\f_{j_n^0}-\f_{j_m^0}))(r,\cdot))(x)dr\bigg |\le &
\int_s^{s_0}\|({\mathcal S}\G(\cdot,s)(\f_{j_n^0}-\f_{j_m^0}))(r,\cdot)\|_{\infty}dr\notag\\
&+\int_{s_0}^t\|({\mathcal S}\G(\cdot,s)(\f_{j_n^0}-\f_{j_m^0}))(r,\cdot)\|_{\infty}dr.
\label{star-aru}
\end{align}
Estimate \eqref{poirot} and our choice of $s_0$ implies that the first term in the right-hand side of \eqref{star-aru} does not
exceed $1/2$. On the other hand, ${\mathcal S}\f(r,\cdot)={\mathcal S}\G(\cdot,s_0)\G(s_0,s)(\f_{j_n^0}-\f_{j_m^0})$
and applying \eqref{poirot}, with $\G(\cdot,s)$ being replaced by $\G(\cdot,s_0)$, we estimate the last term in the right-hand side of \eqref{star-aru}
from above by $c\|G(s_0,s)(\f_{j_n^0}-\f_{j_m^0})\|_{\infty}$. Putting everything together, we conclude that
\begin{eqnarray*}
\sup_{x\in\Rd}\bigg |\int_s^t (G(t,r)({\mathcal S}\G(\cdot,s)(\f_{j_n^0}-\f_{j_m^0}))(r,\cdot))(x)dr\bigg |
\le\frac{1}{2}+c\|G(s_0,s)(\f_{j_n^0}-\f_{j_m^0})\|_{\infty},
\end{eqnarray*}
which, combined with \eqref{repr_formula_intro}, \eqref{stima-evol-op}, shows that
$\|G(t,s)(f_{j^0_n}-f_{j^0_m})\|_{\infty}\le \frac{1}{2}+c\|\G(s_0,s)({\bf f}_{j^0_n}-{\bf f}_{j^0_m})\|_{\infty}$.
The last term in the right-hand side of this inequality vanishes as $n,m \to +\infty$. Therefore, there exists $N_0\in \N$ such that $\|G(t,s)(f_{j^0_n}-f_{j^0_m})\|_{\infty}\le 1$ for any $n, m\ge N_0$.

Now, we fix $s_1\in (s,t)$ such that $s_1-s\le (16K)^{-2}$ and repeat the same construction as above with $s_1$ replacing $s_0$.
We thus determine $N_1\in\N$ and a subsequence $(f_{j_n^1})$ of $(f_{j_n^0})$ such that
$\|G(t,s)(f_{j^1_n}-f_{j^1_m})\|_{\infty}\le 1/2$ for any $m,n\ge N_1$.
Iterating this argument, for any $h\in\N$ we can determine a subsequence $(f_{j^h_n})\subset (f_{j^{h-1}_n})$ and an integer $N_h$ such that
\begin{equation}
\|G(t,s)(f_{j^h_n}-f_{j^h_m})\|_{\infty}\le 2^{-h},\qquad\;\,m,n\ge N_h.
\label{guarito}
\end{equation}

Now, we are almost done and, to conclude the proof, we consider the diagonal sequence $(\psi_n)$ with $\psi_n=f_{j^n_n}$ for any $n\in\N$.
We claim that $G(t,s)\psi_n$ converges uniformly in $\Rd$. For this purpose, we fix $\varepsilon>0$ and $h\in\N$ such $2^{-h}\le\varepsilon$. We also set $N=\max\{h,N_h\}$. With
this choice of $N$, and recalling that $\psi_n, \psi_m\in (f_{j^h_p})$ if $n,m\ge h$, from \eqref{guarito} we deduce that $\|G(t,s)(\psi_n-\psi_m)\|_{\infty}\le \varepsilon$ for any $m,n\ge N$,
which, clearly, shows that $(G(t,s)\psi_n)$ is a Cauchy sequence.
\end{proof}

\subsection{Semilinear systems and systems of markovian forward backward stochastic differential equations}

At first, we consider a Cauchy problem for a semilinear system of parabolic equations and we show that, under suitable assumptions, it admits a mild solution $\uu$.
The special form of the nonlinear part allows us to connect the Cauchy problem with a system of Markovian backward stochastic differential equations (BSDE's for short); in particular, we prove that a solution $(\Y,\Z)$ to the system of BDSE's exists, and that it is possible to write this solution in terms of the function $\uu$. Finally, we investigate the existence of a Nash equilibrium for a non-zero sum stochastic differential game.
Hereafter, we assume Hypotheses \ref{hyp_comp}.

\subsubsection{Semilinear systems}
In this subsection we deal with the backward semilinear Cauchy problem
\be
\left\{\begin{array}{lll}D_t\uu(t,x)+(\bm{\mathcal A}\uu)(t,x)=({\boldsymbol \Psi}(\uu))(t,x), &
t\in[0,T), &
x\in\Rd, \\[1mm]
\uu(T,x)=\g(x), & & x\in\Rd,
\end{array}
\right.
\label{eq:davide-2}
\tag{SL-CP}
\ee
where $\bm{\mathcal A}$ is the operator in \eqref{operat-A} and
$({\boldsymbol\Psi}(\uu))(t,x)={\boldsymbol\psi}(x,\sqrt{Q}(t,x)J_x\uu(t,x))$ for any $t\in [0,T)$ and $x\in\Rd$. We assume Hypotheses \ref{Hyp-stime-gradiente} with $M=\sqrt{Q}$
and the following conditions on $\g$ and $\boldsymbol\psi$.

\begin{hyp}{\label{hyp:dickens}}
The function $\bf g$ belongs to $C_b(\Rd;\R^m)$ and there exist positive constants $c$ and $\beta\in(0,1)$ such that
$|\boldsymbol\psi(x_1,z_1)-\boldsymbol \psi(x_2,z_2)|\leq c(1+|z_1|+|z_2|)(|x_1-x_2|^\beta+|z_1-z_2|^\beta)$ and
$|\boldsymbol\psi(x,z)|\leq c(1+|z|)$ for any $x,x_1,x_2\in\Rd$ and $z,z_1,z_2\in\R^{md}$.
\end{hyp}

\begin{rmk}
{\rm
Without loss of generality, we can suppose that $\alpha$ (see Hypotheses \ref{base}) and $\beta$ coincide, and hereafter we simply denote them by $\alpha$.
}
\end{rmk}

We prove the existence of a mild solution $\uu$ of \eqref{eq:davide-2},
i.e., a function $\uu\in \bm{\mathcal K}_T$ which satisfies
the equation
\begin{equation}
\uu(t,x)=(\widehat\G(T-t,0)\g)(x)-\int_t^T(\widehat\G(T-t,T-s)({\boldsymbol\Psi}(\uu))(s,\cdot))(x)ds,\qquad\;\,(t,x)\in [0,T]\times\Rd,
\label{eq:davide-5}
\end{equation}
Here, $\bm{\mathcal K}_T$ is the set of all functions $\uu\in C_b([0,T]\times \R^d;\R^m)\cap
C^{0,1}([0,T)\times\R^d;\R^m)$ such that
$\|\uu\|_{\bm{\mathcal K}_T}:=\|\uu\|_\infty+[\uu]_{\bm{\mathcal K}_T}:=\|\uu\|_\infty+\sup_{t\in[0,T)}\sqrt{T-t}\|\sqrt{Q}(t,\cdot)(J_x\uu(t,\cdot))^T\|_{\infty}<+\infty$.
Moreover, $\widehat\G(t,s)$ is the evolution operator associated with the family $\{\bm{\mathcal A}(T-t): t\in [0,T]\}$.

We first approximate $\g$ and ${\boldsymbol\psi}$ by two sequences $({\bf g}_n)$ and $({\boldsymbol\psi}^{(n)})$ of globally Lipschitz continuous functions,
defined as follows: ${\bf g}_n(x):=(\varrho_n \star \g)(x)$ and
${\boldsymbol\psi}^{(n)}(x,z):=\vartheta(n^{-1}|z|)(\varrho_n\star {\boldsymbol\psi})(x,z)$ for any $n\in\N$,
 $x\in\Rd$ and $z\in\R^{md}$, where $\star$ denotes the convolution
operator, $\varrho_n(x,z)=\vartheta(n^{-1}|x|)\vartheta(n^{-1}|z|)$ for any $(x,z)\in\Rd\times\R^{md}$, and $\vartheta:\R\to\R$ is smooth and satisfy $\chi_{B_1}\le\vartheta\le\chi_{B_2}$.

The Banach fixed-point theorem yields the existence and uniqueness of a solution $\uu_n$ to the equation \eqref{eq:davide-5} with $\boldsymbol\psi$
replaced by ${\boldsymbol\psi}^{(n)}$ for any $n\in\N$. With some more effort, we then prove that there exists a subsequence $(\uu_{k_n})\subset(\uu_n)$ which converges to a mild solution to \eqref{eq:davide-2}.

\begin{rmk}
\label{rem:tolkien}
{\rm
\begin{enumerate}[\rm (i)]
\item
Since we are assuming Hypotheses \ref{Hyp-stime-gradiente} with $M=\sqrt{Q}$, the evolution operator $\widehat\G(t,s)$ satisfies the estimate
$\sqrt{t-s}\|\sqrt{Q}(T-t,\cdot)(J_x\widehat\G(t,s)\f)^T\|_{\infty}\le c\|\f\|_{\infty}$ for any $0\le s<t\le T$, $\f\in C_b(\Rd;\R^m)$ and
some positive constant $c$. From Hypothesis \ref{base}, we also deduce that $\sqrt{t-s}\|(J_x\widehat\G(t,s)\f)^T\|_{\infty}\le c\lambda_0^{-1/2}\|\f\|_{\infty}$ for same $s$ and $\f$.
\item
The choices of ${\bf g}_n$ and $\boldsymbol\psi^{(n)}$ are also connected with the application to systems of forward backward stochastic differential equations (FBSDE's in short) of the next subsection, where
we show that the convergence of the subsequence $(\uu_{k_n})$ implies the convergence of a sequence of solutions $({\Y}_{k_n},{\Z}_{k_n})$, to a family of approximated systems of FBSDE's, to a pair of processes $(\Y,\Z)$.
\end{enumerate}
}
\end{rmk}

\begin{prop}
For any $n\in\N$ there exists a unique mild solution $\uu_n\in \bm{\mathcal K}_T$ to the Cauchy problem \eqref{eq:davide-2},
with $({\boldsymbol\psi},\g)$ being replaced by $({\boldsymbol\psi}^{(n)},\g_n)$. Moreover, there exists a positive constant $K$, independent of $n$, such that $\|\uu_n\|_{\bm{\mathcal K}_T}\le K$
for any $n\in\N$.
\label{prop:davide-1}
\end{prop}

\begin{proof}
The proof is classical, hence we do not enter too much in details. To enlighten the notation, throughout the proof, we denote by $c$ a positive constant, depending at most on $T$, which may vary from line to line.
Since each function ${\boldsymbol\psi}^{(n)}(x,\cdot)$
is Lipschitz continuous in $\R^{md}$, uniformly with respect to $x\in\R^d$, classical arguments, based on
the Banach fixed-point theorem and the generalized Gronwall lemma, allow us to show that equation
\eqref{eq:davide-5} admits a unique solution, which is defined in $[0,T]\times\Rd$.
To prove that $\sup_{n\in\N}\|\uu_n\|_{\bm{\mathcal K}_T}<+\infty$, we observe that
$|{\boldsymbol\psi}^{(n)}(x,z)|\le c(1+|z|)$ and
$\|\g_n\|_{\infty}\le c\|\g\|_{\infty}$ for any $x\in\Rd$, $z\in\R^{md}$ and $n\in\N$.
From equation \eqref{eq:davide-5}, with ${\boldsymbol\psi}^{(n)}$
replacing ${\boldsymbol\psi}$, and Remark \ref{rem:tolkien}(i) we thus deduce that
\begin{align*}
\sqrt{T-t}\|\sqrt{Q}(t,\cdot)(J_x\uu_n(t,\cdot))^T\|_{\infty}
\le c_T(T-t)^{-\frac{1}{2}}\|\g\|_\infty+c+c\int_t^T\frac{\|\sqrt{Q}(t,\cdot)(J_x\uu_n(t,\cdot))^T\|_{\infty}}{\sqrt{s-t}}ds
\end{align*}
for any $t\in [0,T)$.
The generalized Gronwall lemma shows that $\sqrt{T-t}\|\sqrt{Q}(t,\cdot)(J_x\uu_n(t,\cdot))^T\|_{\infty}\leq c$ for any $t\in [0,T)$.
Now, taking the sup-norm (with respect to $x\in\Rd$) of both the sides of the equation
\begin{equation}
\uu_n(t,x)=(\widehat\G(T-t,0)\g_n)(x)\!-\!\int_t^T(\widehat\G(T-t,T-s)({\boldsymbol\Psi}^{(n)}(\uu_n))(s,\cdot))(x)ds,\qquad\,\,(t,x)\in [0,T)\times\Rd
\label{eq:davide-5a}
\end{equation}
and using this last estimate, we deduce that the sup-norm of $\uu_n$ in $[0,T]\times\Rd$ is bounded from above by a positive constant independent of $n\in\N$.
This completes the proof.
\end{proof}

To go further, we need an intermediate result.
For any $n\in\N$, we introduce the space
$\bm{\mathcal X}_{T,n}:=C_b([0,T-n_T^{-1}]\times \overline{B}_{n_T};\R^m)\cap
C^{0,1}([0,T-n_T^{-1}]\times B_{\hat n};\R^m)$,
where $n_T:=[1/T]+n$, and the operator ${\boldsymbol\Gamma}^{(n)}$, defined by
\begin{align*}
({\boldsymbol\Gamma}^{(n)}(\uu))(t,x):=(\widehat\G(T-t,0){\bf g}_n)(x)-
\int_{t+n_T^{-1}}^T(\widehat\G(T-t,T-s)({\boldsymbol\Psi}^{(n)}(\uu))(s,\cdot))(x)ds
\end{align*}
for any $k,n\in\N$, any $(t,x)\in[0,T-n_T^{-1}]\times\Rd$ and
any $\uu\in\bm{\mathcal K}_T$.

\begin{prop}
The operator ${\boldsymbol\Gamma}^{(n)}$ is compact from $\bm{\mathcal K}_T$ in $\bm{\mathcal X}_{T,n}$.
\end{prop}

\begin{proof}
We fix $n\in\N$ and a bounded subset
$\bm{\mathcal W}\subset \bm{\mathcal K}_T$. Clearly, we can limit ourselves to proving that
the integral term $\tilde{\boldsymbol\Gamma}^{(n)}$  in the definition of the operator ${\boldsymbol\Gamma}^{(n)}$
is compact. For this purpose, we claim that the families
$\bm{\mathcal I}_{h,i}=\{D_i^h\tilde{\boldsymbol\Gamma}^{(n)}(\ww): \ww\in \bm{\mathcal W}\}$ ($h=0,1$, $i=1,\ldots,d$) are equibounded and equicontinuous. Throughout the proof, we denote by $c_n$
positive constants, which may vary from line to line and depend on $n$.
The equiboundedness of the families ${\mathcal I}_{h,i}$ follows easily from estimates \eqref{stima-evol-op} and Remark \ref{rem:tolkien}(i), taking into account
that $|{\boldsymbol\psi}^{(n)}(x,z)|\le C(1+|z|)$ for any $x\in\R^d$, $z\in\R^{md}$ and some positive constant independent of $n$.
To prove the equicontinuity of $\bm{\mathcal I}_{h,i}$, we fix ${\bf w}\in \bm{\mathcal W}$, $r,t\in [0,T-n_T^{-1}]$, with $t>r$, $x,y\in B_{n_T}$ and we observe that
\begin{align}
&|(D_i^h\tilde{\boldsymbol\Gamma}^{(n)}(\ww))(t,x)-(D_i^h\tilde{\boldsymbol\Gamma}^{(n)}(\ww))(r,y)|\notag\\
\le &|(D_i^h\tilde\Gamma^{(n)}(\ww))(r,x)-(D_i^h\tilde\Gamma^{(n)}(\ww))(r,y)|
+\bigg|\int_{r+n_T^{-1}}^{t+n_T^{-1}}(D_i^h(\widehat\G(T-r,T-s)({\boldsymbol\Psi}^{(n)}(\ww))(s,\cdot)))(x)ds\bigg|\notag \\
&\bigg|\int_{t+n_T^{-1}}^T(D_i^h(\widehat\G(T-t,T-s)-\widehat\G(T-r,T-s))({\boldsymbol\Psi}^{(n)}(\ww)))(s,\cdot))(x)ds\bigg|
\label{belinda-c}
\end{align}
for $h=0,1$ and $i=1,\ldots,d$. Theorem \ref{thm-A2} shows that, for any convex compact set $K_0\subset\Rd$ and any $\delta>0$, there exists a positive
constant $c_{K_0,\delta}$ such that $\|\G(\cdot,t_1)\f\|_{C^{1+\alpha/2,2+\alpha}([t_2,T]\times K;\R^m)}\leq c_{K_0,\delta}\|{\bf f}\|_\infty$ for any ${\bf f}\in C_b(\Rd;\R^m)$ and any $t_1,t_2\in [0,T)$ such that $t_2-t_1\ge\delta$.
Using this inequality and \eqref{eq:system_system_gradient_estimate}, for any $x\in B_{n_T}$ we can estimate the functions under the integral signs in the right-hand side of the previous inequality, respectively, by
$c_n(T-s)^{-1/2}$ for any $s\in (r+n_T^{-1},t+n_T^{-1})$ and by $c_n(t-r)^{\alpha/2}(T-s)^{-1/2}$ for any $s\in (t+n_T^{-1},T)$. Similarly, to estimate the last term in \eqref{belinda-c}
we use the estimate $|(J_xD_i^h\widehat\G(T-r,T-s){\boldsymbol\Psi}^{(k)}(\ww))(s,z)|\leq c_n(T-s)^{-1/2}$ which holds true for any $s\in (r+n_T^{-1},T)$ and any $z\in K$. We thus conclude that
$|(D_i^h\tilde\Gamma^{(n)}(\ww))(t,x)-(D_i^h\tilde\Gamma^{(n)}(\ww))(r,y)|
\le  c_n(\sqrt{t-r}+|x-y|)$.
We have so proved that the families $\bm{\mathcal I}_{h,i}$ ($h=0,1$, $i=1,\ldots,d$) are
equicontinuous in $[0,T-n_T^{-1}]\times B_{n_T}$.
Arzel\`a-Ascoli Theorem allows us to conclude that $\tilde{\boldsymbol\Gamma}^{(n)}$ is compact from $\bm{\mathcal K}_T$ in $\bm{\mathcal X}_{T,n}$.
\end{proof}

\begin{prop}
\label{prop:davide-2}
Up to a subsequence, $(\uu_n)$
and $(J_x\uu_n)$ converge
locally uniformly in $[0,T]\times\Rd$ and $[0,T)\times\Rd$, respectively.
If we denote by $\uu$ the limit of $(\uu_n)$, then
$\uu\in\bm{\mathcal K}_{T}$ and it is a mild solution of \eqref{eq:davide-2}.
\end{prop}

\begin{proof}
Since each operator ${\boldsymbol\Gamma}^{(n)}$ is compact and the sequence $(\uu_n)$, defined in Proposition \ref{prop:davide-1} is bounded
in $\bm{\mathcal K}_T$ by a positive constant $K$, using a diagonal argument,
we can define a subsequence $(\uu_{k_n})\subset (\uu_n)$ such that, for any $m\in\N$, ${\boldsymbol\Gamma}^{(m)}(\uu_{k_n})$ converges to some function
${\boldsymbol\zeta}^{(m)}$ in $\bm{\mathcal X}_{T,m}$, as $n\to +\infty$.

Being rather long, we split the rest of the proof in some steps.

{\em Step 1.} Here, we prove that the sequences $(\boldsymbol\Gamma^{(k_n)}(\uu_{k_n}))$ and $({J_x}\boldsymbol\Gamma^{(k_n)}(\uu_{k_n}))$ converge.
Fix $h\in\N$, $(t,x)\in[0,T-h_T^{-1})\times B_{h_T}$, and $\ell,n,m\in\N$ such that
$n,m\geq\ell>h$. Then, for $\gamma=0,1$ and $i=1,\ldots,d$ we can estimate
$|D^{\gamma}_i{\boldsymbol\Gamma^{(k_n)}}(\uu_{k_n})-D^{\gamma}_i{\boldsymbol\Gamma^{(k_m)}}(\uu_{k_m})|\leq
\bm{\mathcal A}_{m,n,i}^{(\ell,\gamma)}+\bm{\mathcal B}^{(\ell,m)}_{i,\gamma}+\bm{\mathcal B}^{(\ell,n)}_{i,\gamma}+\bm{\mathcal C}^{(\ell,m)}_{i,\gamma}
+\bm{\mathcal C}^{(\ell,n)}_{i,\gamma}+\bm{\mathcal D}_{i,\gamma}^{(m,n)}$,
where
\begin{align*}
&\bm{\mathcal A}_{m,n,i}^{(\ell,\gamma)}(t,x) = |(D^{\gamma}_i{\boldsymbol\Gamma^{(k_\ell)}}(\uu_{k_n}))(t,x)
-(D^{\gamma}_i{\boldsymbol\Gamma^{(k_{\ell})}}(\uu_{k_m}))(t,x)|, \notag\\
&\bm{\mathcal B}^{(\ell,p)}_{i,\gamma}(t,x) =\bigg|\int_{t+1/\hat k_\ell}^{T}(D^{\gamma}_i\widehat\G(T-t,T-s)[{\boldsymbol\Psi}^{(k_{\ell})}(\uu_{k_p})-{\boldsymbol\Psi}^{(k_p)}(\uu_{k_p})])(s,x)ds\bigg |,\\
&\bm{\mathcal C}^{(\ell,p)}_{i,\gamma}(t,x) = \bigg|\int_{t+1/\hat k_p}^{t+1/\hat k_\ell}(D^{\gamma}_i\widehat\G(T-t,T-s){\boldsymbol\Psi}^{(k_p)}(\uu_{k_p}))(s,x)ds\bigg|, \\
&\bm{\mathcal D}_{i,\gamma}^{(m,n)}(t,x) =|(\widehat\G(T-t,0)(\g_{k_n}-\g_{k_m}))(x)|.
\end{align*}

Fix $\varepsilon>0$. Clearly, from the first part of the proof, we conclude that, for any $\ell\in\N$, there exists $N_{\ell}\in\N$ such that
$\bm{\mathcal A}_{m,n,i}^{(\ell,\gamma)}(t,x)\leq\varepsilon/5$, for any $n,m\geq N_{\ell}$.

As far as the term $\bm{\mathcal B}^{(\ell,p)}_{i,\gamma}$ ($p=m,n$) is concerned, we observe that from Remark \ref{rem:tolkien}(i) we obtain
\begin{align}
\bm{\mathcal B}^{(\ell,p)}_{i,\gamma}
\leq &c_T\int_{T-\delta}^T\frac{\|({\boldsymbol\Psi}^{(k_{\ell})}(\uu_{{k_p}}))(s,\cdot)-({\boldsymbol\Psi}^{(k_p)}(\uu_{k_p}))(s,\cdot)\|_\infty}{(s-t)^{\gamma/2}}ds\notag\\
&+c_T\int_{t}^{T-\delta}
\frac{\|({\boldsymbol\Psi}^{(k_{\ell})}(\uu_{k_p}))(s,\cdot)-({\boldsymbol\Psi}^{(k_p)}(\uu_{k_p}))(s,\cdot)\|_\infty}{(s-t)^{\gamma/2}} ds
\label{stima-Bnkgamma}
\end{align}
for any $\delta>0$.
Using the estimate $|{\boldsymbol\psi}^{(n)}(x,z)|\le c(1+|z|)$, which holds for any $x\in\Rd$, $z\in\R^{md}$, $n\in\N$ and some positive constant $c$, we get $\|{\boldsymbol\Psi}^{(k_{\ell})}(\uu_{k_p}(s,\cdot))-{\boldsymbol\Psi}^{(k_p)}(\uu_{k_p}(s,\cdot))\|_{\infty}\le c(1+(T-s)^{1/2}K)$. Therefore,
\begin{align*}
c_T\int_{T-\delta}^T\frac{\|({\boldsymbol\Psi}^{(k_{\ell})}(\uu_{k_p}))(s,\cdot)-({\boldsymbol\Psi}^{(k_p)}(\uu_{k_p}))(s,\cdot)\|_{\infty}}{(s-t)^{\gamma/2}} ds
\leq &c_T\left (h_T^{-1}-\delta\right )^{-\frac{\gamma}{2}}\sqrt{\delta}(2K+\sqrt{\delta}),
\end{align*}
for any $t\in [0,T-h_T^{-1}]$, provided that $\delta<h_T^{-1}$.

As far as the second term in the right-hand side of \eqref{stima-Bnkgamma} is concerned,
we first observe that
\begin{align}
|{\boldsymbol\psi}^{(n)}(x,z)-{\boldsymbol\psi}(x,z)|
\leq c(1+|z|)n^{-\alpha},\qquad\;\,x\in\Rd,\;\,z\in B_n
\label{arena}
\end{align}
for any $n\in\N$, as it easily follows from the equality
\begin{align*}
{\boldsymbol\psi}^{(n)}(x,z)-{\boldsymbol\psi}(x,z)
= &\int_{\Rd}dy_1\int_{\R^{md}}\varrho_n(y_1,y_2)({\boldsymbol\psi}(x-y_1,z-y_2)-{\boldsymbol\psi}(x,z))dy_2,\qquad\;\,x\in\Rd,\;\,z\in B_n
\end{align*}
and Hypotheses \ref{hyp:dickens}. Splitting
${\boldsymbol\Psi}^{(k_p)}-{\boldsymbol\Psi}^{(k_{\ell})}=({\boldsymbol\Psi}^{(k_p)}-{\boldsymbol\Psi})-({\boldsymbol\Psi}^{(k_{\ell})}-{\boldsymbol\Psi})$ and using \eqref{arena}, we can estimate
$\|({\boldsymbol\Psi}^{(k_{\ell})}(\uu_{k_p}))(s,\cdot)-({\boldsymbol\Psi}^{(k_p)}(\uu_{k_p}))(s,\cdot)\|_{\infty}\le ck_{\ell}^{-\alpha}(1+K\delta^{-1/2})$ for any
$s\in [t,T-\delta]$ and $k_{p}\ge k_{\ell}\ge K\delta^{-\gamma/2}$, and we thus conclude that
\begin{align*}
\int_{t}^{T-\delta}
\frac{\|({\boldsymbol\Psi}^{(k_{\ell})}(\uu_{k_p})(s,\cdot)-({\boldsymbol\Psi}^{(k_p)}(\uu_{k_p}))(s,\cdot)\|_\infty}{(s-t)^{\gamma/2}} ds
\le c_Tk_\ell^{-\alpha}(1+K\delta^{-\frac{1}{2}})(T-\delta-t)^{1-\frac{\gamma}{2}}.
\end{align*}
Now, it is easy to check that we can fix $\delta$ small and $\ell_1$ large such that
$\bm{\mathcal B}^{(\ell,m)}_{i,\gamma}\le\varepsilon/5$ and $\bm{\mathcal B}^{(\ell,n)}_{i,\gamma}\le\varepsilon/5$ for $m,n>\ell\ge \ell_1$.

Now we consider $\bm{\mathcal C}^{(\ell,p)}_{i,\gamma}$, which, thanks again to Remark \ref{rem:tolkien}(i), we estimate as follows:
\begin{align*}
\bm{\mathcal C}^{(\ell,p)}_{i,\gamma}(t,x)
& \leq c_T\lambda_0^{-\gamma/2}\int_{t+\frac{1}{k_p}}^{t+\frac{1}{k_\ell}}(t-s)^{-\frac{\gamma}{2}}(1+K(T-s)^{-1/2})ds\le c_T\lambda_0^{-\frac{\gamma}{2}}(1+ \sqrt{h_T}K)k_{\ell}^{\frac{\gamma}{2}-1}
\end{align*}
for any $p>\ell$.
Hence, there exists $\ell_2\in\N$ such that
$\bm{\mathcal C}^{(\ell,p)}_{i,\gamma}\le \varepsilon/5$ in $[0,T-h_T^{-1})\times B_{h_T}$ for any $p>\ell_2$.

As far as $\bm{\mathcal D}_{i,\gamma}^{(m,n)}$ is concerned, since $\g_{k_n}$ converges to $\g$ locally uniformly in $\Rd$, from Proposition \ref{prop:convergence_properties}$(ii)$
and Theorem \ref{thm-A2}, we conclude that there exists $\ell_3\in\N$ such that $\bm{\mathcal D}_{i,\gamma}^{(m,n)}\leq \varepsilon$ for any $n,m\geq \ell_3$.
Summing up, if $m,n\ge\max\{N_{\ell_1\vee\ell_2}, \ell_1, \ell_2, \ell_3\}$, then
$\|{\boldsymbol\Gamma}^{(k_n)}(\uu_{k_n})-{\boldsymbol\Gamma}^{(k_m)}(\uu_{k_m})\|_{\bm{\mathcal X}_h}\leq \varepsilon$,
and we are done.

{\em Step 2.} Here, we show that also $(\uu_{k_n})$ converges in $\bm{\mathcal X}_{T,h}$, for any
$h\in\N$. For this purpose, we observe that, from \eqref{eq:davide-5a} it follows that
$\uu_{k_n}= {\boldsymbol\Gamma^{(k_n)}}(\uu_{k_n})-\bm{\mathcal E}_n$, where
\begin{align*}
\bm{\mathcal E}_n(t,x):=\int_t^{t+1/{k_n}}(\widehat\G(T-t,T-s)({\boldsymbol\Psi}^{(k_n)}(\uu_{k_n})))(s,x)ds,\qquad\;\,t\in [0,T-h_T^{-1}]\times B_{h_T}.
\end{align*}
Arguing as we did to estimate the term $\bm{\mathcal C}^{(\ell,p)}_{i,\gamma}$, we conclude that ${\mathcal E}_n$ and ${J_x\mathcal E}_n$ vanish uniformly in $[0,T-h_T^{-1}]\times B_{h_T}$, as $n\to +\infty$.
By the arbitrariness of $h$, we conclude that there exists a function $\uu$ such that $\uu_{k_n}$ and $J_x\uu_{k_n}$ converges locally uniformly in $[0,T)\times\Rd$ to $\uu$ and $J_x\uu$, respectively,
as $n\to +\infty$.

{\em Step 3.} Here, we conclude the proof, by showing that $\uu$ is a mild solution of \eqref{eq:davide-2} and
$\uu\in\bm{\mathcal K}_{T}$. Since the sequence $(\uu_n)$ is bounded in $\bm{\mathcal K}_T$, it follows immediately that $\uu$ belongs to $\bm{\mathcal K}_T$
as well. On the other hand, since
\begin{equation}
\uu_{k_n}(t,x)
= (\widehat\G(T-t,0)\g_{k_n})(x)-\int_{t}^T(\widehat\G(T-t,T-s)({\boldsymbol\Psi}^{(k_n)}(\uu_{k_n})))(s,x)ds
\label{crozza-1}
\end{equation}
and $J_x\uu_{k_n}(s,\cdot)$ converges to $J_x\uu(s,\cdot)$
locally uniformly in $[0,T)\times\Rd$, using Proposition \ref{prop:convergence_properties} and the properties of ${\boldsymbol\psi}^{(k_n)}$,
we deduce that $\widehat\G(T-t,T-s)(({\boldsymbol\Psi}^{k_n}(\uu_{k_n}))(s,\cdot))$ converges pointwise in $[0,T)\times\Rd$ to $\widehat\G(T-t,T-s)(({\boldsymbol\Psi}(\uu))(s,\cdot))$.
Moreover, due again to the estimate $|{\boldsymbol\psi}^{(n)}(x,z)|\le c(1+|z|)$, we can let $n$ tend to $+\infty$ in \eqref{crozza-1} and conclude that $\uu$ satisfies \eqref{eq:davide-2} in $[0,T)\times\Rd$.
Finally, we extend $\uu$ by continuity  in $T$ setting $\uu(T,\cdot)=\g$.
This completes the proof.
\end{proof}

\subsubsection{Systems of markovian FBSDE's}

Here, we study a system of forward-backward stochastic
differential equations and show that its solution can be written
in terms of the mild solution to a semilinear Cauchy problem of the type considered in the previous subsection.

Let $(\Omega,{\mathcal E},\mathbb P)$ be a complete probability space and let $(W_t)$ be a $d$-dimensional
Wiener process. By $(\F_t^W)$ we denote its natural filtration augmented with the negligible sets in ${\mathcal E}$.

For any $p\in[1,+\infty)$, we denote by ${\mathbb H}^p$ and ${\mathbb K}$, respectively, the space of progressively measurable
with respect to $\F_t^W$-random processes $(X_t)$
such that $\|X\|_{{\mathbb H}^p}:= \E\sup_{t\in[0,T]}|{X_t}|^p<+\infty$,
and the space of all the $\F_t^W$-progressively measurable processes
$({\bf Y},{\bf Z})$
such that
\begin{eqnarray*}
\|{({\bf Y},{\bf Z})}\|^2_{{\it{cont}}}:=\E\sup_{t\in[0,T]}|{{\bf Y}_t}|^2+\E\int_0^T|{
{\bf Z}_\sigma}|^2d\sigma<+\infty.
\end{eqnarray*}

We consider the system \eqref{eq:FBSDE}, with $H_j(t,x,z):=\sum_{i=1}^d\sum_{k=1}^m(\tilde B_i(t,x)(G(t,x))^{-1})_{jk}z_{i,k}+
\psi_j(x,z)$ for any $j=1,\ldots,m$, under the following assumptions.
\begin{hyp}
\begin{enumerate}[\rm(i)]
\item
$G_{ij}=G_{ji}\in C^{\alpha/2,\alpha}_{\rm loc}([0,T]\times\Rd)\cap C^{1,2+\alpha}_{\rm loc}([0,T]\times\R^d)$ for some $\alpha\in (0,1)$ and any $i,j=1,\ldots,d$; moreover,
the function $\lambda_G$ is bounded from below by a positive constant;
\item
the entries of the vector-valued function ${\bf b}$ and of the matrices-valued functions $\tilde B_i$ $(i=1,\ldots,d)$ belong to $C^{\alpha/2,\alpha}_{\rm loc}([0,T]\times\Rd)\cap C^{0,1+\alpha}_{\rm loc}([0,T]\times\R^d)$; moreover $\g\in C_b(\Rd;\R^m)$;
\item
${\bf b}$ and the matrices $Q=\frac{1}{2}G^2,\tilde B_i$ $(i=1,\ldots,d)$ satisfy the growth conditions in Hypotheses $\ref{Hyp-stime-gradiente}(iii)$, with $M=G$ and $J=[0,T]$;
\item
${\bf b}$ and $G$ grow at most linearly as $|x|\to +\infty$, uniformly with respect to $t\in [0,T]$, and
$\boldsymbol{\psi}:\Rd\times\R^{md}\to\R^m$ satisfies the inequalities in Hypothesis $\ref{hyp:dickens}$ with $\beta=\alpha$.
\end{enumerate}
\label{hyp:devilheart}
\end{hyp}

We denote by $\uu$ the mild solution to \eqref{eq:davide-2} with $C\equiv 0$, provided by Proposition \ref{prop:davide-2}, and state the main result of this subsection.

\begin{thm}
For any $0\leq t\leq\tau\leq T$ and $x\in\Rd$, set
${\bf Y}(\tau,t,x):=\uu(\tau,X(\tau,t,x))$ and
${\bf Z}(\tau,t,x):=G(\tau,X_\tau(\tau,t,x))(J_x\uu(\tau,X(\tau,t,x)))^T$. Then, $(\Y,\Z)\in\mathbb K$
and $(X,\Y,\Z)$ is an adapted solution to \eqref{eq:FBSDE}.
\label{thm:bauman}
\end{thm}

\begin{proof}
Throughout the proof, $(\uu_{k_n})$ is the subsequence of $(\uu_n)$ provided by Proposition \ref{prop:davide-2}, i.e., $\uu_{k_n}$ is the unique mild solution to \eqref{eq:davide-2}, with $C\equiv0$ and $\g$ and $\boldsymbol\Psi$ replaced by $\g_{k_n}$ and $\boldsymbol\Psi^{(k_n)}$, respectively.

Under Hypotheses \ref{hyp:devilheart}(ii), (iv), there exists a unique adapted process $(X_t)$ with continuous trajectories such that
$X_{\tau}=x+\int_t^{\tau}{\bf b}(\sigma,X_{\sigma})d\sigma+\int_t^{\tau}G(\sigma,X_{\sigma})dW_t$ (see \cite[Thm. 3.4.6]{kunita}).
Now we fix $t\in[0,T]$, $x\in\R^d$, set $X_\tau:=X(\tau,t,x)$ and, for any $\tau\in[t,T]$, define
${\bf Y}_{\tau}=\uu(\tau,X_\tau)$, ${\bf Y}^{k_n}_{\tau}=\uu_{k_n}(\tau,X_\tau)$,
${\bf Z}_{\tau}=G(\tau,X_\tau)(J_x\uu(\tau,X_\tau))^T$
and ${\bf Z}^{k_n}_\tau=G(\tau,X_\tau)(J_x\uu_{k_n}(\tau,X_\tau))^T$.
By \cite{fuhrman-tessitore}, $({\bf Y}^{k_n},{\bf Z}^{k_n})$ satisfies the equation (recall that $\g_{k_n}$ and $\bm\psi^{(k_n)}$ are regular approximations)
\begin{equation}
{\bf Y}^{k_n}_\tau+\int_\tau^T {\bf Z}^{k_n}_\sigma dW_\sigma =
\g_{k_n}(X_T)+\int_\tau^T\H_{k_n}(\sigma,X_\sigma,{\bf Z}^{k_n}_\sigma)d\sigma,\qquad\;\,{\mathbb P}\textrm{-}a.s.,\;\,\tau\in [t,T].
\label{eq:davide-11}
\end{equation}
Our aim consists in showing that we can let $n$ tends to $+\infty$ in both the sides of \eqref{eq:davide-11}
obtaining
\begin{equation}
{\bf Y}_\tau+\int_\tau^T {\bf Z}_\sigma dW_\sigma =
\g(X_T)+\int_\tau^T\H(\sigma,X_\sigma,{\bf Z}_\sigma)d\sigma\qquad\;\,{\mathbb P}\textrm{-}a.s.,\;\,\tau\in [t,T].
\label{eq:davide-12}
\end{equation}
Clearly, ${\bf g}_{k_n}$ converges to ${\bf g}$, locally uniformly in $\Rd$. Moreover, since
$\uu_{k_n}$ and $J_x\uu_{k_n}$ converge locally uniformly in $[0,T)\times\Rd$ to $\uu$ and $J_x\uu$, respectively,
${\bf Y}^{k_n}_{\tau}$ and $\Z^{k_n}_\tau$ converge almost surely in $\Omega$ to ${\bf Y}_{\tau}$ and $\Z_{\tau}$, respectively, as $n\to +\infty$.
As far as the integral term in the right-hand side of \eqref{eq:davide-11} is concerned, we begin by observing that
\begin{align}
|{\boldsymbol\psi}_{k_n}(x,z_1)-{\boldsymbol\psi}(x,z_2)|\le c(1+|z_1|+|z_2|)\big(k_n^{-\alpha}+|{z_1-z_2}|^\alpha\big),
\label{eq:davide-10}
\end{align}
for any $x\in\Rd$, $z_1\in B_n$ and $z_2\in\R^{md}$.
Moreover, Hypothesis \ref{hyp:devilheart}(iv) implies that ${\bf H}$ satisfies the condition in Hypothesis \ref{hyp:devilheart}(iv) with $x_1=x_2=x$.
Since $|{\bf Z}^{k_n}_{\sigma}|\vee |{\bf Z}_{\sigma}|\leq K(T-\sigma)^{-1/2}$ (where $K$ is the constant in Proposition \ref{prop:davide-1}), taking $n$ sufficiently large, from
\eqref{eq:davide-10} we conclude that
\begin{align*}
|{\H_{k_n}(\sigma,X_\sigma,{\bf Z}^{k_n}_\sigma)-\H(\sigma,X_\sigma,{\bf Z}_\sigma)}|
\le &c(1+|{\bf Z}^{k_n}_\sigma|+|{\bf Z}_\sigma|)\big(k_n^{-\alpha}+|G(\sigma,X_\sigma)[(J_x(\uu_{k_n}-\uu))(\sigma,X_{\sigma})]^T|^\alpha\big)
\end{align*}
$\mathbb P$-almost surely, for almost every $\sigma\in (\tau,T)$ and that the last side of the previous inequality vanishes
$d\sigma\otimes\mathbb P$-almost surely, as $n\to +\infty$.
Moreover, $|{\boldsymbol\psi}_{k_n}(x,z)|\le c(1+|z|)$ for any $x\in\Rd$, any $z\in\R^{md}$ and some positive constant $c$, independent of $n$, and this shows that
$|{\bf H}_{k_n}(\sigma,X_\sigma,{\bf Z}^{k_n}_\sigma)|\leq c(1+K(T-\sigma)^{-1/2})$.
By dominated convergence, we now easily conclude that the integral term in the right-hand side of \eqref{eq:davide-11} converges $\mathbb P$-almost surely
to $\int_\tau^T\H(\sigma,X_\sigma,{\bf Z}_\sigma)d\sigma$.

It remains to prove the convergence of $\int_\tau^T {\bf Z}^{k_n}_\sigma dW_\sigma$ to
$\int_\tau^T {\bf Z}_\sigma dW_\sigma$. First, we prove that $\int_\tau^T
{\bf Z}_\sigma dW_\sigma$ makes sense, since this is not guaranteed by the previous
estimates, which show only that the growth of ${\bf Z}_\sigma$ can be estimated by
$(T-\sigma)^{-1/2}$, which is not square integrable in $(\tau,T)$.
For this purpose, we show that $({\bf Z}^{k_n}_{\tau})$ is a Cauchy sequence $L^2((0,T)\times\Omega)$.
This is enough to conclude that ${\bf Z}_{\tau}$ is a square integrable process since
${\bf Z}^{k_n}_\tau$ pointwise tends to ${\bf Z}_\tau$.
Let us set $\overline {\bf Y}^{n,m}_\sigma:={\bf Y}^{k_n}_\sigma-{\bf Y}^{k_m}_\sigma$,
$\overline {\bf Z}^{n,m}_\sigma:={\bf Z}^{k_n}_\sigma-{\bf Z}^{k_m}_\sigma$, $\overline\g^{n,m}_T:= \g_{k_n}(X_T)-\g_{k_m}(X_T)$, $\overline\H^{n,m}_\sigma :=
\H_{k_n}(\sigma,X_\sigma,{\bf Z}^{k_n}_\sigma)-\H_{k_m}(\sigma,X_\sigma,{\bf Z}^{k_m}_\sigma)$
for any $n,m\in\N$ and $\sigma\in[0,T]$. Integrating the It\^o formula
$d|{\overline {\bf Y}^{n,m}_\tau}|^2=2\langle\overline
{\bf Y}^{n,k}_\tau,\overline \H^{n,m}_\tau\rangle d\tau
-2\langle\overline {\bf Y}^{n,m}_\tau,\overline {\bf Z}^{n,m}_\tau\rangle dW_\tau
+|{\overline{\bf Z}^{n,m}_\tau}|^2d\tau$,
we obtain
\begin{align}
|{\overline {\bf Y}^{n,m}_\tau}|^2+\int_\tau^T|{\overline {\bf Z}^{n,m}_\sigma}|^2
d\sigma=|{\overline\g^{n,m}_T}|^2
+2\int_\tau^T\langle\overline {\bf Y}^{n,m}_\sigma,\overline\H^{n,m}_\sigma\rangle
d\sigma-2\int_\tau^T \langle\overline{\bf Y}^{n,m}_\sigma, \overline
{\bf Z}^{n,m}_\sigma \rangle dW_\sigma.
\label{ito}
\end{align}
Since  $({\bf Y}^{k_n},{\bf Z}^{k_n}), ({\bf Y}^{k_m},{\bf Z}^{k_m})\in {\mathbb K}$,
the processes $\{\int_0^{\tau}\langle\overline{\bf Y}^{n,m}_\sigma, \overline
{\bf Z}^{n,m}_\sigma \rangle dW_\sigma: \tau\in [0,T]\}$ are martingales.
Hence, from \eqref{ito} it follows that
\begin{eqnarray*}
\E|{\overline {\bf Y}^{n,m}_\tau}|^2+\E\int_\tau^T|{\overline {\bf Z}^{n,m}_\sigma}|^2
d\sigma=\E|{\overline\g^{n,m}_T}|^2
+2\E\int_\tau^T\langle\overline {\bf Y}^{n,m}_\sigma,\overline\H^{n,m}_\sigma\rangle
d\sigma.
\end{eqnarray*}
Note that
\begin{align*}
\bigg |\E\int_\tau^T\langle\overline {\bf Y}^{n,m}_\sigma,\overline\H^{n,m}_\sigma\rangle
d\sigma\bigg |\le &\E\bigg (\sup_{\tau\in[0,T]}|{\overline {\bf Y}^{n,m}_\tau}|
\int_\tau^T|{\overline \H^{n,m}_\sigma}| d\sigma \bigg )\le K\E\int_\tau^T|{\overline \H^{n,m}_\sigma}|d\sigma,
\end{align*}
and, since
$|{\boldsymbol\psi}_{k_n}(x,z_1)-{\boldsymbol\psi}_{k_m}(x,z_2)|
\le c(1+|z_1|+|z_2|)\big(k_n^{-\alpha}+k_m^{-\alpha}+|{z_1-z_2}|^\alpha\big)$
for any $x\in\Rd$, $z_1\in B_n$ and $z_2\in\R^{md}$, we can estimate
$|{\overline \H^{n,m}_\sigma}|\leq
(1+|{\bf Z}^{k_n}_\sigma|+|{\bf Z}^{k_m}_\sigma|)
\big(k_n^{-\alpha}+k_m^{-\alpha}+|{\bf Z}^{k_n}_\sigma-{\bf Z}^{k_m}_\sigma|^\alpha\big)$
for any $k_n,k_m>K(T-s)^{-1/2}$.
From the definitions of ${\bf Z}^{k_n}$, $\overline {\bf Z}^{n,m}$ and the estimate
$|{\Z}^{k_n}(t,x)|,|{\Z}(t,x)|\leq c(T-t)^{-1/2}$, for any $x\in\Rd$, $t\in(0,T)$ and some positive constant $c$, which follows
from Remark \ref{rem:tolkien}(i), by dominated convergence we conclude that, for any
$\varepsilon>0$, there exists $\bar n\in\N$ such that
$\E\int_0^T|{\overline {\bf Z}^{n,m}_\sigma}|d\sigma\leq \varepsilon$, for any $n,m\geq \bar n$.
Similarly, we can show that, for any $\varepsilon>0$, there exists $\bar
n\in\N$ such that $\E|{\overline\g^{n,m}_T}|^2\leq \varepsilon$, for any
$n,k\geq \bar n$.
Hence, $({\bf Z}^{k_n})$ is a Cauchy sequence in $L^2((0,T)\times\Omega)$, and this
implies that $\int_\tau^T
{\bf Z}_\sigma dW_\sigma$ makes sense and that ${\bf Z}^{k_n}$ converges to ${\bf Z}$ in
$L^2((0,T)\times\Omega)$. Thus,
$\E|{\int_\tau^T({\bf Z}^{k_n}_\sigma-{\bf Z}_\sigma)
dW_\sigma}|^2$ tends to $0$ as $n\to +\infty$.
We thus conclude that $\int_\tau^T{\bf Z}^{k_n}_\sigma dW_\sigma$ tends to
$\int_\tau^T{\bf Z}_\sigma dW_\sigma$ in $\Pp$-almost surely. Hence, we can let $n\to +\infty$
in $\eqref{eq:davide-11}$ and deduce that $({\bf Y},{\bf Z})$ is a solution to
\eqref{eq:davide-12}. Clearly, we have also proved that $({\bf Y},{\bf Z})\in{\mathbb K}$.
\end{proof}

\begin{coro}
For any $t\in[0,T]$, the law of the process $\{{\bf Y}_\tau\}_{\tau\in[t,T]}$,
obtained as the limit
of the sequence
$\{{\bf Y}^{k_n}_\tau\}_{\tau\in[t,T]}$, is uniquely determined, i.e.,
if $(\Omega,\F,\{\F_t\},\Pp)$
and $(\tilde \Omega,\tilde\F,\{\tilde\F_t\},\tilde\Pp)$ are two probability
spaces, and $\{{\bf Y}_\tau\}_{\tau\in[t,T]}$, $\{\tilde{\bf Y}_\tau\}_{\tau\in[t,T]}$
are the random processes of Theorem
$\ref{thm:bauman}$, related respectively to $(\Omega,\F,\{\F_t\},\Pp)$ and $(\tilde \Omega,\tilde\F,\{\tilde\F_t\},\tilde\Pp)$,  then $\{{\bf Y}_\tau\}_{\tau\in[t,T]}$,
$\{\tilde{\bf Y}_\tau\}_{\tau\in[t,T]}$
have the same law.
\end{coro}

\begin{proof}
The result is a straightforward consequence of the uniqueness in
law of the solutions $({\bf Y}^{k_n})$ of the approximated problems
\eqref{eq:davide-11}, and
the $\Pp$-a.s. convergence of $({\bf Y}^{k_n})$ to ${\bf Y}$.
\end{proof}

\subsection{Nash Equilibrium for a non-zero stochastic differential game}

In this last subsection, we adapt our results to obtain the existence of
a Nash equilibrium to a non-zero sum stochastic differential game.
Let $(\Omega,{\mathcal E},\mathbb P)$ be a complete probability space and let $(W_t)$ be a $d$-dimensional
Wiener process. By $(\F_t^W)$ we denote its natural filtration augmented with the negligible sets in ${\mathcal E}$. Let $T>0$, $t\in [0,T)$ and $x\in\Rd$.
As in the previous subsection, by $(X_{\tau}=X_{\tau}^{t,x})$ we denote the unique adapted and continuous solution of the stochastic equation
\begin{eqnarray*}
X_{\tau}^{t,x}=x+\int_t^T{\bf b}(\sigma,X_{\sigma}^{t,x})d\sigma+\int_t^TG(\sigma,X_{\sigma}^{t,x})dW_{\sigma},\qquad\;\,{\mathbb P}{\rm -}a.s.,\;\,\tau\in [t,T].
\end{eqnarray*}
We notice that, exactly as in the previous subsection, the assumptions on ${\bf b}$ and $G$, formulated below (see Hypotheses \ref{hyp:davide-nz}), imply existence
and uniqueness of the solution to the above equation.

We suppose that $m$ players intervene on a system and,
for any player $i=1,\ldots,d$, we introduce the space of admissible controls
$U^i:=\{u:[0,T]\times\Omega\to V^i:u {\textrm{ is a predictable
process}}\}$,
where $V^i\subset\R$ are prescribed sets. Moreover, we introduce the
space of admissible strategies $\bm{\mathcal U}:=\otimes_{i=1}^mU^i$. Given $\uu\in \bm{\mathcal U}$, we define
\begin{equation}
W_{\tau}^{(\uu)}:=W_{\tau}-\int_t^{\tau}{\bf r}(\sigma,X_{\sigma}^{t,x},\uu_{\sigma})d\sigma,\qquad\;\,{\mathbb P}^{(\uu)}:=\rho^{(\uu)}\mathbb P,
\label{eq:silmarilion}
\end{equation}
where $\rho^{(\uu)}=\exp\left (\int_t^{\tau}{\bf r}(\sigma,X_{\sigma}^{t,x},\uu_{\sigma})d\sigma-\frac{1}{2}\int_t^{\tau}({\bf r}(\sigma,X_{\sigma}^{t,x},\uu_{\sigma}))^2d\sigma\right )$. Note that
$W^{(\uu)}$ is a $d$-dimensional Wiener process with respect to $\mathbb P^{(\uu)}$ and that $X$ satisfies
\begin{equation}
X^{(t,x)}_{\tau}=x+\int_t^T{\bf b}(\sigma,X_{\sigma}^{t,x})d\sigma+\int_t^TG(\sigma,X_{\sigma}^{t,x}){\bf r}(\sigma,X_{\sigma}^{t,x},\uu_{\sigma})d\sigma+\int_t^TG(\sigma,X_{\sigma}^{t,x})dW_{\sigma}^{(\uu)},
\label{eq:davide-1}
\end{equation}
${\mathbb P}$-almost surely.
For any $\uu\in \bm{\mathcal U}$, to any player $i$, a cost functional $J^i$ is associated, which depends on the
strategies of the whole players. More precisely,
\begin{equation}
J^i({\bf u})={\mathbb E}^{({\bf u})}\bigg [\int_0^Th^i(X_s,{\bf u}_s)ds+g^i(X_T)\bigg ],\qquad\;\,i=1,\ldots,m,
\label{cost-funct}
\end{equation}
where ${\mathbb E}^{(\uu)}$ is the expectation with respect to $\Pp^{(\uu)}$ and
$\h:\Rd\times \bm{\mathcal U}\to\R^m$, the running cost, and $\g:\Rd\to\R^m$, the terminal cost, are bounded
Borel measurable functions.

\begin{defi}
\label{def-6.13}
The strategy $\bm{\mathcal U}\ni\tilde{\bf u}=\left(\tilde u^1,\ldots,\tilde u^m\right)$ is a Nash equilibrium to problem \eqref{eq:davide-1}-\eqref{cost-funct},
if  $J^i(\tilde{\bf u})\leq J^i(\tilde u^1,\ldots,\tilde u^{i-1}, u^i,\tilde u^{i+1},
\ldots,\tilde u^m)$ for any $u^i\in U^i$ $(i=1,\ldots,m)$.
\end{defi}

If $\tilde{\bf u}$ is a Nash equilibrium for
$i=1,\ldots,m$, then the player $i$ has no earn changing its control $\tilde
u^i$, if the other $m-1$ players choose the strategy
$(\tilde u_1,\ldots,\tilde u_{i-1},\tilde u_{i+1},\ldots,\tilde u_m)$.
The Hamiltonian function $\tilde \H$
associated to this system is defined by $\tilde H_i(t,x,z,{\bf u}):=\langle z^i,{\bf r}(t,x,{\bf u})\rangle+h_i(x,{\bf u})$
for any $x\in\Rd$, $z\in\R^{md}$, ${\bf u}\in U$ ($U$ being the set of all the controls) and $i=1,\ldots,m$.
We assume the following assumptions on ${\bf r}$, $\h=(h_1,\ldots,h_m)$ and $\tilde\H$.
\begin{hyp}
\label{hyp:davide-nz}
\begin{enumerate}[\rm (i)]
\item
${\bf h}\in C^{\gamma}_b(\Rd\times \bm {\mathcal U};\R^m)$ for some $\gamma\in (0,1)$ and ${\bf r}$ splits into the sum ${\bf r}={\bf r}_1+{\bf r}_2$, where
${\bf r}_1\in C^{0,1+\gamma}_{\rm loc}([0,T]\times\Rd;\Rd)\cap C_b([0,T]\times\Rd;\Rd)$ and ${\bf r}_2\in C^{\gamma}_b(\Rd\times \bm{\mathcal U};\Rd)$;
\item
The functions $G$, ${\bf b}$, $\tilde B_i:=(G{\bf r}_1)_iI_m$ $(i=1,\ldots,d)$ and $Q=\frac{1}{2}G^2$ satisfy Hypotheses $\ref{hyp:devilheart}$;
\item
$\tilde\H$ satisfies the generalized minimax condition, i.e., for any $i=1,\ldots,m$ there exists a function $\tilde\uu\in C^{\beta}(\Rd\times\R^{md};\bm{\mathcal U})$ $(\beta\in (0,1))$ such that for any $t\in [0,T]$,
$x\in\Rd$, $z\in\R^{md}$ and $u^i\in U^i$ $(i=1,\ldots,m)$ it holds that
\begin{eqnarray*}
\;\;\;\;\;\;\;\;\;\;\;\tilde H^i(t,x,z,\tilde \uu(x,z))\leq \tilde H^i(t,x,z,\tilde u^1(x,z),\ldots,
\tilde u^{i-1}(x,z), u^i,\tilde u^{i+1}(x,z),\ldots,\tilde u^m(x,z)),
\end{eqnarray*}
$\Pp$-almost surely in $\Omega$.
\end{enumerate}
\end{hyp}

\begin{rmk}
\label{rmk:davide}
{\rm
\begin{enumerate}[\rm (i)]
\item
Since $\tilde B_i:=(G{\bf r_1})_iI_m$, the differential operator $\bm\A$ in \eqref{eq:davide-2} is
uncoupled. We stress that this is the classical setting: indeed, comparing our situation with the classical literature (see e.g., \cite{bens-frehse, bensoussan-reg, friedman-diffgames, friedman-stochastic, mannucci},
it is possible to see that, in view of applications to differential games,
the equations of the system of PDE's are coupled only in the semilinear term, while the linear part is the same for any component.
\item
The function ${\boldsymbol\psi}$, whose components are defined by $\boldsymbol\psi_i(x,z):=\langle z^i,{\bf r}_2(x,\tilde {\bf u}(x,z))\rangle+h_i(x,\tilde{\bf u}(x,z))$ for any $x\in\Rd$, $z\in\R^{md}$ and $i=1,\ldots,m$,
satisfies the Hypothesis \ref{hyp:devilheart}(iv) with $\alpha:=\beta\gamma$.
\item
Under Hypotheses \ref{hyp:davide-nz}, Theorem \ref{thm:bauman} holds true.
\end{enumerate}}
\end{rmk}

\begin{thm}
There exists a Nash equilibrium for
problem \eqref{eq:davide-1}-\eqref{cost-funct}.
\end{thm}

\begin{proof}
We begin by proving that
$J^1(\tilde \uu)\leq J^1(u^1,\tilde u^2\ldots,\tilde u^m)$ for any $u^1\in U^1$, where $\tilde\uu$ is as in Hypothesis \ref{hyp:davide-nz}. For this purpose, we fix $u_t^1$ arbitrarily in $U^i$,
set $\widehat{\bf u}^{1}_t:=(u^1_t,\tilde u^2_t,\ldots,\tilde u^d_t)$ and observe that $X_{\tau}$ satisfies
\begin{align*}
X_{\tau}
& = x+\int_t^\tau {\bf b}(\sigma, X_\sigma)d\sigma
+\int_t^\tau G(\sigma, X_{\sigma}){\bf r}(\sigma,X_{\sigma},\widehat
{\bf u}^{1}_\sigma)d\sigma+\int_t^\tau G(\sigma,X_\sigma)dW^{(\widehat{\bf u}^1)}_\sigma
\end{align*}
for any $\tau\in [t,T]$,
where
$\tilde W_\tau^{(\hat\uu)}$ is as in \eqref{eq:silmarilion}, with $\uu$ being replaced by $\widehat{\uu}^1$.
We now introduce the backward system
\begin{eqnarray*}
{\bf Y}_\tau+\int_t^\tau {\bf Z}_\sigma d\tilde W_\sigma =
\g(X_T)+\int_t^\tau\tilde \H(\sigma,X_{\sigma},{\bf Z}_\sigma,\widehat
{\bf u}^{1}_\sigma) d\sigma,
\end{eqnarray*}
where $\tilde\H$ has been introduced after Definition \ref{def-6.13}.
By Theorem \ref{thm:bauman}, this system admits a solution $(X,{\bf Y},{\bf Z})$.
Writing the backward system
with respect to $W^{(\widehat {\bf u}^{1})}$, we get
\begin{equation}
{\bf Y}_\tau+\int_\tau^T {\bf Z}_\sigma dW^{(\widehat {\bf u}^{1})}_\sigma
+\int_\tau^T{\bf Z}_{\sigma}{\bf r}(\sigma,X_{\sigma},\widehat {\bf u}^{1}_\sigma)d\sigma=
\g(X_T)+\int_\tau^T\H(\sigma,X_{\sigma},{\bf Z}_\sigma)d\sigma.
\label{eq:differential_games_backward_system_tilde}
\end{equation}
Note that $\E^{(\widehat{\bf u}^{1})}\left(\int_0^T|{{\bf Z}_t}|^2dt\right)^{1/2}<+\infty$.
Indeed,
\begin{align*}
\bigg |\int_\tau^T {\bf Z}_\sigma dW^{(\widehat{\bf u}^{1})}_\sigma\bigg |
& \leq 2\sup_{\tau\in[0,T]}|{\bf Y}_\tau|+\int_\tau^T\big(|{\bf Z}_\sigma|+|\H(\sigma,X_\sigma,{\bf Z}_\sigma)|\big)d\sigma.
\end{align*}

Taking into account the Burkholder-Davis-Gundy inequalities (see \cite[Thm. 3.28]{kar-shreve}) and using the estimate
$|\H(\sigma,X_{\sigma},{\bf Z}_\sigma)|\le |\H(\sigma,X_{\sigma},0)|+|{\bf Z}_\sigma|^{1+\alpha}$, which follows from Hypothesis \ref{hyp:davide-nz}
and Remark \ref{rmk:davide},
we get
\begin{align*}
\E^{(\widehat{\bf u}^{1})}\bigg (\int_0^T|{{\bf Z}_t}|^2dt\bigg )^{\frac{1}{2}}\!
\le & c \E^{(\widehat{\bf u}^{1})}\!\sup_{\tau\in[0,T]}|{\bf Y}_\tau|+
c\E^{(\widehat {\bf u}^{1})}\!\int_\tau^T\!(|{\bf Z}_\sigma|\!+\!|\H(\sigma,X_{\sigma},0)|)d\sigma\!+\!c\E^{(\widehat {\bf u}^{1})}\int_\tau^T|{\bf Z}_\sigma|^{1+\alpha}d\sigma\\
\le & c(\tilde\E\varrho^2)^{\frac{1}{2}}\bigg (\tilde\E\sup_{\tau\in[0,T]}|{\bf Y}_\tau|
+\tilde\E\int_\tau^T\big(|{\bf Z}_\sigma|^2+|\H(\sigma,X_{\sigma},0)|^2\big)d\sigma\bigg )^{\frac{1}{2}} \\
&+c(\tilde\E\varrho^{\frac{2}{1-\alpha}})^{\frac{1-\alpha}{2}}\bigg(\tilde\E\int_\tau^T|{\bf Z}_\sigma|^2d\sigma\bigg)^{\frac{2}{1+\alpha}},
\end{align*}
and the last side of the previous chain of inequalities is finite.
Hence, taking the conditional expectation in
\eqref{eq:differential_games_backward_system_tilde} with respect to
$\Pp^{(\widehat{\bf u}^{1})}$
and $\tau=t$, we
obtain
\begin{align*}
{\bf Y}_t=\E^{(\widehat {\bf u}^{1})}[\varphi(X_T)|\F_t]+\E^{(\widehat {\bf u}^{1})}\bigg[\int_t^T
[\H(\sigma,X_{\sigma},{\bf Z}_\sigma,\widehat {\bf u}^{1}_\sigma)-{\bf Z}_\sigma
{\bf r}(\sigma,X_{\sigma},\widehat {\bf u}^{1}_\sigma)]d\sigma\bigg |\F_t\bigg].
\end{align*}

Adding and subtracting $\displaystyle\E^{(\widehat{\bf u}^{1})}
\bigg [\int_t^T \h(X_{\sigma},{\widehat
{\bf u}^{1}}_\sigma)d\sigma\bigg |\F_t\bigg ]$
and setting $t=0$, we get
\begin{align}
{\bf Y}_0& = {\bf J}({\widehat{\bf u}^{1}})+\E^{(\widehat{\bf u}^{1})}
\bigg[\int_0^T[\H(\sigma,X_{\sigma},{\bf Z}_\sigma,)-{\bf Z}_{\sigma}
{\bf r}(\sigma,X_{\sigma},{\widehat{\bf u}^{1}}_\sigma)-\h(X_{\sigma},{\widehat{\bf u}^{1}}_\sigma)]d\sigma\bigg ]\notag\\
& = {\bf J}({\widehat {\bf u}^{1}})+\E^{(\widehat {\bf u}^{1})}\bigg [
\int_0^T(\H(\sigma,X_{\sigma},{\bf Z}_\sigma)-\tilde\H(\sigma,X_{\sigma},{\bf Z}_\sigma,\widehat{\bf u}^{1}_\sigma))d\sigma\bigg ],
\label{eq:davide-14}
\end{align}
where ${\bf J}=(J^1,\ldots,J^m)$. Considering the first component of the first and last side of \eqref{eq:davide-14} and
using the minimax condition in Hypothesis \ref{hyp:davide-nz}(iii), we
conclude that $Y_{0,1}\leq J^1({\widehat{\bf u}^{1}})$.
Applying the same argument with $\widehat\uu^1$ being replaced by $\tilde\uu$, we get $Y_{0,1}=J^1(\tilde\uu)$. Moreover, the same procedure yields that
$J^k(\tilde\uu)=Y_{0,k}\le J^k(\widehat{\uu}^k)$ for any $k=1,\ldots,m$.
\end{proof}

\section{Examples}
Here, we provide some classes of operators $\bm\A$ to which the results in the previous sections apply.
\begin{example}
{\rm
\begin{enumerate}[\rm (i)]
\item
Let $Q_{ij}\equiv \delta_{ij}$, $B_i(t,x)=-x_i(1+|x|^2)^rg(t)\hat{B}_i$, $C(t,x)=-|x|^2(1+|x|^2)^ph(t)\hat{C}$
for any $(t,x)\in I\times \Rd$, $i,j=1,\ldots,d$, where $\hat B_i$ ($i=1,\ldots,d$) and $\hat C$ are constant and positive definite matrices, $g,h \in C^{\alpha/2}_{\rm loc}(I)$ have positive infimum, $g$ is bounded in $I$, and $p>2r\ge 0$. We observe that
$\mathcal{K}_{\bm\eta,0}(t,x)
\ge -(1+|x|^2)^{2r}\|g\|_{\infty}^2\sum_{i=1}^d x_i^2|\hat B_i|^2 +4|x|^2(1+|x|^2)^p h_0\lambda_{\hat C}$
for any $t\in I$, $x\in\Rd$, $\eta\in \partial B_1$, where $h_0$ denotes the positive infimum of the function $h$. Since $p>2r$, the function $\mathcal{K}_{\bm\eta,0}(t,\cdot)$ tends to $+\infty$ as $|x|\to +\infty$, uniformly with respect to $t \in I$ and $\eta \in \partial B_1$. Therefore, Hypothesis \ref{hyp:existence_uniqueness}(i) is satisfied with $\varepsilon=1$ and with a suitable choice of the constant $\kappa$. On the other hand, the function $\varphi$, defined by $\varphi(x)=1+|x|^2$, for any $x \in \Rd$, satisfies Hypothesis \ref{hyp:existence_uniqueness}(ii), with $\varepsilon=1$, for any $\lambda>0$.
\item
Let $Q_{ij}(t,x)=q(t)(1+|x|^2)^k\delta_{ij}$, $B_i(t,x)=-b(t)x_i(1+|x|^2)^rI_m+\tilde{b}(t)(1+|x|^2)^p\tilde{B}_i$ and $ C(t,x)=-c(t)(1+|x|^2)^\gamma \hat{C}$ for any $(t,x)\in I\times \Rd$ and $i,j=1,\ldots,d$. Here, $q\in C^{\alpha/2}_{\rm loc}(I)\cap C_b(I)$, the functions $b, \tilde{b}$ and $c$ belong to $ C^{\alpha/2}_{\rm loc}(I)$ and are locally bounded in $I$, $q$ and $b$ have positive infima and $\tilde{B}_i$ ($i=1,\ldots,d$), $\hat{C}$ are constant and positive definite matrices. The exponents $k,r,p,\gamma$ are nonnegative, $r >k-1$ and there exists $\sigma \in (0,1)$ such that $0\le p \le k\sigma$ and $\gamma> k(2\sigma-1)$. Clearly Hypotheses \ref{hyp_comp}(i), (ii) are satisfied. Moreover, since $2|x|^2\le |x|^2+1$ for any $|x|\ge 1$, we can estimate
$(\tilde{\A}(t)\varphi)(x)\le g(\varphi(x))$, for any $t\in I$ and $|x| \ge 1$, where $\varphi$ is as in (i) and $g:[0,+\infty)\to \R$ is defined by $g(y)= b_0y^{r+1}-K$ for any $y\ge 0$ and some positive constant $K$, $b_0$ being the infimum of $b$.
Then, Hypothesis \ref{hyp_comp}(iii) and the assumptions of Theorem \ref{compactness} are satisfied, with $W=\varphi$ and $g$ as above.
We thus conclude that $\G(t,s)$ is compact in $C_b(\Rd;\Rm)$.\end{enumerate}
}
\end{example}
We now provide a class of operators $\bm{\mathcal A}$ whose coefficients satisfy Hypotheses
\ref{hyp_comp} and \ref{Hyp-stime-gradiente}.

\begin{example}
{\rm Let
$Q(t,x)=q(t)(1+|x|^2)^kQ_0$, $b(t,x)=-b(t)(1+|x|^2)^px$, $\tilde B_i(t,x)=\tilde b(t)\tilde B_{0,i}(1+|x|^2)^r$
for any $(t,x)\in I\times\Rd$, where $Q_0$ is a positive definite matrices, $B_{0,i}$ are constant matrices, $q,b,\tilde b\in C^{\alpha/2}_{\rm loc}(I)$
and $b$, $q$ are bounded from below by a positive constant. Finally, the entries $C_{ij}=C_{ji}$ of the matrix-valued function
$C$ belong to $C^{\alpha/2,\alpha}_{\rm loc}(I\times\Rd)\cap C^{0,1+\alpha}_{\rm loc}(I\times\Rd)$, $\sup_{I\times\Rd}\Lambda_C>0$
and $|D_iC(t,x)|\le c(1+|x|^2)^{\tau}$ for any $(t,x)\in I\times\Rd$ and some $c>0$ and $\tau\ge 0$.
We take $M(x):=(1+|x|^2)^s$ for any $x\in\Rd$ and some constant $0<s\le 1/2$ and we set
$a=p$, if $s<1/2$ and $a=p-1$ if $s=1/2$. We further assume that $k\ge 2r$, $2k-2\le a$, $2s+2\tau\le a$, $2r<2s+a$ and $k+s<p+1$.
Straightforward computations show that
\begin{align*}
\lambda_{\mathcal M}(t,x)\leq
& -b(t)(1+|x|^2)^{p-1}(1+|x|^2-2s|x|^2) \\
& -2s(1+|x|^2)^{k-2}q(t)
({\rm Tr}(Q_0)(1+|x|^2)+2(s-1)\langle Q_0(x)x,x\rangle)\leq -c(1+|x|^2)^a,
\end{align*}
for any $(t,x)\in I\times\Rd$.
The functions $\psi_j$ ($j=1,\ldots,5$) satisfy the following conditions:
$\psi_1(t,x)=O(|x|^{2r})$, $\psi_2(t,x)=O(|x|^{2r-1})$, $\psi_3(t,x)=O(|x|^{2\tau})$, $\psi_4\equiv 0$,
$\psi_5(t,x)=O(|x|^{2s-1})$ and $\psi_6(t,x)=O(|x|^{2k-1})$,
as $|x|\to +\infty$, uniformly with respect to $t\in J\subset I$, where $J$ is an arbitrary bounded interval.
Hence, conditions \eqref{iesolo-1}-\eqref{iesolo-3} are easily satisfied. As it is easily seen, also Hypothesis \ref{hyp_comp}(ii), (iii)  are satisfied with $\sigma=1/2$
and $\varphi(x)=1+|x|^2$ for any $x\in\Rd$.

The conditions on the growth of the coefficients of the operator $\bm{\mathcal A}$ can be weakened if we assume that $M\equiv I$.
As it has been stressed in Remark \ref{rem-4.7}, in this case we have to assume only the first condition in \eqref{iesolo-1},
the second one in \eqref{iesolo-2} and all but the last condition in \eqref{iesolo-3}. This provides us with the following
conditions on $a$ (hence $p$), $k$, $r$ and $\tau$: $2r\le k<p+1$, $2\tau\le p$, $p>2r$.
}
\end{example}

\appendix

\section{A priori estimates for solutions to parabolic systems}
In this appendix we prove some a priori estimates for classical solutions to nonautonomous parabolic equations associated
with a system of elliptic operators as in \eqref{operat-A} whose coefficients satisfy Hypotheses \ref{base}.
To enlighten the notation, we set $\|\cdot\|_{h,R}=\|\cdot\|_{C^h_b(B_R;\R^k)}$
for any $h\in\N\cup\{0\}$ and $k\in\N$, $R>0$. Moreover, for any $\alpha,\beta\ge 0$, we denote by $\|\cdot\|_{\alpha,\beta}$ the norm of the space
$C^{\alpha,\beta}([s,T]\times\Omega)$, when $\Omega$ is a domain of $\R^m$ (possibly, $\Omega=\R^m$). Finally, for any $\alpha>0$, we denote by $\|\cdot\|_{\alpha}$
the Euclidean norm in $C^{\alpha}_b(\R^h)$ when $h\in\{1,d\}$.

We recall that, for any $0\le\alpha<\theta$ and any bounded domain $\Omega$ of class $C^{\theta}$, there exists
a positive constant $c$ such that
\begin{equation}
\|\f\|_{C^{\alpha}(\overline\Omega;\R^k)}\le c\|\f\|_{\infty}^{1-\frac{\alpha}{\theta}}\|\f\|^{\frac{\alpha}{\theta}}_{C^{\theta}(\overline\Omega;\R^k)},
\label{caroselli}
\end{equation}
for any $\f\in C^{\theta}(\overline\Omega;\R^k)$ ($k\ge 1$). (The same estimate holds true with $\Omega=\R^k$.) Moreover, if ${\mathcal T}\in {\mathcal L}(C(\overline\Omega;\R^k);C^{\beta}(\overline\Omega;\R^k))\cap {\mathcal L}(C^{\theta}_0(\overline\Omega;\R^k);C^{\beta}(\overline\Omega;\R^k))$ for some $\beta,\theta>0$, then ${\mathcal T}$ is bounded
from $C^{\alpha}_0(\overline\Omega;\R^k)$ to $C^{\beta}(\overline\Omega;\R^k)$, for any $\alpha\in (0,\theta)\setminus\N$, and
\begin{align}
\|{\mathcal T}\|_{{\mathcal L}(C^{\theta}_0(\overline\Omega;\R^k);C^{\beta}(\overline\Omega;\R^k))}\le
\|{\mathcal T}\|_{{\mathcal L}(C(\overline\Omega;\R^k);C^{\beta}(\overline\Omega;\R^k))}^{1-\frac{\alpha}{\theta}}
\|{\mathcal T}\|_{{\mathcal L}(C^{\theta}_0(\overline\Omega;\R^k);C^{\beta}(\overline\Omega;\R^k))}.
\label{caroselli-1}
\end{align}

\begin{prop}
\label{prop-A1}
Let $\Omega\subset\Rd$ be an open set, $T>s\in I$, and $\uu\in C_b([s,T]\times\overline\Omega;\R^m)\cap C^{1,2}((s,T)\times\Omega;\R^m)$
satisfy the equation $D_t\uu=\bm{\mathcal A}\uu+\g$ in $(s,T)\times\Omega$ for some $\g\in C^{\alpha/2,\alpha}((s,T)\times\Omega;\R^m)$. Further, assume that the function $t\mapsto (t-s)\|\uu(t,\cdot)\|_{C^2_b(\Omega;\R^m)}$ is bounded in $(s,T)$. Then, for any $R_1>0$ and any $x_0\in\Omega$, such that $B_{R_1}(x_0)\Subset\Omega$, there exists a positive
constant $c=c(R_1,\lambda_0,s,T)$ such that, for any $t\in (s,T)$,
\begin{align}
&(t-s)\|D^2_x\uu(t,\cdot)\|_{L^{\infty}(B_{R_1}(x_0);\R^m)}+\sqrt{t-s}\,\|J_x\uu(t,\cdot)\|_{L^{\infty}(B_{R_1}(x_0);\R^m)}\notag\\
\leq &c(\|\uu\|_{C_b([s,T]\times\overline\Omega;\R^m)}
+\|\g\|_{C^{\alpha/2,\alpha}((s,T)\times\Omega;\R^m)}).
\label{eq:gradient_interior_estimates}
\end{align}
\end{prop}

\begin{proof}
Throughout the proof, we denote by $c$ a positive constant, which can vary from line to line, may depend on $R_1$ and $T$, but
it is independent of $n$. Up to a translation, we can assume that $x_0=0$. We fix $R_1$ as in the statement and $R_2$ such that $B_{R_2}(x_0)\Subset\Omega$.
Then, we define $r_n:=2R_1-R_2+(R_2-R_1)\sum_{k=0}^n2^{-k}$ for any $n\in\N\cup\{0\}$. Further, we consider
a sequence $(\vartheta_n)\subset C_c^\infty(\R^d)$ of
functions satisfying $\chi_{B_n}\le\vartheta_n\le\chi_{B_{n+1}}$ for any $n\in\N$.
As it is easily seen, $\|\vartheta_n\|_{C^k_b(\Rd)}\le 2^{kn}c$ for any $k=0,1,2,3$.
Let us set $\uu_n:=\vartheta_n\uu$ and observe that each function $\uu_n$ vanishes on $[s,T]\times \partial B_{r_{n+1}}(x_0)$ and
$D_t\uu_n=\bm{\mathcal A}\uu_n+{\bf g}_n$ in $(s,T)\times B_{r_{n+1}}(x_0)$,
where ${\bf g}_n=\vartheta_n\g-{\rm Tr}(QD^2\vartheta_n)\uu-2(J_x\uu)Q\nabla\vartheta_n-\sum_{j=1}^dD_j\vartheta_nB_j\uu$.
In view of the variation-of-constants-formula it thus follows that
\begin{eqnarray*}
\uu(t,x)=({\bf G}_{n+1}^{\mathcal D}(t,s)\vartheta_n\uu(s,\cdot))(x)+\int_s^t({\bf G}_{n+1}^{\mathcal D}(t,r){\bf g}_n(r,\cdot))(x)dr,\qquad\;\,t\in [s,T],\;\,x\in B_{r_n},
\end{eqnarray*}
It is well known that
$(t-s)\|{\bf G}_{n+1}^{\mathcal D}(t,s)\h\|_{2,x_0,r_{n+1}}\le c\|\h\|_{0,r_{n+1}}$ and
$\sqrt{t-s}\|{\bf G}_{n+1}^{\mathcal D}(t,s)\k\|_{2,r_{n+1}}\le c\|\k\|_{1,r_{n+1}}$
for any $t\in (s,T)$, any $n\in\N$
and any $\h\in C(\overline{B}_{r_{n+1}};\R^m)$, $\k\in C^1_0(\overline{B}_{r_{n+1}};\R^m)$, respectively.
Note that the constant $c$ in the previous two estimates is independent of
$n$ since it depends on the ellipticity constant of the operator $\bm{\mathcal A}$ and on the H\"older norms of its coefficients in
$(s,T)\times B_{r_{n+1}}(x_0)$, which can be estimated in terms of the same norms taken in $(s,T)\times B_{R_2}(x_0)$.

Estimate \eqref{caroselli-1}, with $\theta=1$, $\beta=2$, ${\mathcal T}={\bf G}_{n+1}^{\mathcal D}(t,s)$, yields that
$(t-s)^{1-\frac{\alpha}{2}}\|{\bf G}_{n+1}^{\mathcal D}(t,s)\boldsymbol{\varphi}\|_{2,r_{n+1}}\le c\|\boldsymbol{\varphi}\|_{\alpha,r_{n+1}}$
for any $\boldsymbol{\varphi}\in C^{\alpha}_0(B_{r_{n+1}};\R^m)$ and $t\in (s,T)$. Since $\g_n(\sigma,\cdot)\in C^{\alpha}_0(B_{r_{n+1}};\R^m)$ for any $\sigma\in (s,s+1)$,
we can estimate
\begin{align}
(t-s) \|\uu_n(t,\cdot)\|_{2,r_n}
\leq c\|\uu\|_{\infty}+c\int_s^t(t-\sigma)^{-1+\frac{\alpha}{2}}\|{\bf g}_n(\sigma,\cdot)\|_{\alpha,r_{n+1}}d\sigma,\qquad\;\,t\in (s,T).
\label{stima-u}
\end{align}
Note that $\|{\bf g}_n(\sigma,\cdot)\|_{\alpha,r_{n+1}}\le c\|\vartheta_n\|_{2+\alpha,r_{n+1}}(\|\uu(\sigma,\cdot)\|_{1+\alpha,r_{n+1}}+\|\g(\sigma,\cdot)\|_{\alpha,r_{n+1}})$
for any $\sigma\in (s,T)$. Using \eqref{caroselli} and Young inequality, for any $\sigma\in (s,T)$, $\varepsilon>0$ and $n\in\N$ we can estimate
\begin{align*}
\|\uu(\sigma,\cdot)\|_{1,r_{n+1}}
&\le  c(\sigma-s)^{-\frac{1}{2}}\|\uu\|_{\infty}^{\frac{1}{2}}\zeta_{n+1}^{\frac{1}{2}}\le (\sigma-s)^{-\frac{1}{2}}(c\varepsilon^{-1}\|\uu\|_{\infty}+\varepsilon\zeta_{n+1}),\\[1mm]
\|J_x\uu(\sigma,\cdot)\|_{\alpha,r_{n+1}}
&\le c(\sigma-s)^{-\frac{\alpha+1}{2}}\|\uu\|_{\infty}^{\frac{1-\alpha}{2}}\zeta_{n+1}^{\frac{1+\alpha}{2}}\le
(\sigma-s)^{-\frac{\alpha+1}{2}}(c\varepsilon^{-\frac{1+\alpha}{1-\alpha}}\|\uu\|_{\infty}+\varepsilon\zeta_{n+1}).
\end{align*}
where $\zeta_n:=\sup_{\sigma\in (s,T)}(\sigma-s)\|\uu(\sigma,\cdot)\|_{2,r_n}$. Since $\|\vartheta_n\|_{C^{2+\alpha}_b(\R^d)}\le c8^n$ for any $n\in\N$,
from these last three estimates we conclude that
$\|{\bf g}_n(\sigma,\cdot)\|_{\alpha,r_{n+1}}\le 8^n(\sigma-s)^{-\frac{\alpha+1}{2}}
(c\varepsilon^{-\frac{1+\alpha}{1-\alpha}}\|\uu\|_{\infty}\!+\varepsilon\zeta_{n+1})+c8^n\|\g(\sigma,\cdot)\|_{\alpha,x_0,r_{n+1}}$
for any $\sigma\in (s,s+1)$ and $\varepsilon>0$. Replacing this estimate
into \eqref{stima-u} yields
\begin{align}
\zeta_n\leq &c\|\uu\|_{\infty}+8^nc\left (\varepsilon^{-\frac{1+\alpha}{1-\alpha}}\|\uu\|_{\infty}+\varepsilon\zeta_{n+1}+\|\g\|_{\alpha/2,\alpha}\right ),\qquad\;\,n\in\N,\;\,\varepsilon\in (0,1).
\label{stima-zeta-n}
\end{align}
Let us fix $\eta\in (0,64^{-1/(1-\alpha)})$ and $\varepsilon=8^{-n}c^{-1}\eta$.
Multiplying both sides of \eqref{stima-zeta-n} by $\eta^n$ and summing from $1$ to $N\in\N$, we get
$\zeta_0-\eta^{N+1}\zeta_{N+1}\le c\|\uu\|_{\infty}\sum_{k=0}^N64^{\frac{k}{1-\alpha}}\eta^k
+c\|\g\|_{\alpha/2,\alpha}\sum_{k=0}^N(8\eta)^k
\le c(\|\uu\|_{\infty}+\|\g\|_{\alpha/2,\alpha})$,
since both the two series converge, due to the choice of $\eta$.
To conclude, we observe that $\eta^{N+1}\zeta_{N+1}$ tends to $0$ as $N\to +\infty$.
Indeed, by assumptions, $\zeta_{N+1}$ is bounded, uniformly with respect to $N$.
It thus follows that $\eta^{n+1}\zeta_{n+1}$ vanishes as $n\to +\infty$.
We have so proved that $(t-s)\|\uu(t,\cdot)\|_{2,x_0,R_1}\le c(\|\uu\|_{\infty}+\|\g\|_{\alpha/2,\alpha})$ for any $t\in (s,T)$.
Again, estimate \eqref{caroselli} implies that
$\|J_x\uu(t,\cdot)\|_{0,x_0,R_1}\le c\|\uu(t,\cdot)\|_{0,x_0,R_1}^{1/2}\|\uu(t,\cdot)\|_{2,x_0,R_1}^{1/2}$ for any $t\in (s,T)$,
and this allows us to complete the proof of \eqref{eq:gradient_interior_estimates}.
\end{proof}

\begin{thm}
\label{thm-A2}
Fix $T>s\in I$ and let $\uu\in
C^{1+\alpha/2,2+\alpha}_{\rm loc}((s,T]\times\Rd;\R^m)$
satisfy the differential equation
$D_t\uu = \bm{\mathcal A}\uu+\g$ in $(s,T]\times\Rd$, for some $\g\in C^{\alpha/2,\alpha}_{\rm loc}((s,T]\times\Rd;\R^m)$.
Then, for any $\tau\in(0,T-s)$ and any pair of bounded open sets $\Omega_1$ and $\Omega_2$ such that $\Omega_1\Subset\Omega_2$, there exists a positive
constant $c$, depending on $\Omega_1, \Omega_2, \tau, s, T$, but being independent of $\uu$, such that
\begin{align}
\|\uu\|_{C^{1+\alpha/2,2+\alpha}((s+\tau,T)\times\Omega_1;\R^m)}
\leq c(\|\uu\|_{C_b((s+\tau/2,T)\times\Omega_2;\R^m)}+\|\g\|_{C^{\alpha/2,\alpha}((s+\tau/2,T)\times\Omega_2;\R^m)}).
\label{app_eq}
\end{align}
Moreover, for any bounded set $J\subset I$ and any $\delta_0>0$, the constant $c$ depends on $\delta_0$ but is independent of $(s,T)\in J\times J$ such that $T-s\ge\delta_0$.
\end{thm}

\begin{proof}
The main step of the proof consists in showing that, for any $x_0\in\overline{\Omega_1}$ and
any $r>0$, such that $B_{2r}(x_0)\Subset\Omega_2$, estimate \eqref{app_eq} is satisfied with
$\Omega_1=B_r(x_0)$ and $\Omega_2=B_{2r}(x_0)$. Indeed, once this latter estimate is proved, a covering argument will allow us to obtain easily estimate
\eqref{app_eq}.  So, let us prove \eqref{app_eq} with $\Omega_1$ and $\Omega_2$ as above and $x_0=0$ (indeed, up to a translation, we can reduce to this case).
Throughout the proof, we denote by $c$ a positive constant, independent of $\uu$, $\g$ and $n$, which may vary from line to line.
For any $n\in\N$, $t\in\R$ and $x\in\Rd$, set
\begin{eqnarray*}
\varphi_n(t)=\varphi\left(1+\frac{t-s-t_{n+1}}{t_n-t_{n+1}}\right),
\qquad\;\, \vartheta_n(x)=\vartheta\left(1+\frac{|x|-r_n}{r_{n+1}-r_n}\right),
\end{eqnarray*}
where $\varphi,\vartheta\in C^\infty(\R)$ satisfy
$\chi_{[2,\infty)}\leq\varphi\leq\chi_{[1,\infty)}$ and
$\chi_{(-\infty,1]}\leq\vartheta\leq\chi_{(-\infty,2]}$, $r_n=(2-2^{-n})r$
and $t_n=(2^{-1}+2^{-n-1})\tau$, for any $n\in\N$. For any $n\in\N$, the function
$\vv_n:=\uu\varphi_n\vartheta_n$ vanishes at $t=s$ and satisfies $D_t\vv_n=\hat{\bm{\mathcal A}}\vv_n+{\bf g}_n$,
where
${\bf g}_n=\varphi_n\vartheta_n\g-\varphi_n\uu{\rm Tr}(QD^2\vartheta_n)
-\vartheta_n\sum_{j=1}^dD_j\vartheta_nB_j\uu-2\varphi_n(J_x\uu)Q\nabla\vartheta_n+\varphi_n'\vartheta_n\uu$
and $\hat{\bm{\mathcal A}}$ is a nonautonomous elliptic operator with bounded and smooth coefficients, which coincides
with the coefficients of the operator $\bm{\mathcal A}$ in $[s,T]\times B_{2r}$ (recall that $\vv_n$ is compactly supported in $[s,T]\times B_{2r}$).

By well known results, $\|\vv_n\|_{1+\alpha/2,2+\alpha}\leq c\|{\bf g}_n\|_{\alpha/2,\alpha}$
and, for any $J$ and $\delta_0$ as in the statement, the constant $c$ is independent of $(s,T)\in J\times J$ such that $T-s\ge\delta_0$.
Estimating
$\|\g_n\|_{\alpha/2,\alpha}$, we get
\begin{align}
\|\vv_n\|_{1+\alpha/2,2+\alpha}\leq &c\|\vartheta_n\|_3\|\varphi\|_2\big (\|\uu\|_{C^{\alpha/2,1+\alpha}((s+t_{n+1},T)\times
B_{r_{n+1}};\R^m)}+\|\g\|_{C^{\alpha/2,\alpha}((s+t_{n+1},T)\times B_{r_{n+1}};\R^m)}\big )\notag \\
\le & 2^{5n}c\big (\|\vv_{n+1}\|_{\alpha/2,\alpha}+\|J_x\vv_{n+1}\|_{\alpha/2,\alpha}
+\|\g\|_{C^{\alpha/2,\alpha}((s+\tau/2,T)\times B_{2r};\R^m)}\big ).
\label{app_eq-1}
\end{align}
Now, using \eqref{caroselli} we get
$\|\boldsymbol{\zeta}\|_\alpha  \leq c(\varepsilon^{-j-\frac{(j+1)\alpha}{2}}\|\boldsymbol{\zeta}\|_{\infty}+\varepsilon\|\boldsymbol{\zeta}\|_{2+\alpha})$ ($j=1,2$) and
$\|\boldsymbol{\psi}\|_{\alpha/2}  \leq c(\varepsilon^{-\frac{\alpha}{2}}\|\boldsymbol{\psi}\|_{\infty}+\varepsilon\|\boldsymbol{\psi}\|_{1+\alpha/2})$
for any $\boldsymbol{\zeta}\in C_b^{2+\alpha}(\R^d;\R^m)$,
$\boldsymbol{\psi}\in C_b^{1+\alpha/2}([s,T];\R^m)$, $\varepsilon(0,1)$ and some positive constant $c$ independent of $\varepsilon$. Applying these estimates to $\vv_{n+1}$, we deduce that
\begin{equation}
\|\vv_{n+1}\|_{\alpha/2,\alpha}+\|J_x\vv_{n+1}\|_{0,\alpha}
\leq c(\varepsilon\|\vv_{n+1}\|_{1+\alpha/2,2+\alpha}+\varepsilon^{-(1+\alpha)}\|\vv_{n+1}\|_{\infty}),\qquad\;\,\varepsilon\in (0,1)\;\,n\in\N.
\label{app_eq-2}
\end{equation}

To estimate the $\alpha/2$-H\"older norm of the function $J_x\vv_{n+1}(\cdot,x)$ for any $x\in\R^d$, we observe that $\vv_{n+1}\in {\rm Lip}([s,T],C^\alpha_b(\R^n;\R^m))$ and $\|\vv_{n+1}\|_{\rm Lip}\leq c\|\vv_{n+1}\|_{1+\alpha/2,2+\alpha}$.
Indeed, writing
\begin{align*}
\vv_{n+1}(t_2,x)-\vv_{n+1}(t_1,x)=\int_{t_1}^{t_2}D_t\vv_{n+1}(\sigma,x)d\sigma,\qquad\;\,t_1,t_2\in [s,T],\;\,x\in\Rd,
\end{align*}
we easily deduce that
$\|\vv_{n+1}(t_2,\cdot)-\vv_{n+1}(t_1,\cdot)\|_{\alpha}\le c\|D_t\vv_{n+1}\|_{0,\alpha}|t_2-t_1|$ for any $t_1,t_2\in [s,T]$.
Since the function $\vv_{n+1}$ is bounded in $[s,T]$ with values in $C^{2+\alpha}_b(\R^d)$, by interpolation we get
\begin{align*}
\|\vv_{n+1}(t_2,\cdot)-\vv_{n+1}(t_1,\cdot)\|_1\le & c\|\vv_{n+1}(t_1,\cdot)-\vv_{n+1}(t_2,\cdot)\|_{\alpha}^{\frac{1+\alpha}2}
\|\vv_{n+1}(t_2,\cdot)-\vv_{n+1}(t_1,\cdot)\|_{2+\alpha}^{\frac{1-\alpha}2} \\
\le & c\|\vv_{n+1}\|_{1+\alpha/2,2+\alpha}|t_2-t_1|^{\frac{1+\alpha}2}
\end{align*}
for any $t_1,t_2\in [s,T]$.
This shows that $J_x\vv_{n+1}\in C^{(1+\alpha)/2,0}((s,T)\times\Rd;\R^{md})$ and
\begin{equation}
\|J_x\vv_{n+1}\|_{(1+\alpha)/2,0}\leq c\|\vv_{n+1}\|_{1+\alpha/2,2+\alpha}.
\label{app_eq-3}
\end{equation}

Using the interpolative estimate
$\|\boldsymbol{\psi}\|_{\alpha/2}\leq c(\varepsilon\|\boldsymbol{\psi}\|_{(1+\alpha)/2}+\varepsilon^{-\alpha}\|\boldsymbol{\psi}\|_{\infty})$,
which holds for any $\varepsilon\in (0,1)$ and any $\boldsymbol{\psi}\in C_b^{(1+\alpha)/2}([s,T];\R^m)$,
\eqref{app_eq-2} and \eqref{app_eq-3}, we obtain that
\begin{equation}
\|J_x\vv_{n+1}\|_{\alpha/2,\alpha}\leq c(\varepsilon\|\vv_{n+1}\|_{1+\alpha/2,2+\alpha}+\varepsilon^{-(1+\alpha)}\|\vv_{n+1}\|_{\infty}
+\varepsilon^{-\alpha}\|J_x\vv_{n+1}\|_{\infty}).
\label{app_eq-4}
\end{equation}
Now,
$\|J_x\vv_{n+1}\|_{\infty}\le\delta\|\vv_{n+1}\|_{0,2+\alpha}+\delta^{-\frac{1}{1+\alpha}}\|\vv_{n+1}\|_{\infty}$ for any $\delta\in (0,1)$.
Choosing $\delta=\varepsilon^{1+\alpha}$ and replacing this estimate in \eqref{app_eq-4}, we get
$\|J_x\vv_{n+1}\|_{\alpha/2,\alpha}\leq c(\varepsilon\|\vv_{n+1}\|_{1+\alpha/2,2+\alpha}+\varepsilon^{-(1+\alpha)}\|\vv_{n+1}\|_{\infty})$.
This estimate, \eqref{app_eq-1} and \eqref{app_eq-2} yield that
\begin{align*}
\|\vv_n\|_{1+\alpha/2,2+\alpha}\leq 2^{5n} c(\varepsilon\|\vv_{n+1}\|_{1+\alpha/2,2+\alpha}+\varepsilon^{-(1+\alpha)}\|\vv_{n+1}\|_{\infty}
+\|\g\|_{C^{\alpha/2,\alpha}((s+\tau/2,T)\times B_{2r};\R^m)}),
\end{align*}
for any $\varepsilon\in(0,1)$. We fix $\eta\in(0,2^{-5(2+\alpha)})$ and choose $\varepsilon=\varepsilon_n=2^{-5n}c^{-1}\eta$; from the previous estimate we obtain
$\zeta_n\le (\eta\zeta_{n+1}+2^{5n(2+\alpha)}c\|\uu\|_{C_b((s+\tau/2)\times B_{2r})}+2^{5n}c\|\g\|_{C^{\alpha/2,\alpha}((s+\tau/2,T)\times B_{2r};\R^m)})$,
where $\zeta_n=\|\vv_n\|_{1+\alpha/2,2+\alpha}$. Now, we can proceed as in the proof of Proposition \ref{prop-A1}.
\end{proof}

\end{document}